\pdfoutput=1
\documentclass{article}
\usepackage[utf8]{inputenc}
\usepackage{color}
\usepackage{verbatim}
\usepackage{refstyle}
\usepackage{mathrsfs}
\usepackage{mathtools}
\usepackage{amsmath}
\usepackage{amsthm}
\usepackage{amssymb}
\usepackage{graphicx}
\usepackage{setspace}
\PassOptionsToPackage{normalem}{ulem}
\usepackage{ulem}
\usepackage{nomencl}
\providecommand{\makenomenclature}{\makeglossary}
\makenomenclature
\usepackage[unicode=true,
 bookmarks=true,bookmarksnumbered=true,bookmarksopen=true,bookmarksopenlevel=1,
 breaklinks=false,pdfborder={0 0 1},backref=false,colorlinks=true]
 {hyperref}

\makeatletter

\AtBeginDocument{\providecommand\Secref[1]{\ref{Sec:#1}}}
\AtBeginDocument{\providecommand\Corref[1]{\ref{Cor:#1}}}
\AtBeginDocument{\providecommand\Lemref[1]{\ref{Lem:#1}}}
\AtBeginDocument{\providecommand\Thmref[1]{\ref{Thm:#1}}}
\AtBeginDocument{\providecommand\Subsecref[1]{\ref{Subsec:#1}}}
\AtBeginDocument{\providecommand\Eqref[1]{\ref{Eq:#1}}}
\AtBeginDocument{\providecommand\Factref[1]{\ref{Fact:#1}}}
\AtBeginDocument{\providecommand\Defref[1]{\ref{Def:#1}}}
\providecommand{\tabularnewline}{\\}
\RS@ifundefined{subsecref}
  {\newref{subsec}{name = \RSsectxt}}
  {}
\RS@ifundefined{thmref}
  {\def\RSthmtxt{theorem~}\newref{thm}{name = \RSthmtxt}}
  {}
\RS@ifundefined{lemref}
  {\def\RSlemtxt{lemma~}\newref{lem}{name = \RSlemtxt}}
  {}

\theoremstyle{plain}
\newtheorem{thm}{\protect\theoremname}
\theoremstyle{remark}
\newtheorem{rem}[thm]{\protect\remarkname}
\theoremstyle{definition}
\newtheorem{defn}[thm]{\protect\definitionname}
\theoremstyle{plain}
\newtheorem{fact}[thm]{\protect\factname}
\theoremstyle{plain}
\newtheorem{lem}[thm]{\protect\lemmaname}
\theoremstyle{remark}
\newtheorem*{rem*}{\protect\remarkname}
\theoremstyle{plain}
\newtheorem{cor}[thm]{\protect\corollaryname}
\theoremstyle{remark}
\newtheorem{claim}[thm]{\protect\claimname}

\@ifundefined{date}{}{\date{}}
\usepackage{fullpage}
\usepackage{textgreek}
\usepackage{graphicx}
\usepackage{amsfonts}
\usepackage{amsthm}
\allowdisplaybreaks
\usepackage{cases}
\usepackage{tikz-cd}
\usepackage{cancel}
\usepackage{faktor}

\DeclareMathOperator{\vol}{vol}
\let\div\relax
\DeclareMathOperator{\div}{div}
\DeclareMathOperator{\Div}{div}

\DeclareMathOperator{\grad}{grad}
\DeclareMathOperator{\colim}{colim}

\DeclareMathOperator{\supp}{supp}

\DeclareMathOperator{\avg}{avg}

\usepackage{refstyle}
\newref{cor}{name = corollary~, names = corollaries~, Name = Corollary~, Names = Corollaries~}
\newref{conj}{name= conjecture~, names = conjectures~, Name = Conjecture~, Names = Conjectures~}
\newref{exa}{name= example~, names = examples~, Name = Example~, Names = Examples~}
\newref{rem}{name= remark~, names = remarks~, Name = Remark~, Names = Remarks~} 
\newref{thm}{name= theorem~, names = theorems~, Name = Theorem~, Names = Theorems~}
\newref{lem}{name= lemma~, names = lemmas~, Name = Lemma~, Names = Lemmas~}
\newref{def}{name= definition~, names = definitions~, Name = Definition~, Names = Definitions~}
\newref{fact}{name= fact~, names = facts~, Name = Fact~, Names = Facts~}
\newref{blackbox}{name= blackbox~, names = blackboxes~, Name = Blackbox~, Names = Blackboxes~}
\newref{sec}{name= section~, names = sections~, Name = Section~, Names = Sections~}
\newref{subsec}{name= subsection~, names = subsections~, Name = Subsection~, Names = Subsections~}

\usepackage[framemethod=TikZ]{mdframed}
 
\surroundwithmdframed[
   topline=false,
   rightline=false,
   bottomline=false,
   leftmargin=\parindent,
   skipabove=\medskipamount,
   skipbelow=\medskipamount
]{proof}

\def\XXint#1#2#3{{\setbox0=\hbox{$#1{#2#3}{\int}$ }
\vcenter{\hbox{$#2#3$ }}\kern-.6\wd0}}

\usepackage{accents}

\providetoggle{nomsort}
\settoggle{nomsort}{true} % true = sort by use, false = sort as usual

\makeatletter
\iftoggle{nomsort}{%
    \let\old@@@nomenclature=\@@@nomenclature        
        \newcounter{@nomcount} \setcounter{@nomcount}{0}%
        \renewcommand\the@nomcount{\two@digits{\value{@nomcount}}}% Ensure 10>01
        \def\@@@nomenclature[#1]#2#3{% Taken from package documentation
          \addtocounter{@nomcount}{1}%
        \def\@tempa{#2}\def\@tempb{#3}%
          \protected@write\@nomenclaturefile{}%
          {\string\nomenclatureentry{\the@nomcount\nom@verb\@tempa @[{\nom@verb\@tempa}]%
          \begingroup\nom@verb\@tempb\protect\nomeqref{\theequation}%
          |nompageref}{\thepage}}%
          \endgroup
          \@esphack}%
      }{}
\makeatother

\newcommand\iso{\xrightarrow{
   \,\smash{\raisebox{-0.65ex}{\ensuremath{\scriptstyle\sim}}}\,}}

\usepackage{tensind}
\tensordelimiter{?}

\hypersetup{citecolor=purple,linkcolor=blue}

\newcommand\restr[2]{{% we make the whole thing an ordinary symbol
  \left.\kern-\nulldelimiterspace % automatically resize the bar with \right
  #1 % the function
  \vphantom{\big|} % pretend it's a little taller at normal size
  \right|_{#2} % this is the delimiter
  }}

\usepackage{extarrows}

\AtBeginDocument{ 
\DeclareBibliographyAlias{article}{std}
\DeclareBibliographyAlias{book}{std}
\DeclareBibliographyAlias{booklet}{std}
\DeclareBibliographyAlias{collection}{std}
\DeclareBibliographyAlias{inbook}{std}
\DeclareBibliographyAlias{incollection}{std}
\DeclareBibliographyAlias{inproceedings}{std}
\DeclareBibliographyAlias{manual}{std}
\DeclareBibliographyAlias{misc}{std}
\DeclareBibliographyAlias{online}{std}
\DeclareBibliographyAlias{patent}{std}
\DeclareBibliographyAlias{periodical}{std}
\DeclareBibliographyAlias{proceedings}{std}
\DeclareBibliographyAlias{report}{std}
\DeclareBibliographyAlias{thesis}{std}
\DeclareBibliographyAlias{unpublished}{std}
\DeclareBibliographyAlias{*}{std}

\DeclareBibliographyDriver{std}{%
  \usebibmacro{bibindex}%
  \usebibmacro{begentry}%
  \usebibmacro{author/editor+others/translator+others}%
  \setunit{\labelnamepunct}\newblock
  \usebibmacro{title}%
  \newunit\newblock
  \printfield{journaltitle}
  \newunit\newblock
  \printfield{url}
  \newunit\newblock
  \usebibmacro{date}%
  \newunit\newblock
  \usebibmacro{finentry}}
}

\makeatother

\usepackage[style=alphabetic]{biblatex}
\providecommand{\claimname}{Claim}
\providecommand{\corollaryname}{Corollary}
\providecommand{\definitionname}{Definition}
\providecommand{\factname}{Fact}
\providecommand{\lemmaname}{Lemma}
\providecommand{\remarkname}{Remark}
\providecommand{\theoremname}{Theorem}

\begin{document}
\title{Construction of the Hodge-Neumann heat kernel, local Bernstein estimates,
and Onsager's conjecture in fluid dynamics}
\author{Khang Manh Huynh}
\maketitle
\begin{abstract}
Most recently, in \cite{huynhHodgetheoreticAnalysisManifolds2019},
we introduced the theory of heatable currents and proved Onsager's
conjecture on Riemannian manifolds with boundary, where the weak solution
has $B_{3,1}^{\frac{1}{3}}$ spatial regularity. In this sequel, by
applying techniques from geometric microlocal analysis to construct
the Hodge-Neumann heat kernel, we obtain off-diagonal decay and local
Bernstein estimates, and then use them to extend the result to the
Besov space $\widehat{B}_{3,V}^{\frac{1}{3}}$, which generalizes
both the space $\widehat{B}_{3,c(\mathbb{N})}^{1/3}$ from \parencite{isettHeatFlowApproach2014}
and the space $\underline{B}_{3,\text{VMO}}^{1/3}$ from \parencite{bardosOnsagerConjectureBounded2019,nguyenEnergyConservationInhomogeneous2020}
--- the best known function space where Onsager's conjecture holds
on flat backgrounds. 
\end{abstract}

\paragraph*{Acknowledgments}

The author is grateful to Terence Tao, Pierre Albin, Daniel Grieser,
András Vasy and Rafe Mazzeo for valuable discussions during the preparation
of this work.

\tableofcontents{}

\section{Introduction}

Recall the incompressible Euler equation in fluid dynamics:

\begin{equation}
\left\{ \begin{array}{rll}
\partial_{t}\mathcal{V}+\Div\left(\mathcal{V}\otimes\mathcal{V}\right) & =-\grad\mathfrak{p} & \text{ in }M\\
\Div\mathcal{V} & =0 & \text{ in }M\\
\left\langle \mathcal{V},\nu\right\rangle  & =0 & \text{ on }\partial M
\end{array}\right.\label{eq:Euler}
\end{equation}

where $\left\{ %
\begin{tabular}{l}
$(M,g)$ is an oriented, compact smooth Riemannian manifold with boundary,
dimension $\geq2$\tabularnewline
$\nu$ is the outwards unit normal vector field on $\partial M$.\tabularnewline
$I\subset\mathbb{R}$ is an open interval, $\mathcal{V}:I\to\mathfrak{X}M$,
$\mathfrak{p}:I\times M\to\mathbb{R}$.\tabularnewline
\end{tabular}\right.$

We observe that the Neumann condition $\left\langle \mathcal{V},\nu\right\rangle =0$
means $\mathcal{V}\in\mathfrak{X}_{N}$, where $\mathfrak{X}_{N}$
is the set of vector fields which are tangent to the boundary. 

The last two conditions can also be rewritten as $\mathcal{V}=\mathbb{P}\mathcal{V}$,
where $\mathbb{P}$ is the Leray projection operator.

Roughly speaking, Onsager's conjecture says that the energy $\left\Vert \mathcal{V}(t,\cdot)\right\Vert _{L^{2}(M)}$
is a.e. constant in time when $\mathcal{V}$ is a weak solution whose
regularity is at least $\frac{1}{3}$. Making that statement precise
is part of the challenge.

In the boundaryless case, the ``positive direction'' (conservation
when regularity is at least $\frac{1}{3})$ has been known for a long
time \parencite{eyinkEnergyDissipationViscosity1994,constantinOnsagerConjectureEnergy1994,cheskidovEnergyConservationOnsager2008}.
The ``negative direction'' (failure of energy conservation when
regularity is less than $\frac{1}{3}$) is substantially harder \parencite{delellisDissipativeEulerFlows2014,delellisDissipativeContinuousEuler2013},
and was finally settled by Isett in his seminal paper \cite{isettProofOnsagerConjecture2018}
(see the survey in \parencite{delellisTurbulenceGeometryNash2019}
for more details and references).

Since then more attention has been directed towards the case with
boundary on flat backgrounds \parencite{bardosOnsagerConjectureIncompressible2018,drivasOnsagerConjectureAnomalous2018,bardosExtensionOnsagerConjecture2019,bardosOnsagerConjecturePhysical2019,nguyenOnsagerConjectureEnergy2019,bardosOnsagerConjectureBounded2019}.
The case of manifolds without boundary was first handled via a heat-flow
approach in \parencite{isettHeatFlowApproach2014}. This inspired
the consideration of manifolds with boundary in \cite{huynhHodgetheoreticAnalysisManifolds2019},
with the weak solution lying in $L_{t}^{3}B_{3,1}^{\frac{1}{3}}$,
the largest space in which the trace theorem applies. However, the
best results on flat backgrounds hold in the slightly bigger space
$L_{t}^{3}\underline{B}_{3,\text{VMO}}^{\frac{1}{3}}$, so this sequel
aims to make that improvement.

In essence, the absolute Neumann heat flow, created via functional
analysis, is a replacement for the usual convolution on flat spaces,
with special properties like commutativity with divergence. However,
obtaining a pointwise profile of heat kernels for differential forms
(let alone their derivatives) is a difficult problem, so it was hard
to reconcile the heat-flow approach with local-type convolution arguments
on flat backgrounds. Even the definition of $\underline{B}_{3,\text{VMO}}^{\frac{1}{3}}$
itself is local, and it was not immediately obvious that the heat-flow
approach could handle such function spaces.

The solution to this is a manual construction of the Hodge-Neumann
heat kernel (\Secref{Construction-of-heat-kernel}), using techniques
from microlocal analysis and index theory (in particular, Richard
Melrose's calculus on manifolds with corners \parencite{melroseAtiyahPatodiSingerIndexTheorem2018,melroseCalculusConormalDistributions1992}).
The theory mimics the development of pseudodifferential operators,
in creating a filtered algebra that quantifies how ``nonsingular''
an operator is as we approach the edges. In particular, much like
the pseudolocality of {\textPsi}DOs, the construction yields a precise
description near the diagonal, as well as rapid decay away from the
diagonal. This enables the use of the heat flow as local convolution,
and we obtain local Bernstein estimates which allow us to handle VMO-type
function spaces.

The addition of local Besov-type estimates also marks another stage
of development for the theory of \textbf{intrinsic harmonic analysis}
for differential forms (including scalar functions and vector fields)
on compact Riemannian manifolds with boundary, originally set forth
in the prequel with Hodge theory as the foundation. In particular,
we have extended the notion of tempered distributions, and the methods
of Littlewood-Paley frequency decomposition (e.g. Bernstein-type estimates),
which have proved useful on flat backgrounds for problems in fluid
dynamics and dispersive PDEs (cf. \parencite{taoQuantitativeFormulationGlobal2009,taoLocalisationCompactnessProperties2013,taoNonlinearDispersiveEquations2006,lemarie-rieussetRecentDevelopmentsNavierStokes2002}),
to manifolds with boundary. More history and references can be found
in \parencite{huynhHodgetheoreticAnalysisManifolds2019}.

\subsection{Main result}

To state the main result, we need some terminology. 

The standard Besov spaces $B_{p,q}^{s}$, and the absolute Neumann
heat semigroup $e^{s\widetilde{\Delta_{N}}}$ were discussed in \cite{huynhHodgetheoreticAnalysisManifolds2019}.

For $r>0$, we define $M_{>r}:=\left\{ x\in M:\mathrm{dist}(x,\partial M)>r\right\} $.
Let $\accentset{\circ}{M}$ denote the interior of $M$.

For $p\in\left(1,\infty\right)$, we say $X\in\widehat{B}_{p,V}^{1/p}\mathfrak{X}\left(M\right)$
if $X\in L^{p}\mathfrak{X}\left(M\right)$ and $\forall r>0:$
\[
\left(\frac{1}{\sqrt{s}}\right)^{\frac{1}{p}}\left\Vert X-e^{s\widetilde{\Delta_{N}}}X\right\Vert _{L^{p}\left(M_{>r}\right)}\xrightarrow{s\to0}0
\]
Or equivalently (by \Corref{equiv_cN}), $\left(\sqrt{s}\right)^{1-\frac{1}{p}}\left\Vert e^{s\widetilde{\Delta_{N}}}X\right\Vert _{W^{1,p}\left(M_{>r}\right)}\xrightarrow{s\to0}0$

Similarly, for $p\in\left(1,\infty\right)$, we say $\mathcal{X}\in L_{t}^{p}\widehat{B}_{p,V}^{1/p}\mathfrak{X}\left(M\right)$
if $\mathcal{X}\in L_{t}^{p}L^{p}\mathfrak{X}\left(M\right)$ and
$\forall r>0:$ 
\[
\left(\frac{1}{\sqrt{s}}\right)^{\frac{1}{p}}\left\Vert \mathcal{X}-e^{s\widetilde{\Delta_{N}}}\mathcal{X}\right\Vert _{L_{t}^{p}L^{p}\left(M_{>r}\right)}\xrightarrow{s\to0}0
\]

As shown in \Lemref{equal_isett}, $\widehat{B}_{3,V}^{\frac{1}{3}}$
contains the space $\widehat{B}_{3,c(\mathbb{N})}^{1/3}$ from \parencite{isettHeatFlowApproach2014}
(with equality when $\partial M=\emptyset$). While on flat backgrounds,
by \Thmref{contain_VMO}, $\widehat{B}_{3,V}^{1/3}$ coincides with
$\underline{B}_{3,\text{VMO}}^{\frac{1}{3}}$ from \parencite{bardosOnsagerConjectureBounded2019,nguyenEnergyConservationInhomogeneous2020,wiedemannConservedQuantitiesRegularity2020}.

Let $\mathfrak{X}_{00}$ be the space of smooth vector fields compactly
supported in the interior of $M$. We say $\left(\mathcal{V},\mathfrak{p}\right)$
is a \textbf{weak solution }to the Euler equation when
\begin{itemize}
\item $\mathcal{V}\in L_{\mathrm{}}^{3}\left(I,\mathbb{P}L^{3}\mathfrak{X}\right)$,
$\mathfrak{p}\in L^{\frac{3}{2}}(I,H^{-\beta}\left(M\right))$ for
any $\beta\in\mathbb{N}_{0}$
\item $\forall\mathcal{X}\in C_{c}^{\infty}\left(I,\mathfrak{X}_{00}\right):$$\iint_{I\times M}\left\langle \mathcal{V},\partial_{t}\mathcal{X}\right\rangle +\left\langle \mathcal{V}\otimes\mathcal{V},\nabla\mathcal{X}\right\rangle +\mathfrak{p}\Div\mathcal{X}=0$.

The last condition means $\partial_{t}\mathcal{V}+\mathrm{div}(\mathcal{V}\otimes\mathcal{V})+\grad\mathfrak{p}=0$
as spacetime distributions.

\end{itemize}
\begin{rem}[Local elliptic regularity]
 As $\mathcal{V}\in L_{t}^{3}L^{3}\mathfrak{X}$, we have $\Delta\mathfrak{p}=-\Div\left(\Div\left(\mathcal{V}\otimes\mathcal{V}\right)\right)$
in $L_{t}^{\frac{3}{2}}H^{-2,\frac{3}{2}}(M)$. By embedding, there
is $\beta\in\mathbb{N}_{1}$ such that $\mathfrak{p}\in L^{\frac{3}{2}}\left(I,H^{-\beta,\frac{3}{2}}\left(M\right)\right)$.
Let $K\subset\subset W\subset\subset\accentset{\circ}{M}$ where $K$
and $W$ are precompact open sets. Then by interior elliptic regularity
(see \cite[Subsection 5.11, Theorem 11.1]{taylorPartialDifferentialEquations2011a}
and \cite[Subsection 13.6]{taylorPartialDifferentialEquations2011}),
we have for a.e. $t\in I:$ 
\[
\left\Vert \mathfrak{p}\left(t\right)\right\Vert _{L^{\frac{3}{2}}\left(K\right)}\lesssim_{K,W}\left\Vert \Delta\mathfrak{p}\left(t\right)\right\Vert _{H^{-2,\frac{3}{2}}(W)}+\left\Vert \mathfrak{p}\left(t\right)\right\Vert _{H^{-\beta,\frac{3}{2}}\left(W\right)}
\]
Then we can conclude $\mathfrak{p}\in L_{t}^{\frac{3}{2}}L^{\frac{3}{2}}(K)$,
for any $K\subset\accentset{\circ}{M}$ precompact. 
\end{rem}

As can be seen in \parencite{nguyenOnsagerConjectureEnergy2019,bardosOnsagerConjectureBounded2019,nguyenEnergyConservationInhomogeneous2020},
the correct replacement for the trace theorem is the following ``strip
decay'' hypothesis near the boundary:

\begin{equation}
\left\Vert \left(\frac{\left|\mathcal{V}\right|^{2}}{2}+\mathfrak{p}\right)\left\langle \mathcal{V},\widetilde{\nu}\right\rangle \right\Vert _{L_{t}^{1}L^{1}\left(M_{[\frac{r}{2},r]},\mathrm{avg}\right)}\xrightarrow{r\downarrow0}0\label{eq:strip-decay-hypo}
\end{equation}
where $\left\{ %
\begin{tabular}{l}
$\widetilde{\nu}$ is the extension of $\nu$ near the boundary.\tabularnewline
$M_{[r/2,r]}=\left\{ x\in M:\mathrm{dist}(x,\partial M)\in\left[r/2,r\right]\right\} $.\tabularnewline
$\mathrm{\avg}$ means the measure is normalized to become a probability
measure.\tabularnewline
\end{tabular}\right.$ 
\begin{thm}
Let $M$ be as in (\ref{eq:Euler}). Then $\left\Vert \mathcal{V}(t,\cdot)\right\Vert _{L^{2}(M)}$
is a.e. constant in time if $\left(\mathcal{V},\mathfrak{p}\right)$
is a weak solution with $\mathcal{V}\in L_{t}^{3}\mathbb{P}L^{3}\mathfrak{X}\cap L_{t}^{3}\widehat{B}_{3,V}^{\frac{1}{3}}\mathfrak{X}$
and (\ref{eq:strip-decay-hypo}) being true.
\end{thm}

\subsection{Outline of the paper}

In \Secref{Onsager's-conjecture}, we summarize the key tools from
\parencite{huynhHodgetheoreticAnalysisManifolds2019}, discuss some
connections between the heat flow and Besov spaces, and then prove
Onsager's conjecture. However, at certain points we will need some
local-type estimates involving the heat flow, which are themselves
derived from the construction of the heat kernel. To avoid interrupting
the flow of the paper, the local estimates are proved in \Secref{Local_analysis},
while the construction of the kernel, arguably the most technical
step of the paper, can be found in \Secref{Construction-of-heat-kernel}.

\section{Common notation\label{sec:Common-notation}}

Some common notation we use:
\begin{itemize}
\item $A\lesssim_{x,\neg y}B$ means $A\leq CB$ where $C>0$ depends on
$x$ and not $y$. Similarly, $A\sim_{x,\neg y}B$ means $A\lesssim_{x,\neg y}B$
and $B\lesssim_{x,\neg y}A$. When the dependencies are obvious by
context, we do not need to make them explicit.
\item $\mathbb{N}_{0},\mathbb{N}_{1}:$ the set of natural numbers, starting
with $0$ and $1$ respectively.
\item DCT: dominated convergence theorem, FTC: fundamental theorem of calculus,
WLOG: without loss of generality.
\item TVS: topological vector space. For TVS $X$, $Y\leq X$ means $Y$
is a subspace of $X$.
\item $\mathcal{L}(X,Y):$ the space of continuous linear maps from TVS
$X$ to $Y$. Also $\mathcal{L}(X)=\mathcal{L}(X,X)$.
\item $C^{0}(S\to Y)$: the space of bounded, continuous functions from
metric space $S$ to normed vector space $Y$. Not to be confused
with $C_{\text{loc}}^{0}(S\to Y)$, which is the space of locally
bounded, continuous functions.
\item $\left\Vert x\right\Vert _{D(A)}=\left\Vert x\right\Vert _{X}+\left\Vert Ax\right\Vert _{X}$
and $\left\Vert x\right\Vert _{D(A)}^{*}=\left\Vert Ax\right\Vert _{X}$
where $A$ is an unbounded operator on (real/complex) Banach space
$X$ and $x\in D(A).$ Note that $\left\Vert \cdot\right\Vert _{D(A)}^{*}$
is not always a norm. We also define $D(A^{\infty})=\cap_{k\in\mathbb{N}_{1}}D(A^{k}).$
\item $B(x,r)=B_{r}(x)$: the open ball of radius $r$ centered at $x$
in a metric space.
\end{itemize}

\section{Onsager's conjecture\label{sec:Onsager's-conjecture}}

\subsection{Summary of preliminaries}

We will quickly summarize the key tools that we need for the proof
(see \parencite[Subsection 3.1]{huynhHodgetheoreticAnalysisManifolds2019}
for the precise locations where they are proved).
\begin{defn}
For the rest of the paper, unless otherwise stated, let $M$ be a
compact, smooth, Riemannian $n$-dimensional manifold, with no or
smooth boundary. We also let $I\subset\mathbb{R}$ be an open time
interval. We write $M_{<r}=\{x\in M:\mathrm{dist}(x,\partial M)<r\}$
for $r>0$ small. Similarly define $M_{\geq r},M_{<r},M_{[r_{1},r_{2}]}$
etc. Let $\accentset{\circ}{M}$ denote the interior of $M$.

By the musical isomorphism, we can consider $\mathfrak{X}M$ (the
space of \textbf{smooth vector fields}) mostly the same as $\Omega^{1}(M)$
(the space of \textbf{smooth $1$-forms}), \emph{mutatis mutandis}.
We note that $\mathfrak{X}M$, $\mathfrak{X}\left(\partial M\right)$
and $\restr{\mathfrak{X}M}{\partial M}$ are different. Unless otherwise
stated, let the implicit domain be $M$, so $\mathfrak{X}$ stands
for $\mathfrak{X}M$, and similarly $\Omega^{k}$ for $\Omega^{k}M$.
For $X\in\mathfrak{X}$, we write $X^{\flat}$ as its dual $1$-form.

Let $\mathfrak{X}_{00}\left(M\right)$ denote the set of smooth vector
fields of compact support in $\accentset{\circ}{M}$. We define $\Omega_{00}^{k}\left(M\right)$
similarly (smooth differential forms with compact support in $\accentset{\circ}{M}$).

Let $\nu$ denote the outwards unit normal vector field on $\partial M$.
$\nu$ can be extended via geodesics to a smooth vector field $\widetilde{\nu}$
which is of unit length near the boundary (and cut off at some point
away from the boundary).

For $X\in\mathfrak{X}M,$ define $\mathbf{n}X=\left\langle X,\nu\right\rangle \nu\in\left.\mathfrak{X}M\right|_{\partial M}$
(the \textbf{normal part}) and $\mathbf{t}X=\left.X\right|_{\partial M}-\mathbf{n}X$
(the \textbf{tangential part}). We note that $\mathbf{t}X$ and $\mathbf{n}X$
only depend on $\restr{X}{\partial M}$, so $\mathbf{t}$ and $\mathbf{n}$
can be defined on $\restr{\mathfrak{X}M}{\partial M}$, and $\mathbf{t}\left(\left.\mathfrak{X}M\right|_{\partial M}\right)\iso\mathfrak{X}(\partial M)$.

For $\omega\in\Omega^{k}\left(M\right),$ define $\mathbf{t}\omega$
and $\mathbf{n}\omega$ by 
\[
\mathbf{t}\omega(X_{1},...,X_{k}):=\omega(\mathbf{t}X_{1},...,\mathbf{t}X_{k})\;\;\forall X_{j}\in\mathfrak{X}M,j=1,...,k
\]
and $\mathbf{n}\omega=\left.\omega\right|_{\partial M}-\mathbf{t}\omega$.
Note that $\left(\mathbf{n}X\right)^{\flat}=\mathbf{n}X^{\flat}\;\forall X\in\mathfrak{X}$.

Let $\nabla$ denote the \textbf{Levi-Civita connection}, $d$ the
\textbf{exterior derivative}, $\delta$ the \textbf{codifferential},
and $\Delta=-\left(d\delta+\delta d\right)$ the \textbf{Hodge-Laplacian},
which is defined on vector fields by the musical isomorphism.

Familiar scalar function spaces such as $L^{p},W^{m,p}$ (\textbf{Lebesgue-Sobolev
spaces}), $B_{p,q}^{s}$ (\textbf{Besov spaces}), $C^{0,\alpha}$
(\textbf{Holder spaces}) can be defined on $M$ by partitions of unity
and given a unique topology. Similarly, we define such function spaces
for \textbf{tensor fields} and \textbf{differential forms} on $M$
by partitions of unity and local coordinates. For instance, we can
define $L^{2}\mathfrak{X}$ or $B_{3,1}^{\frac{1}{3}}\mathfrak{X}$.
\end{defn}

\begin{fact}
\label{fact:Fact_embed} $\forall\alpha\in\left(\frac{1}{3},1\right),\forall p\in\left(1,\infty\right):W^{1,p}\mathfrak{X}\hookrightarrow B_{p,1}^{\frac{1}{p}}\mathfrak{X\hookrightarrow}L^{p}\mathfrak{X}$
and $C^{0,\alpha}\mathfrak{X}=B_{\infty,\infty}^{\alpha}\mathfrak{X}\hookrightarrow B_{3,\infty}^{\alpha}\mathfrak{X}\hookrightarrow B_{3,1}^{\frac{1}{3}}\mathfrak{X}\hookrightarrow B_{3,\infty}^{\frac{1}{3}}\mathfrak{X}$ 
\end{fact}

\begin{defn}
We write $\left\langle \cdot,\cdot\right\rangle $ to denote the \textbf{Riemannian
fiber metric} for tensor fields on $M$. We also define the dot product
\[
\left\langle \left\langle \sigma,\theta\right\rangle \right\rangle =\int_{M}\left\langle \sigma,\theta\right\rangle \vol
\]
where $\sigma$ and $\theta$ are tensor fields of the same type,
while $\vol$ is the \textbf{Riemannian volume form}. When there is
no possible confusion, we will omit writing $\vol$.

Define $\Omega(M)=\bigoplus_{k=0}^{n}\Omega^{k}(M)$ as the \textbf{graded
algebra} of differential forms where multiplication is the wedge product.
We then naturally define $W^{m,p}\Omega(M)=\bigoplus_{k=0}^{n}W^{m,p}\Omega^{k}(M)$,
and similarly for $B_{p,q}^{s},F_{p,q}^{s}$ spaces. Spaces like $\Omega_{N}\left(M\right)$,
$\Omega_{00}\left(M\right)$ are also defined by direct sums.

We define $\mathfrak{X}_{N}=\{X\in\mathfrak{X}:\begin{array}{c}
\mathbf{n}X=0\end{array}\}$ (\textbf{Neumann condition}). In order to define the Neumann condition
for less regular vector fields, we use the \textbf{trace theorem}.
We can similarly define $\Omega_{N}^{k}$.
\end{defn}

\begin{fact}[Trace theorem]
\label{fact:fact_trace}  Let $p\in[1,\infty)$. Then
\begin{itemize}
\item $B_{p,1}^{\frac{1}{p}}\left(M\right)\twoheadrightarrow L^{p}\left(\partial M\right)$
and $B_{p,1}^{\frac{1}{p}}\mathfrak{X}M\twoheadrightarrow L^{p}\restr{\mathfrak{X}M}{\partial M}$
are continuous surjections.
\item $\forall m\in\mathbb{N}_{1}:B_{p,1}^{m+\frac{1}{p}}\mathfrak{X}M\twoheadrightarrow B_{p,1}^{m}\restr{\mathfrak{X}M}{\partial M}\hookrightarrow W^{m,p}\restr{\mathfrak{X}M}{\partial M}$
is continuous.
\end{itemize}
\end{fact}

\begin{defn}
We define $\mathbb{P}$ as the \textbf{Leray projection}, which projects
$\mathfrak{X}$ onto $\mathrm{Ker}\left(\restr{\mathrm{div}}{\mathfrak{X}_{N}}\right)$.
Note that the Neumann condition is enforced by $\mathbb{P}$.
\begin{fact}
$\forall m\in\mathbb{N}_{0},\forall p\in\left(1,\infty\right)$, \textbf{$\mathbb{P}$
}is continuous on $W^{m,p}\mathfrak{X}$ and $\mathbb{P}\left(W^{m,p}\mathfrak{X}\right)=W^{m,p}\text{-}\mathrm{cl}\left(\mathbb{P}\mathfrak{X}\right)$
(closure in the $W^{m,p}$-topology).
\end{fact}

\end{defn}

We collect some results regarding our heat flow in one place:
\begin{fact}[Absolute Neumann heat flow]
\label{fact:Fact_heat} There exists a semigroup of operators $\left(S(t)\right)_{t\geq0}$
acting on $\cup_{p\in\left(1,\infty\right)}L^{p}\mathfrak{X}$ such
that
\begin{enumerate}
\item $S\left(t_{1}\right)S\left(t_{2}\right)=S\left(t_{1}+t_{2}\right)\;\forall t_{1},t_{2}\geq0$
and $S\left(0\right)=1$.
\item $\forall p\in\left(1,\infty\right),\forall X\in L^{p}\mathfrak{X}:$
\begin{enumerate}
\item $S(t)X\in\mathfrak{X}_{N}$ and $\partial_{t}\left(S(t)X\right)=\Delta S(t)X$
$\forall t>0$.
\item $S(t)X\xrightarrow[t\to t_{0}]{C^{\infty}}S\left(t_{0}\right)X$ $\forall t_{0}>0$.
\item $\left\Vert S(t)X\right\Vert _{W^{m,p}}\lesssim_{m,p}\left(\frac{1}{t}\right)^{\frac{m}{2}}\left\Vert X\right\Vert _{L^{p}}$
$\forall m\in\mathbb{N}_{0},\forall t\in\left(0,1\right)$.
\item $S(t)X\xrightarrow[t\to0]{L^{p}}X$.
\end{enumerate}
\item $\forall p\in\left(1,\infty\right),\forall X\in W^{1,p}\mathfrak{X}_{N}:$
\begin{enumerate}
\item $\left\Vert S(t)X\right\Vert _{W^{m+1,p}}\lesssim_{m,p}\left(\frac{1}{t}\right)^{\frac{m}{2}}\left\Vert X\right\Vert _{W^{1,p}}$
$\forall m\in\mathbb{N}_{0},\forall t\in\left(0,1\right)$.
\item $S(t)X\xrightarrow[t\to0]{W^{1,p}}X$.
\end{enumerate}
\item $S\left(t\right)\mathbb{P}=\mathbb{P}S\left(t\right)$ on $W^{m,p}\mathfrak{X}$
$\forall m\in\mathbb{N}_{0},\forall p\in\left(1,\infty\right),\forall t\geq0$.
\item $\left\langle \left\langle S(t)X,Y\right\rangle \right\rangle =\left\langle \left\langle X,S(t)Y\right\rangle \right\rangle \forall t\geq0,\forall p\in\left(1,\infty\right),\forall X\in L^{p}\mathfrak{X},\forall Y\in L^{p'}\mathfrak{X}$.
\end{enumerate}
\end{fact}

These estimates precisely fit the analogy $e^{t\Delta}\approx P_{\leq\frac{1}{\sqrt{t}}}$
where $P$ is the \textbf{Littlewood-Paley projection}. 

Analogous results hold for scalar functions and differential forms. 

We observe some basic identities from differential geometry:
\begin{itemize}
\item Using \textbf{Penrose abstract index notation}, for any smooth tensors
$T_{a_{1}...a_{k}}$, we define $\left(\nabla T\right)_{ia_{1}...a_{k}}=\nabla_{i}T_{a_{1}...a_{k}}$
and $\div T=\nabla^{i}T_{ia_{2}...a_{k}}.$
\item For all smooth tensors $T_{a_{1}...a_{k}}$ and $Q_{a_{1}...a_{k+1}}$:
\[
\int_{M}\nabla_{i}\left(T_{a_{1}...a_{k}}Q^{ia_{1}...a_{k}}\right)=\int_{M}\nabla_{i}T_{a_{1}...a_{k}}Q^{ia_{1}...a_{k}}+\int_{M}T_{a_{1}...a_{k}}\nabla_{i}Q^{ia_{1}...a_{k}}=\int_{\partial M}\nu_{i}T_{a_{1}...a_{k}}Q^{ia_{1}...a_{k}}
\]
\item $\left(\nabla_{a}\nabla_{b}-\nabla_{b}\nabla_{a}\right)T^{ij}{}_{kl}=-R_{ab\sigma}{}^{i}T^{\sigma j}{}_{kl}-R_{ab\sigma}{}^{j}T^{i\sigma}{}_{kl}+R_{abk}{}^{\sigma}T^{ij}{}_{\sigma l}+R_{abl}{}^{\sigma}T^{ij}{}_{k\sigma}$
for any tensor $T^{ij}{}_{kl}$, where $R$ is the \textbf{Riemann
curvature tensor}.\textbf{ }Similar identities hold for other types
of tensors. When we do not care about the exact indices and how they
contract, we can just write the \textbf{schematic identity }$\left(\nabla_{a}\nabla_{b}-\nabla_{b}\nabla_{a}\right)T^{ij}{}_{kl}=R*T.$
As $R$ is bounded on compact $M$, interchanging derivatives is a
zeroth-order operation on $M$. In particular, we have the \textbf{Weitzenbock
formula}:
\begin{equation}
\Delta X=\nabla_{i}\nabla^{i}X+R*X\;\forall X\in\mathfrak{X}M\label{eq:Weitzen_schematic}
\end{equation}
\end{itemize}
There is an elementary lemma which is useful for convergence (the
proof is straightforward and omitted):
\begin{lem}[Dense convergence]
\label{lem:dense_conv} Let $X,Y$ be (real/complex) Banach spaces
and $X_{0}\leq X$ be norm-dense. Let $(T_{j})_{j\in\mathbb{N}}$
be bounded in $\mathcal{L}(X,Y)$ and $T\in\mathcal{L}(X,Y)$.

If $T_{j}x_{0}\to Tx_{0}\;\forall x_{0}\in X_{0}$ then $T_{j}x\to Tx\;\forall x\in X$.
\end{lem}

As the heat flow does not preserve compact supports in $\accentset{\circ}{M}$,
it is not defined on distributions. This inspires the formulation
of \textbf{heatable currents}.
\begin{defn}[Heatable currents]

Define:
\begin{itemize}
\item $\mathscr{D}\Omega^{k}=\Omega_{00}^{k}=\colim\{\left(\Omega_{00}^{k}\left(K\right),C^{\infty}\text{ topo}\right):K\subset\accentset{\circ}{M}\text{ compact}\}$
as the space of \textbf{test $k$-forms} with \textbf{Schwartz's topology}\footnote{Confusingly enough, ``Schwartz's topology'' refers to the topology
on the space of distributions, not the topology for Schwartz functions.} (colimit in the category of locally convex TVS).
\item $\mathscr{D}'\Omega^{k}=\left(\mathscr{D}\Omega^{k}\right)^{*}$ as
the space of\textbf{ $k\text{-}$currents} (or \textbf{distributional
$k$}-\textbf{forms}), equipped with the weak{*} topology.
\item $\mathscr{D}_{N}\Omega^{k}=\{\omega\in\Omega^{k}:\mathbf{n}\Delta^{m}\omega=0,\mathbf{n}d\Delta^{m}\omega=0\;\forall m\in\mathbb{N}_{0}\}$
as the space of \textbf{heated $k$-forms }with the Frechet $C^{\infty}$
topology and $\mathscr{D}'_{N}\Omega^{k}=\left(\mathscr{D}_{N}\Omega^{k}\right)^{*}$
as the space of\textbf{ heatable $k$-currents} (or \textbf{heatable
distributional $k$}-\textbf{forms}) with the weak{*} topology.
\end{itemize}
In particular, $\mathscr{D}_{N}\mathfrak{X}$ is defined from $\mathscr{D}_{N}\Omega^{1}$
by the musical isomorphism, and it is invariant under our heat flow
(much like how the space of Schwartz functions $\mathcal{S}(\mathbb{R}^{n})$
is invariant under the Littlewood-Paley projection). By that analogy,
heatable currents are tempered distributions on manifolds, and we
can write 
\[
\left\langle \left\langle S(t)\Lambda,X\right\rangle \right\rangle =\left\langle \left\langle \Lambda,S\left(t\right)X\right\rangle \right\rangle \;\forall\Lambda\in\mathscr{D}'_{N}\mathfrak{X},\forall X\in\mathscr{D}_{N}\mathfrak{X},\forall t\geq0
\]
where the dot product $\left\langle \left\langle \cdot,\cdot\right\rangle \right\rangle $
is simply abuse of notation.
\end{defn}

\begin{fact}
\label{fact:D_N-basic-properties}Some basic properties of $\mathscr{D}_{N}\Omega\left(M\right)$
and $\mathscr{D}'_{N}\Omega\left(M\right)$:
\begin{itemize}
\item $\left\langle \left\langle \Delta X,Y\right\rangle \right\rangle =\left\langle \left\langle X,\Delta Y\right\rangle \right\rangle \;\forall X,Y\in\mathscr{D}_{N}\mathfrak{X}$.
\item $\mathscr{D}\Omega\hookrightarrow\mathscr{D}_{N}\Omega$ and $L^{p}\Omega\hookrightarrow\mathscr{D}'_{N}\Omega$
$\forall p\in\left(1,\infty\right)$. 
\item $S(t)\Lambda\in\mathscr{D}_{N}\Omega$ $\forall t>0,\forall\Lambda\in\mathscr{D}'_{N}\Omega$.
(a heatable current becomes heated once the heat flow is applied)
\item $W^{1,p}\text{-}\mathrm{cl}\left(\mathscr{D}_{N}\Omega\right)=W^{1,p}\Omega_{N}$
and $B_{p,1}^{\frac{1}{p}}\text{-}\mathrm{cl}\left(\mathbb{P}\mathscr{D}_{N}\Omega\right)=\mathbb{P}B_{p,1}^{\frac{1}{p}}\Omega_{N}$
$\forall p\in\left(1,\infty\right)$ 
\item $\forall X\in\mathscr{D}_{N}\Omega:S(t)X\xrightarrow[t\downarrow0]{C^{\infty}}X$
and $\partial_{t}\left(S(t)X\right)=\Delta S(t)X=S(t)\Delta X\;\forall t\geq0.$ 
\item $\forall t\in(0,1),\forall m,m'\in\mathbb{N}_{0},\forall p\in(1,\infty),\forall X\in\mathscr{D}_{N}\Omega:$
\begin{enumerate}
\item $\left\Vert S(t)X\right\Vert _{W^{m+m',p}}\lesssim\left(\frac{1}{t}\right)^{\frac{m'}{2}}\left\Vert X\right\Vert _{W^{m,p}}$
\item $\left\Vert S(t)X\right\Vert _{W^{m,p}}\lesssim\left(\frac{1}{t}\right)^{\frac{1}{2}\left(m-\frac{1}{p}\right)}\left\Vert X\right\Vert _{B_{p,1}^{\frac{1}{p}}}$
when $m\geq1$
\item $\left\Vert S(t)X\right\Vert _{B_{p,1}^{m+m'+\frac{1}{p}}}\lesssim\left(\frac{1}{t}\right)^{\frac{1}{2p}+\frac{m'}{2}}\left\Vert X\right\Vert _{W^{m,p}}$
\end{enumerate}
\end{itemize}
\end{fact}

\subsection{Heating the nonlinear term}

\label{subsec:heating_nonlinear} Recall integration by parts:
\[
\left\langle \left\langle \Div\left(Y\otimes Z\right),X\right\rangle \right\rangle =-\left\langle \left\langle Y\otimes Z,\nabla X\right\rangle \right\rangle +\int_{\partial M}\left\langle \nu,Y\right\rangle \left\langle Z,X\right\rangle \;\forall X,Y,Z\in\mathfrak{X}\left(M\right)
\]
Let $U,V\in B_{3,1}^{\frac{1}{3}}\mathfrak{X}$. Then $U\otimes V\in L^{1}\mathfrak{X}$
and $\Div\left(U\otimes V\right)$ is defined as a distribution. So
we will define the heatable $1$-current $\left(\Div\left(U\otimes V\right)\right)^{\flat}$
by
\[
\left\langle \left\langle \Div\left(U\otimes V\right),X\right\rangle \right\rangle :=-\left\langle \left\langle U\otimes V,\nabla X\right\rangle \right\rangle +\int_{\partial M}\left\langle \nu,U\right\rangle \left\langle V,X\right\rangle \;\forall X\in\mathscr{D}_{N}\mathfrak{X}\;(X\text{ is heated})
\]
It is continuous on $\mathscr{D}_{N}\mathfrak{X}$ since $\left|\left\langle \left\langle \Div\left(U\otimes V\right),X\right\rangle \right\rangle \right|\lesssim\left\Vert U\right\Vert _{B_{3,1}^{\frac{1}{3}}}\left\Vert V\right\Vert _{B_{3,1}^{\frac{1}{3}}}\left\Vert X\right\Vert _{B_{3,1}^{\frac{1}{3}}}+\left\Vert U\right\Vert _{L^{3}}\left\Vert V\right\Vert _{L^{3}}\left\Vert \nabla X\right\Vert _{L^{3}}$.
By the same formula and reasoning, we see that $\left(\Div\left(U\otimes V\right)\right)^{\flat}$
is not just heatable, but also a continuous linear functional on $\left(\mathfrak{X}\left(M\right),C^{\infty}\text{ topo}\right)$.

On the other hand, we can get away with less regularity by assuming
$U\in\mathbb{P}L^{2}\mathfrak{X}$ and $V\in L^{2}\mathfrak{X}$.
Then $\left(\Div\left(U\otimes V\right)\right)^{\flat}$ is heatable
as we simply need to define 
\begin{equation}
\left\langle \left\langle \Div\left(U\otimes V\right),X\right\rangle \right\rangle =-\left\langle \left\langle U\otimes V,\nabla X\right\rangle \right\rangle \;\forall X\in\mathfrak{X}\label{eq:int_parts}
\end{equation}

\subsection{\label{subsec:Besov-spaces}Besov spaces}

For the rest of the proof, we will write $e^{t\Delta}$ for the absolute
Neumann heat flow, as we will not need another heat flow. For $\varepsilon>0$
and vector field $X$, we will write $X^{\varepsilon}$ for $e^{\varepsilon\Delta}X$.
\nomenclature{$e^{t\Delta}$}{the absolute Neumann heat flow, defined for the proof of Onsager's conjecture\nomrefpage}

Now we define a crude version of the Littlewood-Paley projections:
$P_{\leq t}=e^{\frac{1}{t^{2}}\Delta}$ for $t>0$ and $P_{N}=P_{\leq N}-P_{\leq\frac{N}{2}}$
for $N>1,N\in2^{\mathbb{Z}}.$ 

The definition of $P_{\leq t}$ gives a quick Bernstein estimate:
\begin{thm}
\label{thm:global_bernstein}For $N\geq1$ and $X\in\mathscr{D}'_{N}\Omega^{k}$,
$\left\Vert P_{N}X\right\Vert _{p}\lesssim\frac{1}{N^{2}}\left\Vert P_{\leq\sqrt{2}N}X\right\Vert _{W^{2,p}}\lesssim\frac{1}{N}\left\Vert P_{\leq2N}X\right\Vert _{W^{1,p}}$
\end{thm}

\begin{proof}
Recall that $e^{\varepsilon\Delta}X\in\mathscr{D}{}_{N}\Omega^{k}$
$\forall\varepsilon>0$. Then observe that
\[
P_{N}X=\left(\exp\left(\frac{\Delta}{2N^{2}}\right)-\exp\left(\frac{7\Delta}{2N^{2}}\right)\right)\exp\left(\frac{\Delta}{2N^{2}}\right)X=\int_{\frac{7}{2N^{2}}}^{\frac{1}{2N^{2}}}\Delta e^{t\Delta}\exp\left(\frac{\Delta}{2N^{2}}\right)X\;\mathrm{d}t
\]
 and $P_{\leq\sqrt{2}N}=P_{\leq2N}P_{\leq2N}$.
\end{proof}
\begin{defn}
For $\alpha\in(0,1)$, $p\in\left(1,\infty\right),$ $q\in\left[1,\infty\right]$,
we define the \textbf{Besov heat space} $\widehat{B}_{p,q}^{\alpha}\Omega^{k}$
as the space of heatable $k$-currents $X$ where the norm 
\begin{align*}
\left\Vert X\right\Vert _{\widehat{B}_{p,q}^{\alpha}} & =\left\Vert X\right\Vert _{L^{p}}+\left\Vert s^{\frac{1}{2}\left(1-\alpha\right)}\left\Vert e^{s\Delta}X\right\Vert _{W^{1,p}}\right\Vert _{L^{q}\left(\frac{\mathrm{d}s}{s},\left(0,1\right)\right)}\\
 & \sim\left\Vert X\right\Vert _{L^{p}}+\left\Vert N^{\alpha-1}\left\Vert P_{\leq N}X\right\Vert _{W^{1,p}}\right\Vert _{l_{N}^{q}\left(N\in2^{\mathbb{Z}},N>1\right)}
\end{align*}
is finite. 
\end{defn}

Recall the theory of real interpolation. The following fact justifies
the name ``Besov'' in Besov heat space:
\begin{thm}
$\left[L^{p}\Omega^{k},W^{1,p}\Omega_{N}^{k}\right]_{\theta,q}=\widehat{B}_{p,q}^{\theta}\Omega^{k}$
for $q\in[1,\infty],p\in\left(1,\infty\right),\theta\in\left(0,1\right)$.
\end{thm}

\begin{proof}
By definition, $\widehat{B}_{p,q}^{\theta}\Omega^{k}\hookrightarrow L^{p}\Omega^{k}$.
We first show $\widehat{B}_{p,q}^{\theta}\Omega^{k}\hookrightarrow\left[L^{p}\Omega^{k},W^{1,p}\Omega_{N}^{k}\right]_{\theta,q}$.

Assume $\left\Vert X\right\Vert _{\widehat{B}_{p,q}^{\theta}}\leq1$.
Then we decompose $X\stackrel{L^{p}}{=}P_{\leq1}X+\sum_{N>1,N\in2^{\mathbb{Z}}}P_{N}X$
. Set $X_{0}=P_{\leq1}X$ and $X_{k}=P_{2^{-k}}X$ $\forall k\in\mathbb{Z},k\leq-1$.
Then by the J-method, and the fact that $X=\sum_{k\leq0}X_{k}$, we
have
\begin{align*}
\left\Vert X\right\Vert _{\left[L^{p}\Omega^{k},W^{1,p}\Omega_{N}^{k}\right]_{\theta,q}} & \lesssim\left\Vert 2^{-k\theta}\left\Vert X_{k}\right\Vert _{L^{p}}+2^{k\left(1-\theta\right)}\left\Vert X_{k}\right\Vert _{W^{1,p}}\right\Vert _{l_{k}^{q}\left(k\leq0\right)}\\
 & \lesssim\left\Vert X\right\Vert _{L^{p}}+\left\Vert \left(\frac{1}{2}\right)^{-m\theta}\left\Vert P_{2^{m}}X\right\Vert _{L^{p}}+\left(\frac{1}{2}\right)^{m\left(1-\theta\right)}\left\Vert P_{2^{m}}X\right\Vert _{W^{1,p}}\right\Vert _{l_{m}^{q}\left(m\geq1\right)}\\
 & \lesssim\left\Vert X\right\Vert _{L^{p}}+\left\Vert \left(\frac{1}{2}\right)^{m\left(1-\theta\right)}\left\Vert P_{\leq2^{m+1}}X\right\Vert _{W^{1,p}}\right\Vert _{l_{m}^{q}\left(m\geq1\right)}\lesssim1
\end{align*}
Now we will show $\left[L^{p}\Omega^{k},W^{1,p}\Omega_{N}^{k}\right]_{\theta,q}\hookrightarrow\widehat{B}_{p,q}^{\theta}\Omega^{k}$.
Assume $\left\Vert Y\right\Vert _{\left[L^{p}\Omega^{k},W^{1,p}\Omega_{N}^{k}\right]_{\theta,q}}\leq1$,
then $\left\Vert Y\right\Vert _{L^{p}}\lesssim1$. We will use the
K-method: for any $N\geq1,Y_{0}\in L^{p}\Omega^{k},Y_{1}\in W^{1,p}\Omega_{N}^{k}$
such that $Y=Y_{0}+Y_{1},$ we have
\[
\left\Vert P_{\leq N}Y\right\Vert _{W^{1,p}}\leq\left\Vert P_{\leq N}Y_{0}\right\Vert _{W^{1,p}}+\left\Vert P_{\leq N}Y_{1}\right\Vert _{W^{1,p}}\lesssim N\left\Vert Y_{0}\right\Vert _{L^{p}}+\left\Vert Y_{1}\right\Vert _{W^{1,p}}
\]
Note that this is why we need $W^{1,p}\Omega_{N}^{k}$ instead of
$W^{1,p}\Omega^{k}$. Then
\[
N^{\theta-1}\left\Vert P_{\leq N}Y\right\Vert _{W^{1,p}}\lesssim\inf_{Y_{0}+Y_{1}=Y}N^{\theta}\left\Vert Y_{0}\right\Vert _{L^{p}}+N^{\theta-1}\left\Vert Y_{1}\right\Vert _{W^{1,p}}=N^{\theta}K\left(\frac{1}{N},Y\right)
\]
so 
\[
\left\Vert N^{\theta-1}\left\Vert P_{\leq N}Y\right\Vert _{W^{1,p}}\right\Vert _{l_{N}^{q}\left(N\in2^{\mathbb{Z}},N>1\right)}\lesssim\left\Vert N^{-\theta}K\left(N,Y\right)\right\Vert _{l_{N}^{q}\left(N\in2^{\mathbb{Z}},N<1\right)}\leq1
\]
\end{proof}
\begin{rem}
We recover the standard Besov space when the manifold is boundaryless,
effectively generalizing the proof in \parencite[Appendix B]{isettHeatFlowApproach2014}.
More importantly, in the case with boundary, we have 
\[
\mathbb{P}B_{3,q}^{\frac{1}{3}}\Omega^{k}=\left[\mathbb{P}L^{3}\Omega^{k},\mathbb{P}W^{1,p}\Omega^{k}\right]_{1/3,q}=\left[\mathbb{P}L^{3}\Omega^{k},\mathbb{P}W^{1,p}\Omega_{N}^{k}\right]_{1/3,q}=\mathbb{P}\widehat{B}_{3,q}^{\frac{1}{3}}\Omega^{k}
\]
for $q\in\left[1,\infty\right]$. The fact that we need to apply the
Leray projection is an important technicality.
\end{rem}

\begin{defn}
\label{def:besov_v_space}For $p\in\left(1,\infty\right)$, we say
$X\in\widehat{B}_{p,V}^{1/p}\mathfrak{X}\left(M\right)$ if $X\in L^{p}\mathfrak{X}\left(M\right)$
and $\forall r>0:$
\begin{equation}
N^{\frac{1}{p}-1}\left\Vert P_{\leq N}X\right\Vert _{W^{1,p}\left(M_{>r}\right)}\xrightarrow{N\to\infty}0\label{eq:temp_3r}
\end{equation}
Similarly, we say $\mathcal{X}\in L_{t}^{p}\widehat{B}_{p,V}^{1/p}\mathfrak{X}\left(M\right)$
if $\mathcal{X}\in L_{t}^{p}L^{p}\mathfrak{X}\left(M\right)$ and
$\forall r>0:$
\begin{equation}
N^{\frac{1}{p}-1}\left\Vert P_{\leq N}\mathcal{X}\right\Vert _{L_{t}^{p}W^{1,p}\left(M_{>r}\right)}\xrightarrow{N\to\infty}0\label{eq:vanishing_property}
\end{equation}
\end{defn}

\begin{rem*}
The vanishing property in (\ref{eq:vanishing_property}) becomes important
for the commutator estimate in Onsager's conjecture at the critical
regularity $\frac{1}{3}$, while higher regularity has enough room
for vanishing in norm (which is better).

It is shown in \Corref{equiv_cN} that (\ref{eq:temp_3r}) is equivalent
to 
\[
N^{\frac{1}{p}}\left\Vert P_{>N}X\right\Vert _{L^{p}\left(M_{>r}\right)}\xrightarrow{N\to\infty}0\;\forall r>0
\]
\end{rem*}
We briefly note that when $\partial M=\emptyset,$ it is customary
to set $\mathrm{dist}\left(x,\partial M\right)=\infty$, $M_{>r}=M=\accentset{\circ}{M}$,
$M_{<r}=\emptyset$, and $\mathscr{D}_{N}\mathfrak{X}M=\mathscr{D}\mathfrak{X}M=\mathfrak{X}M$.

Recall the space $\widehat{B}_{3,c(\mathbb{N})}^{1/3}\mathfrak{X}=\widehat{B}_{3,\infty}^{1/3}\text{-cl}\left(\mathscr{D}_{N}\mathfrak{X}\right)$
from \parencite{isettHeatFlowApproach2014}. 
\begin{lem}
\label{lem:equal_isett}$\widehat{B}_{3,c(\mathbb{N})}^{1/3}\mathfrak{X}\hookrightarrow\widehat{B}_{3,V}^{1/3}\mathfrak{X}$.
When $\partial M=\emptyset$, $\widehat{B}_{3,V}^{1/3}\mathfrak{X}=\widehat{B}_{3,c(\mathbb{N})}^{1/3}\mathfrak{X}$. 
\end{lem}

\begin{proof}
Observe that $\mathscr{D}_{N}\mathfrak{X}\hookrightarrow\widehat{B}_{3,V}^{1/3}\mathfrak{X}$
. For any $r>0,N\geq1$ and $X\in\widehat{B}_{3,\infty}^{1/3}$, 
\[
N^{-2/3}\left\Vert P_{\leq N}X\right\Vert _{W^{1,3}\left(M_{>r}\right)}\leq N^{-2/3}\left\Vert P_{\leq N}X\right\Vert _{W^{1,3}\left(M\right)}\lesssim\left\Vert X\right\Vert _{\widehat{B}_{3,\infty}^{1/3}\left(M\right)}
\]
 Because $\left\{ f\in l^{\infty}\left(\mathbb{N}\right):f(k)\xrightarrow{k\to\infty}0\right\} $
is closed in $l^{\infty}\left(\mathbb{N}\right)$, we conclude $\widehat{B}_{3,c(\mathbb{N})}^{1/3}\mathfrak{X}\hookrightarrow\widehat{B}_{3,V}^{1/3}\mathfrak{X}$.

On the other hand, when $\partial M=\emptyset$, observe that $M_{>r}=M$.
Let $X\in\widehat{B}_{3,V}^{1/3}\mathfrak{X}$. We aim to show $P_{\leq K}X\xrightarrow[K\to\infty]{\widehat{B}_{3,\infty}^{1/3}}X$.
For any $N,K\in2^{\mathbb{N}_{0}}:$
\[
N^{-2/3}\left\Vert P_{\leq N}X\right\Vert _{W^{1,3}\left(M\right)}\gtrsim_{\neg K,\neg N}N^{-2/3}\left\Vert P_{\leq N}\left(1-P_{\leq K}\right)X\right\Vert _{W^{1,3}\left(M\right)}\xrightarrow{K\to\infty}0
\]
Let $N_{0}\in2^{\mathbb{N}_{1}}$. Then observe that 
\begin{align*}
\limsup_{K\to\infty}\left\Vert N^{-2/3}\left\Vert P_{\leq N}\left(1-P_{\leq K}\right)X\right\Vert _{W^{1,3}\left(M\right)}\right\Vert _{l_{N>1}^{\infty}} & \leq\underbrace{\limsup_{K\to\infty}\left\Vert N^{-2/3}\left\Vert P_{\leq N}\left(1-P_{\leq K}\right)X\right\Vert _{W^{1,3}\left(M\right)}\right\Vert _{l_{N}^{\infty}\left(N\in2^{\mathbb{N}_{0}},N<N_{0}\right)}}_{0}\\
 & \;\;\;\;\;+\left\Vert N^{-2/3}\left\Vert P_{\leq N}X\right\Vert _{W^{1,3}\left(M\right)}\right\Vert _{l_{N}^{\infty}\left(N\in2^{\mathbb{N}_{0}},N\geq N_{0}\right)}
\end{align*}
As $N_{0}$ is arbitrary, let $N_{0}\to\infty$ and we are done.
\end{proof}
\begin{rem}
On the other hand, \Thmref{contain_VMO} shows that, on flat backgrounds,
$\widehat{B}_{3,V}^{1/3}$ coincides with the VMO-type Besov space
$\underline{B}_{3,\text{VMO}}^{1/3}$ from \parencite{bardosOnsagerConjectureBounded2019,nguyenEnergyConservationInhomogeneous2020}.
\end{rem}

We will also need to borrow a result from \Secref{Local_analysis},
which allows us to employ cutoffs. 
\begin{fact}[Pointwise multiplier]
\label{fact:cutoff_interpol} If $f\in\mathscr{D}\left(M\right)$
and $\mathcal{X}\in L_{t}^{3}\widehat{B}_{3,V}^{\frac{1}{3}}\mathfrak{X},$
then $f\mathcal{X}\in L_{t}^{3}\widehat{B}_{3,V}^{\frac{1}{3}}\mathfrak{X}$. 
\end{fact}

\subsection{Proof of Onsager's conjecture}
\begin{defn}
We define the cutoffs 
\begin{equation}
\psi_{r}(x)=\Psi_{r}\left(\mathrm{dist}\left(x,\partial M\right)\right)\label{eq:cutoff}
\end{equation}
where $r>0$ small, $\Psi_{r}\in C^{\infty}([0,\infty),[0,\infty))$
such that $\mathbf{1}_{[0,\frac{3}{4}r)}\geq\Psi_{r}\geq\mathbf{1}_{[0,\frac{r}{2}]}$
and $\left\Vert \Psi'_{r}\right\Vert _{\infty}\lesssim\frac{1}{r}.$\\
Then there is $f_{r}$ smooth such that $\nabla\psi_{r}(x)=f_{r}(x)\widetilde{\nu}(x)$
with $|f_{r}(x)|\lesssim\frac{1}{r}$ and $\supp f_{r}\subset M_{[\frac{r}{2},\frac{3r}{4}]}$.
\nomenclature{$\psi_{r},f_{r}$}{cutoffs on $M$ living near the boundary\nomrefpage}
\end{defn}

Let $\chi_{r}=1-\psi_{r}$. Then $\nabla\chi_{r}=-f_{r}\widetilde{\nu}$.
As usual, there is a \textbf{commutator estimate} which we will now
assume (leaving the proof to later):
\begin{align}
 & \int_{I}\eta\left\langle \left\langle \Div\left(\mathcal{U}\otimes\chi_{r}\mathcal{U}\right)^{2\varepsilon},\left(\chi_{r}\mathcal{U}\right)^{2\varepsilon}\right\rangle \right\rangle -\int_{I}\eta\left\langle \left\langle \Div\left(\mathcal{U}^{2\varepsilon}\otimes\left(\chi_{r}\mathcal{U}\right)^{2\varepsilon}\right),\left(\chi_{r}\mathcal{U}\right)^{2\varepsilon}\right\rangle \right\rangle \nonumber \\
= & \int_{I}\eta\left\langle \left\langle \Div\left(\mathcal{U}\otimes\chi_{r}\mathcal{U}\right)^{3\varepsilon},\left(\chi_{r}\mathcal{U}\right)^{\varepsilon}\right\rangle \right\rangle -\int_{I}\eta\left\langle \left\langle \Div\left(\mathcal{U}^{2\varepsilon}\otimes\left(\chi_{r}\mathcal{U}\right)^{2\varepsilon}\right)^{\varepsilon},\left(\chi_{r}\mathcal{U}\right)^{\varepsilon}\right\rangle \right\rangle \xrightarrow{\varepsilon\downarrow0}0\label{eq:commutator_est-1}
\end{align}
for fixed $r>0$, $\mathcal{U}\in L_{t}^{3}\widehat{B}_{3,V}^{\frac{1}{3}}\mathfrak{X}\cap L_{t}^{3}\mathbb{P}L^{3}\mathfrak{X},\eta\in C_{c}^{\infty}\left(I\right)$.
\begin{rem*}
For any $U$ in $\mathbb{P}L^{2}\mathfrak{X}$ and $V\in L^{2}\mathfrak{X}$,
$\Div\left(U\otimes V\right)^{\flat}$ is a heatable $1$-current
(see \Subsecref{heating_nonlinear}). In particular, for $\varepsilon>0$,
$\Div\left(U\otimes V\right)^{\varepsilon}$ is smooth and
\begin{equation}
\left\langle \left\langle \Div\left(U\otimes V\right)^{\varepsilon},Y\right\rangle \right\rangle =-\left\langle \left\langle U\otimes V,\nabla\left(Y^{\varepsilon}\right)\right\rangle \right\rangle \;\forall Y\in\mathfrak{X}\label{eq:distribution_parts_ok}
\end{equation}
Consequently, (\ref{eq:commutator_est-1}) is well-defined.

Notation: we write $\Div\left(\mathcal{U}\otimes\mathcal{V}\right)^{\varepsilon}$
for $\left(\Div\left(\mathcal{U}\otimes\mathcal{V}\right)\right)^{\varepsilon}$
and $\nabla\mathcal{U}^{\varepsilon}$ for $\nabla\left(\mathcal{U}^{\varepsilon}\right)$
(recall that the heat flow does not work on tensors $\mathcal{U}\otimes\mathcal{V}$
and $\nabla\mathcal{U}$). 
\end{rem*}
\begin{thm}[Onsager's conjecture]
 Let $M$ be a compact, oriented Riemannian manifold with no or smooth
boundary. Let $\left(\mathcal{V},\mathfrak{p}\right)$ be a weak solution
and $\mathcal{V}\in L_{t}^{3}\widehat{B}_{3,V}^{\frac{1}{3}}\mathfrak{X}\cap L_{t}^{3}\mathbb{P}L^{3}\mathfrak{X}$.

Assume (\ref{eq:commutator_est-1}) is true. Also assume strip decay:
\[
\left\Vert \left(\frac{\left|\mathcal{V}\right|^{2}}{2}+\mathfrak{p}\right)\left\langle \mathcal{V},\widetilde{\nu}\right\rangle \right\Vert _{L_{t}^{1}L^{1}\left(M_{[\frac{r}{2},r]},\mathrm{avg}\right)}\xrightarrow{r\downarrow0}0
\]

Then we can show
\[
\int_{I}\eta'(t)\left\langle \left\langle \mathcal{V}(t),\mathcal{V}(t)\right\rangle \right\rangle \mathrm{d}t=0\;\forall\eta\in C_{c}^{\infty}(I)
\]
Consequently, $\left\langle \left\langle \mathcal{V}(t),\mathcal{V}(t)\right\rangle \right\rangle $
is constant for a.e. $t\in I$.
\end{thm}

\begin{proof}
Let $\Phi\in C_{c}^{\infty}(\mathbb{R})$ and $\Phi_{\tau}\xrightarrow{\tau\downarrow0}\delta_{0}$
be a radially symmetric mollifier. Write $\mathcal{V}^{\varepsilon}$
for $e^{\varepsilon\Delta}\mathcal{V}$ (spatial mollification) and
$\mathcal{V}_{\tau}$ for $\Phi_{\tau}*\mathcal{V}$ (temporal mollification).
First, we use the cutoff $\chi_{r}$ and mollify in time and space
\[
\frac{1}{2}\int_{I}\eta'\left\langle \left\langle \mathcal{V},\mathcal{V}\right\rangle \right\rangle \stackrel{\text{DCT}}{=}\lim_{r\downarrow0}\lim_{\varepsilon\downarrow0}\lim_{\tau\downarrow0}\frac{1}{2}\int_{I}\eta'\left\langle \left\langle \left(\chi_{r}\mathcal{V}\right)_{\tau}^{\varepsilon},\left(\chi_{r}\mathcal{V}\right)_{\tau}^{\varepsilon}\right\rangle \right\rangle 
\]
Then for $\varepsilon,\tau$ small, we want to get rid of the time
derivative:
\begin{align*}
\frac{1}{2}\int_{I}\eta'\left\langle \left\langle \left(\chi_{r}\mathcal{V}\right)_{\tau}^{\varepsilon},\left(\chi_{r}\mathcal{V}\right)_{\tau}^{\varepsilon}\right\rangle \right\rangle  & =-\int_{I}\eta\left\langle \left\langle \partial_{t}\left(\chi_{r}\mathcal{V}\right)_{\tau}^{\varepsilon},\left(\chi_{r}\mathcal{V}\right)_{\tau}^{\varepsilon}\right\rangle \right\rangle \\
 & =-\int_{I}\left\langle \left\langle \partial_{t}\left(\eta\left(\chi_{r}\mathcal{V}\right)_{\tau}^{\varepsilon}\right),\left(\chi_{r}\mathcal{V}\right)_{\tau}^{\varepsilon}\right\rangle \right\rangle +\int_{I}\eta'\left\langle \left\langle \left(\chi_{r}\mathcal{V}\right)_{\tau}^{\varepsilon},\left(\chi_{r}\mathcal{V}\right)_{\tau}^{\varepsilon}\right\rangle \right\rangle 
\end{align*}
We now use the definition of weak solution (WS), and exploit the commutativity
between spatial and temporal operators. For the sake of exposition,
we will freely cancel the error terms that go to zero upon taking
the limits. At the end of the proof, we will show why they can be
cancelled. 
\begin{align*}
 & \frac{1}{2}\int_{I}\eta'\left\langle \left\langle \left(\chi_{r}\mathcal{V}\right)_{\tau}^{\varepsilon},\left(\chi_{r}\mathcal{V}\right)_{\tau}^{\varepsilon}\right\rangle \right\rangle =\int_{I}\left\langle \left\langle \partial_{t}\left(\eta\left(\chi_{r}\mathcal{V}\right)_{\tau}^{\varepsilon}\right),\left(\chi_{r}\mathcal{V}\right)_{\tau}^{\varepsilon}\right\rangle \right\rangle =\int_{I}\left\langle \left\langle \partial_{t}\left[\left(\eta\left(\chi_{r}\mathcal{V}\right)_{\tau}^{2\varepsilon}\right)_{\tau}\chi_{r}\right],\mathcal{V}\right\rangle \right\rangle \\
\stackrel{\mathrm{WS}}{=\joinrel=} & -\int_{I}\left\langle \left\langle \nabla\left[\left(\eta\left(\chi_{r}\mathcal{V}\right)_{\tau}^{2\varepsilon}\right)_{\tau}\chi_{r}\right],\mathcal{V}\otimes\mathcal{V}\right\rangle \right\rangle -\left\langle \left\langle \Div\left[\left(\eta\left(\chi_{r}\mathcal{V}\right)_{\tau}^{2\varepsilon}\right)_{\tau}\chi_{r}\right],\mathfrak{p}\right\rangle \right\rangle \\
= & -\int_{I}\left\langle \left\langle \nabla\left[\left(\eta\left(\chi_{r}\mathcal{V}\right)_{\tau}^{2\varepsilon}\right)_{\tau}\right]\chi_{r},\mathcal{V}\otimes\mathcal{V}\right\rangle \right\rangle -\cancel{\left\langle \left\langle \left(\eta\left(\chi_{r}\mathcal{V}\right)_{\tau}^{2\varepsilon}\right)_{\tau}\otimes\nabla\chi_{r},\mathcal{V}\otimes\mathcal{V}\right\rangle \right\rangle }-\left\langle \left\langle \Div\left(\eta\left(\chi_{r}\mathcal{V}\right)_{\tau}^{2\varepsilon}\right)_{\tau}\chi_{r},\mathfrak{p}\right\rangle \right\rangle \\
 & \;\;\;\;\;-\cancel{\left\langle \left\langle \left(\eta\left(\chi_{r}\mathcal{V}\right)_{\tau}^{2\varepsilon}\right)_{\tau}\cdot\nabla\chi_{r},\mathfrak{p}\right\rangle \right\rangle }\\
= & -\int_{I}\left\langle \left\langle \left(\eta\nabla\left(\chi_{r}\mathcal{V}\right)_{\tau}^{2\varepsilon}\right)_{\tau}\chi_{r},\mathcal{V}\otimes\mathcal{V}\right\rangle \right\rangle -\cancel{\left\langle \left\langle \left(\eta\Div\left(\left(\chi_{r}\mathcal{V}\right)^{2\varepsilon}\right)_{\tau}\right)_{\tau}\chi_{r},\mathfrak{p}\right\rangle \right\rangle }\\
= & -\int_{I}\eta\left\langle \left\langle \nabla\left(\chi_{r}\mathcal{V}\right)_{\tau}^{2\varepsilon},\chi_{r}\left(\mathcal{V}\otimes\mathcal{V}\right)_{\tau}\right\rangle \right\rangle 
\end{align*}
 As there is no longer a time derivative on $\mathcal{V}$, we will
get rid of $\tau$ by letting $\tau\downarrow0$ (fine as $\mathcal{V}$
is $L^{3}$ in time). Also recall \Eqref{distribution_parts_ok}:
\begin{align*}
 & \lim_{r\downarrow0}\lim_{\varepsilon\downarrow0}\frac{1}{2}\int_{I}\eta'\left\langle \left\langle \left(\chi_{r}\mathcal{V}\right)^{\varepsilon},\left(\chi_{r}\mathcal{V}\right)^{\varepsilon}\right\rangle \right\rangle =-\lim_{r\downarrow0}\lim_{\varepsilon\downarrow0}\int_{I}\eta\left\langle \left\langle \nabla\left(\chi_{r}\mathcal{V}\right)^{2\varepsilon},\mathcal{V}\otimes\chi_{r}\mathcal{V}\right\rangle \right\rangle \\
= & \lim_{r\downarrow0}\lim_{\varepsilon\downarrow0}\int_{I}\eta\left\langle \left\langle \left(\chi_{r}\mathcal{V}\right)^{\varepsilon},\Div\left(\mathcal{V}\otimes\chi_{r}\mathcal{V}\right)^{\varepsilon}\right\rangle \right\rangle \\
= & \lim_{r\downarrow0}\lim_{\varepsilon\downarrow0}\int_{I}\eta\left\langle \left\langle \left(\chi_{r}\mathcal{V}\right)^{\varepsilon},\Div\left(\mathcal{V}^{\varepsilon}\otimes\left(\chi_{r}\mathcal{V}\right)^{\varepsilon}\right)\right\rangle \right\rangle  & \text{(commutator estimate)}\\
= & \lim_{r\downarrow0}\lim_{\varepsilon\downarrow0}\int_{I}\eta\left\langle \left\langle \left(\chi_{r}\mathcal{V}\right)^{\varepsilon},\nabla_{\mathcal{V}^{\varepsilon}}\left(\chi_{r}\mathcal{V}\right)^{\varepsilon}\right\rangle \right\rangle =\lim_{r\downarrow0}\lim_{\varepsilon\downarrow0}\int_{I}\eta\int_{M}\mathcal{V}^{\varepsilon}\left(\frac{\left|\left(\chi_{r}\mathcal{V}\right)^{\varepsilon}\right|^{2}}{2}\right)=0 & \text{as }\mathcal{V}^{\varepsilon}\in\mathbb{P}\mathfrak{X}
\end{align*}

We are done. As promised, we now show why we could cancel the error
terms previously. Let us calculate
\begin{equation}
-\lim_{r\downarrow0}\lim_{\varepsilon\downarrow0}\lim_{\tau\downarrow0}\int_{I}\left\langle \left\langle \left(\eta\left(\chi_{r}\mathcal{V}\right)_{\tau}^{2\varepsilon}\right)_{\tau}\otimes\nabla\chi_{r},\mathcal{V}\otimes\mathcal{V}\right\rangle \right\rangle +\left\langle \left\langle \left(\eta\left(\chi_{r}\mathcal{V}\right)_{\tau}^{2\varepsilon}\right)_{\tau}\cdot\nabla\chi_{r}+\left(\eta\Div\left(\left(\chi_{r}\mathcal{V}\right)^{2\varepsilon}\right)_{\tau}\right)_{\tau}\chi_{r},\mathfrak{p}\right\rangle \right\rangle \label{eq:big_expr_ons}
\end{equation}
Recall from \parencite{huynhHodgetheoreticAnalysisManifolds2019}
that $\delta_{c}=\delta\restriction_{\Omega_{N}}$ and $\delta_{c}^{\mathscr{D}_{N}^{'}}$
is the extension of $\delta_{c}$ to heatable currents, defined by
\[
\delta_{c}^{\mathscr{D}_{N}^{'}}\Lambda\left(\phi\right)=\Lambda\left(d\phi\right)\;\forall\Lambda\in\mathscr{D}_{N}^{'}\Omega,\forall\phi\in\mathscr{D}_{N}\Omega
\]
Then the fact that $\mathbb{P}\mathcal{V}^{\flat}=\mathcal{V}^{\flat}$
is equivalent to $\delta_{c}^{\mathscr{D}_{N}^{'}}\mathcal{V}^{\flat}=0$.
This implies:
\begin{equation}
-\Div\left(\left(\chi_{r}\mathcal{V}\right)^{2\varepsilon}\right)=\delta_{c}\left(\left(\chi_{r}\mathcal{V}^{\flat}\right)^{2\varepsilon}\right)=\left(\delta_{c}^{\mathscr{D}_{N}^{'}}\left(\chi_{r}\mathcal{V}^{\flat}\right)\right)^{2\varepsilon}=\left(-\nabla\chi_{r}\cdot\mathcal{V}+\chi_{r}\delta_{c}^{\mathscr{D}_{N}^{'}}\mathcal{V}^{\flat}\right)^{2\varepsilon}=\left(f_{r}\widetilde{\nu}\cdot\mathcal{V}\right)^{2\varepsilon}\label{eq:intrin_exam}
\end{equation}
With that simplification, and the lack of any time derivatives, (\ref{eq:big_expr_ons})
becomes
\begin{align*}
 & \lim_{r\downarrow0}\lim_{\varepsilon\downarrow0}\int_{I}\eta\left\langle \left\langle \left(\chi_{r}\mathcal{V}\right)^{2\varepsilon}\otimes f_{r}\widetilde{\nu},\mathcal{V}\otimes\mathcal{V}\right\rangle \right\rangle +\eta\left\langle \left\langle \left(\chi_{r}\mathcal{V}\right)^{2\varepsilon}\cdot f_{r}\widetilde{\nu},\mathfrak{p}\right\rangle \right\rangle +\eta\left\langle \left\langle \left(f_{r}\widetilde{\nu}\cdot\mathcal{V}\right)^{2\varepsilon},\chi_{r}\mathfrak{p}\right\rangle \right\rangle \\
= & \lim_{r\downarrow0}\int_{I}\eta\left\langle \left\langle \mathcal{V}\cdot\mathcal{V},\chi_{r}f_{r}\widetilde{\nu}\cdot\mathcal{V}\right\rangle \right\rangle +2\eta\left\langle \left\langle \mathcal{V}\cdot\chi_{r}f_{r}\widetilde{\nu},\mathfrak{p}\right\rangle \right\rangle =\lim_{r\downarrow0}\int_{I}2\eta\left\langle \left\langle \frac{\left|\mathcal{V}\right|^{2}}{2}+\mathfrak{p},\chi_{r}f_{r}\widetilde{\nu}\cdot\mathcal{V}\right\rangle \right\rangle \\
= & \lim_{r\downarrow0}O\left(\int_{I}\left|\eta\right|\int_{M_{[\frac{r}{2},r]}}\left|\frac{\left|\mathcal{V}\right|^{2}}{2}+\mathfrak{p}\right|\frac{1}{r}\left|\left\langle \widetilde{\nu},\mathcal{V}\right\rangle \right|\right)=0 & \text{(strip decay)}
\end{align*}
\end{proof}
\begin{rem}
The proof did not much use the Besov regularity of $\mathcal{V}$,
which is mainly used for the commutator estimate.
\end{rem}

It is the commutator estimate that presents the main difficulty. We
proceed similarly as in \parencite{isettHeatFlowApproach2014}.

\uline{Note that from this point on \mbox{$r>0$} is fixed.}

Let $\mathcal{U}\in L_{t}^{3}\widehat{B}_{3,V}^{\frac{1}{3}}\mathfrak{X}\cap L_{t}^{3}\mathbb{P}L^{3}\mathfrak{X}$
and $\chi_{r}$ be as before.

By setting $\mathcal{U}(t)$ to $0$ for $t$ in a null set, WLOG
we assume $\mathcal{U}(t)\in\mathbb{P}L^{3}\mathfrak{X}\cap\widehat{B}_{3,V}^{1/3}\mathfrak{X}\;\forall t\in I$.

Define the commutator 
\[
\mathcal{W}(t,s)=\Div\left(\mathcal{U}\left(t\right)\otimes\chi_{r}\mathcal{U}\left(t\right)\right)^{3s}-\Div\left(\mathcal{U}^{2s}\otimes\left(\chi_{r}\mathcal{U}\left(t\right)\right)^{2s}\right)^{s}
\]
When $t$ and $s$ are implicitly understood, we will not write them.
As $\Div\left(\mathcal{U}(t)\otimes\mathcal{U}(t)\right)^{3s}$ solves
$\left(\partial_{s}-3\Delta\right)\mathcal{X}=0$, we define $\mathcal{N}=\left(\partial_{s}-3\Delta\right)\mathcal{W}$.
Then $\mathcal{W}$ and $\mathcal{N}$ obey the Duhamel formula.
\begin{lem}[Duhamel]
For fixed $t_{0}\in I$ and $s>0$: $\int_{\varepsilon}^{s}\mathcal{N}\left(t_{0},\sigma\right)^{3(s-\sigma)}\mathrm{d}\sigma\xrightarrow{\varepsilon\downarrow0}\mathcal{W}\left(t_{0},s\right)$
in $\mathscr{D}'_{N}\mathfrak{X}$.
\end{lem}

\begin{proof}
Let $\varepsilon>0$. By the smoothing effect of $e^{s\Delta}$, $\mathcal{W}(t_{0},\cdot)$
and $\mathcal{N}(t_{0},\cdot)$ are in $C_{\mathrm{loc}}^{0}\left((0,1],\mathscr{D}_{N}\mathfrak{X}\right)$.
As $\left(e^{s\Delta}\right)_{s\geq0}$ is a $C_{0}$ semigroup on
$\left(H^{m}\text{-}\mathrm{cl}\left(\mathscr{D}_{N}\mathfrak{X}\right),\left\Vert \cdot\right\Vert _{H^{m}}\right)$
$\forall m\in\mathbb{N}_{0}$, and a semigroup basically corresponds
to an ODE (cf. \parencite[Appendix A, Proposition 9.10 \& 9.11]{taylorPartialDifferentialEquations2011a}),
from $\partial_{s}\mathcal{W}=3\Delta\mathcal{W}+\mathcal{N}$ for
$s\geq\varepsilon$ we get the Duhamel formula
\[
\forall s>\varepsilon:\mathcal{W}(t_{0},s)=\mathcal{W}\left(t_{0},\varepsilon\right)^{3(s-\varepsilon)}+\int_{\varepsilon}^{s}\mathcal{N}\left(t_{0},\sigma\right)^{3\left(s-\sigma\right)}\mathrm{d}\sigma
\]
So we only need to show $\mathcal{W}\left(t_{0},\varepsilon\right)^{3(s-\varepsilon)}\xrightarrow[\varepsilon\downarrow0]{\mathscr{D}'_{N}\mathfrak{X}}$0.
Let $X\in\mathscr{D}_{N}\mathfrak{X}.$ 
\begin{align*}
 & \left\langle \left\langle X,\mathcal{W}\left(t_{0},\varepsilon\right)^{3(s-\varepsilon)}\right\rangle \right\rangle =\left\langle \left\langle X^{3(s-\varepsilon)},\Div\left(\mathcal{U}\left(t_{0}\right)\otimes\chi_{r}\mathcal{U}\left(t_{0}\right)\right)^{3\varepsilon}-\Div\left(\mathcal{U}\left(t_{0}\right)^{2\varepsilon}\otimes\left(\chi_{r}\mathcal{U}\left(t_{0}\right)\right)^{2\varepsilon}\right)^{\varepsilon}\right\rangle \right\rangle \\
= & -\left\langle \left\langle \nabla\left(X^{3s}\right),\mathcal{U}\left(t_{0}\right)\otimes\chi_{r}\mathcal{U}\left(t_{0}\right)\right\rangle \right\rangle +\left\langle \left\langle \nabla\left(X^{3s-2\varepsilon}\right),\mathcal{U}\left(t_{0}\right)^{2\varepsilon}\otimes\left(\chi_{r}\mathcal{U}\left(t_{0}\right)\right)^{2\varepsilon}\right\rangle \right\rangle \xrightarrow{\varepsilon\downarrow0}0.
\end{align*}
\end{proof}
From now on, we write $\int_{0+}^{s}$ for $\lim_{\varepsilon\downarrow0}\int_{\varepsilon}^{s}$.
Then 
\[
\int_{I}\mathrm{d}t\;\eta\left(t\right)\left\langle \left\langle \mathcal{W}\left(t,s\right),\mathcal{U}\left(t\right)^{s}\right\rangle \right\rangle =\int_{I}\mathrm{d}t\;\eta\left(t\right)\int_{0+}^{s}\mathrm{d}\sigma\left\langle \left\langle \mathcal{N}\left(t,\sigma\right)^{3(s-\sigma)},\mathcal{U}\left(t\right)^{s}\right\rangle \right\rangle 
\]

\begin{defn}
Define the $k$-jet fiber norm $\left|X\right|_{J^{k}}=\left(\sum\limits _{j=0}^{k}\left|\nabla^{\left(j\right)}X\right|^{2}\right)^{\frac{1}{2}}\;\forall X\in\mathfrak{X}$.
\end{defn}

Let $K\left(\sigma,x,y\right)$ be the kernel of the heat flow at
time $\sigma>0$. Then by \Secref{Construction-of-heat-kernel}, we
obtain off-diagonal decay for all derivatives: 
\begin{fact}[Off-diagonal decay]
\label{fact:off-diagonal-decay} For any multi-index $\gamma$ and
$x\neq y$, $D_{\sigma,x,y}^{\gamma}K\left(\sigma,x,y\right)=O\left(\sigma^{\infty}\right)$
as $\sigma\downarrow0$, locally uniform in $\{x\neq y\}$.
\end{fact}

For convenience, we will write $\boxed{\mathcal{Y}=\chi_{r}\mathcal{U}}$.
Then for $r>0,\sigma\in\left(0,1\right)$ and $x\in M_{<r/4}:$ 
\begin{align}
\left|\mathcal{Y}^{\sigma}(t,x)\right|_{J^{2}} & \lesssim_{M,r}O_{r}\left(\sigma^{\infty}\right)\left\Vert \mathcal{U}(t)\right\Vert _{L^{3}\left(M_{>r/2}\right)}\label{eq:est_bd}
\end{align}
which implies $\left\Vert \mathcal{Y}^{\sigma}(t)\right\Vert _{W^{2,3}\left(M_{<r/4}\right)}+\left\Vert \mathcal{Y}^{\sigma}(t)\right\Vert _{W^{2,3}\mathfrak{X}M|_{\partial M}}\lesssim_{M,r}O_{r}\left(\sigma^{\infty}\right)\left\Vert \mathcal{U}(t)\right\Vert _{L^{3}\left(M_{>r/2}\right)}$.

We now handle the most important error estimates that will appear
in our analysis.
\begin{lem}[2 error estimates]
\label{lem:err_estimates} For fixed $r>0$ small, we have

\begin{equation}
\lim_{s\downarrow0}\int_{I}\left|\eta\right|\int_{0+}^{s}\mathrm{d}\sigma\int_{M}\left|\mathcal{U}^{2\sigma}\right|_{J^{1}}\left|\mathcal{Y}^{2\sigma}\right|_{J^{1}}\left|\mathcal{Y}^{4s-2\sigma}\right|_{J^{1}}=0\label{eq:1st_error_est}
\end{equation}
and 

\begin{equation}
\lim_{s\downarrow0}\int_{I}\left|\eta\right|\int_{0+}^{s}\mathrm{d}\sigma\int_{\partial M}\left|\mathcal{U}^{2\sigma}\right|_{J^{1}}\left|\mathcal{Y}^{2\sigma}\right|_{J^{1}}\left|\mathcal{Y}^{4s-2\sigma}\right|_{J^{2}}=0\label{eq:2nd_error_est}
\end{equation}
.
\end{lem}

\begin{proof}
We split (\ref{eq:1st_error_est}) into 2 regions: $M_{<r/4}$ and
$M_{\geq r/4}$. Observe that
\begin{align*}
 & \int_{I}\left|\eta\right|\int_{0+}^{s}\mathrm{d}\sigma\int_{M_{<r/4}}\left|\mathcal{U}^{2\sigma}\right|_{J^{1}}\left|\mathcal{Y}^{2\sigma}\right|_{J^{1}}\left|\mathcal{Y}^{4s-2\sigma}\right|_{J^{1}}\\
\lesssim & \int_{I}\left|\eta\right|\int_{0+}^{s}\mathrm{d}\sigma\;\left\Vert \mathcal{U}^{2\sigma}\right\Vert _{W^{1,3}\left(M_{<r/4}\right)}\left\Vert \mathcal{Y}^{2\sigma}\right\Vert _{W^{1,3}\left(M_{<r/4}\right)}\left\Vert \mathcal{Y}^{4s-2\sigma}\right\Vert _{W^{1,3}\left(M_{<r/4}\right)}\\
\lesssim & O_{r}\left(s^{\infty}\right)\int_{I}\mathrm{d}t\;\left|\eta\left(t\right)\right|\left\Vert \mathcal{U}(t)\right\Vert _{L^{3}\left(M\right)}^{3}\int_{0+}^{s}\mathrm{d}\sigma\;\left(\frac{1}{\sigma}\right)^{1/2}\xrightarrow{s\downarrow0}0.
\end{align*}

Define $B\left(t,s\right)=s^{\frac{1}{3}}\left\Vert \mathcal{U}\left(t\right)^{s}\right\Vert _{W^{1,3}\left(M_{\geq r/4}\right)}$
and $C\left(t,s\right)=s^{\frac{1}{3}}\left\Vert \mathcal{Y}\left(t\right)^{s}\right\Vert _{W^{1,3}\left(M_{\geq r/4}\right)}$. 

By \Factref{cutoff_interpol}, $\mathcal{Y}\in L_{t}^{3}\widehat{B}_{3,V}^{1/3}\mathfrak{X}$. 

Therefore, $\left\Vert B\left(t,s\right)\right\Vert _{L_{t}^{3}}$
and $\left\Vert C\left(t,s\right)\right\Vert _{L_{t}^{3}}$ are continuous
in $s$ and converge to $0$ as $s\to0$ by (\ref{eq:vanishing_property}).
Observe that
\begin{align*}
 & \int_{I}\left|\eta\right|\int_{0+}^{s}\mathrm{d}\sigma\int_{M_{\geq r/4}}\left|\mathcal{U}^{2\sigma}\right|_{J^{1}}\left|\mathcal{Y}^{2\sigma}\right|_{J^{1}}\left|\mathcal{Y}^{4s-2\sigma}\right|_{J^{1}}\\
\lesssim & \int_{I}\left|\eta\right|\int_{0+}^{s}\mathrm{d}\sigma\left\Vert \mathcal{U}^{2\sigma}\right\Vert _{W^{1,3}\left(M_{\geq r/4}\right)}\left\Vert \mathcal{Y}^{2\sigma}\right\Vert _{W^{1,3}\left(M_{\geq r/4}\right)}\left\Vert \mathcal{Y}^{4s-2\sigma}\right\Vert _{W^{1,3}\left(M_{\geq r/4}\right)}\\
= & \int_{I}\mathrm{d}t\left|\eta(t)\right|\int_{0+}^{s}\mathrm{d}\sigma\left(\frac{1}{\sigma}\right)^{\frac{2}{3}}\left(\frac{1}{2s-\sigma}\right)^{\frac{1}{3}}B\left(t,2\sigma\right)C\left(t,2\sigma\right)C\left(t,4s-2\sigma\right)\\
\lesssim_{\eta} & \int_{0+}^{s}\mathrm{d}\sigma\;\left(\frac{1}{\sigma}\right)^{\frac{2}{3}}\left(\frac{1}{2s-\sigma}\right)^{\frac{1}{3}}\left\Vert B\left(t,2\sigma\right)\right\Vert _{L_{t}^{3}}\left\Vert C\left(t,2\sigma\right)\right\Vert _{L_{t}^{3}}\left\Vert C\left(t,4s-2\sigma\right)\right\Vert _{L_{t}^{3}}\\
\stackrel{\sigma=s\tau}{=} & \int_{0+}^{1}\mathrm{d}\tau\;\left(\frac{1}{\tau}\right)^{\frac{2}{3}}\left(\frac{1}{2-\tau}\right)^{\frac{1}{3}}\left\Vert B\left(t,2s\tau\right)\right\Vert _{L_{t}^{3}}\left\Vert C\left(t,2s\tau\right)\right\Vert _{L_{t}^{3}}\left\Vert C\left(t,4s-2s\tau\right)\right\Vert _{L_{t}^{3}}\\
\xrightarrow[\mathrm{DCT}]{s\downarrow0} & \;0
\end{align*}
So (\ref{eq:1st_error_est}) is proven. For (\ref{eq:2nd_error_est}),
observe that
\begin{align*}
 & \int_{I}\left|\eta\right|\int_{0+}^{s}\mathrm{d}\sigma\int_{\partial M}\left|\mathcal{U}^{2\sigma}\right|_{J^{1}}\left|\mathcal{Y}^{2\sigma}\right|_{J^{1}}\left|\mathcal{Y}^{4s-2\sigma}\right|_{J^{2}}\\
\lesssim & O_{r}\left(s^{\infty}\right)\int_{I}\mathrm{d}t\;\left|\eta\left(t\right)\right|\left\Vert \mathcal{U}\left(t\right)\right\Vert _{L^{3}\left(M\right)}^{2}\int_{0+}^{s}\mathrm{d}\sigma\;\left\Vert \mathcal{U}\left(t\right)^{2\sigma}\right\Vert _{W^{1,3}\mathfrak{X}M|_{\partial M}}\\
\lesssim & O_{r}\left(s^{\infty}\right)\int_{I}\mathrm{d}t\;\left|\eta\left(t\right)\right|\left\Vert \mathcal{U}\left(t\right)\right\Vert _{L^{3}\left(M\right)}^{2}\int_{0+}^{s}\mathrm{d}\sigma\;\left\Vert \mathcal{U}\left(t\right)^{2\sigma}\right\Vert _{B_{3,1}^{1+1/3}\left(M\right)}\\
\lesssim & O_{r}\left(s^{\infty}\right)\int_{I}\mathrm{d}t\;\left|\eta\left(t\right)\right|\left\Vert \mathcal{U}\left(t\right)\right\Vert _{L^{3}\left(M\right)}^{3}\int_{0+}^{s}\mathrm{d}\sigma\;\left(\frac{1}{\sigma}\right)^{2/3}\\
 & \xrightarrow{s\downarrow0}0
\end{align*}
where we used (\ref{eq:est_bd}) to pass to the second line, and the
trace theorem to pass to the third line.
\end{proof}
Note that
\begin{align*}
\mathcal{N}\left(t,\sigma\right) & =\left(\partial_{\sigma}-3\Delta\right)\left(-\Div\left(\mathcal{U}^{2\sigma}\otimes\mathcal{Y}^{2\sigma}\right)^{\sigma}\right)=-2\Div\left(\Delta\mathcal{U}^{2\sigma}\otimes\mathcal{Y}^{2\sigma}\right)^{\sigma}-2\Div\left(\mathcal{U}^{2\sigma}\otimes\Delta\mathcal{Y}^{2\sigma}\right)^{\sigma}\\
 & +2\Delta\Div\left(\mathcal{U}^{2\sigma}\otimes\mathcal{Y}^{2\sigma}\right)^{\sigma}
\end{align*}

Now, we finally show
\[
\int_{I}\eta\left\langle \left\langle \mathcal{W}(s),\mathcal{Y}^{s}\right\rangle \right\rangle =\int_{I}\mathrm{d}t\;\eta\left(t\right)\left\langle \left\langle \mathcal{W}(t,s),\mathcal{Y}(t)^{s}\right\rangle \right\rangle \xrightarrow{s\downarrow0}0
\]

\begin{proof}[Proof of the commutator estimate]
 First we integrate by parts into three components:
\begin{align*}
 & \int_{I}\eta\left\langle \left\langle \mathcal{W}(s),\mathcal{Y}^{s}\right\rangle \right\rangle =\int_{I}\mathrm{d}t\;\eta\left(t\right)\int_{0+}^{s}\mathrm{d}\sigma\left\langle \left\langle \mathcal{N}\left(t,\sigma\right)^{3(s-\sigma)},\mathcal{Y}\left(t\right)^{s}\right\rangle \right\rangle \\
= & \int_{I}\mathrm{d}t\;\eta\left(t\right)\int_{0+}^{s}\mathrm{d}\sigma\left\langle \left\langle \mathcal{N}\left(t,\sigma\right),\mathcal{Y}\left(t\right)^{4s-3\sigma}\right\rangle \right\rangle \\
= & 2\int_{I}\eta\int_{0+}^{s}\mathrm{d}\sigma\left\langle \left\langle \Delta\mathcal{U}^{2\sigma}\otimes\mathcal{Y}^{2\sigma},\nabla\left(\mathcal{Y}^{4s-2\sigma}\right)\right\rangle \right\rangle +2\int_{I}\eta\int_{0+}^{s}\mathrm{d}\sigma\left\langle \left\langle \mathcal{U}^{2\sigma}\otimes\Delta\mathcal{Y}^{2\sigma},\nabla\left(\mathcal{Y}^{4s-2\sigma}\right)\right\rangle \right\rangle \\
 & \;-2\int_{I}\eta\int_{0+}^{s}\mathrm{d}\sigma\left\langle \left\langle \mathcal{U}^{2\sigma}\otimes\mathcal{Y}^{2\sigma},\nabla\left(\Delta\mathcal{Y}^{4s-2\sigma}\right)\right\rangle \right\rangle 
\end{align*}
Note that for the third component, we used \Factref{D_N-basic-properties}
to move the Laplacian. 

We now use the Penrose abstract index notation to estimate the three
components. To clean up the notation, we only focus on the integral
on $M$, with the other integrals $2\int_{I}\eta\int_{0+}^{s}\mathrm{d}\sigma\left(\cdot\right)$
in variables $t$ and $\sigma$ implicitly understood. We also use\textbf{
schematic identities} for linear combinations of similar-looking tensor
terms where we do not care how the indices contract (recall \Eqref{Weitzen_schematic}). 

By \Lemref{err_estimates}, it is easy to check that all the terms
with $R$ or $\nu$ will be negligible (going to $0$ in the limit),
and interchanging derivatives will be a negligible action. We write
$\approx$ to throw the negligible error terms away. 

First component:
\begin{align*}
 & \int_{M}\left\langle \Delta\mathcal{U}^{2\sigma}\otimes\mathcal{Y}^{2\sigma},\nabla\left(\mathcal{Y}^{4s-2\sigma}\right)\right\rangle =\cancel{\int_{M}R*\mathcal{U}^{2\sigma}*\mathcal{Y}^{2\sigma}*\nabla\left(\mathcal{Y}^{4s-2\sigma}\right)}+\int_{M}\begin{array}{c}
\nabla_{i}\nabla^{i}\left(\mathcal{U}^{2\sigma}\right)^{j}\left(\mathcal{Y}^{2\sigma}\right)^{l}\nabla_{j}\left(\mathcal{\mathcal{Y}}^{4s-2\sigma}\right)_{l}\end{array}\\
\approx & \cancel{\int_{\partial M}\nu_{i}\nabla^{i}\left(\mathcal{U}^{2\sigma}\right)^{j}\left(\mathcal{Y}^{2\sigma}\right)^{l}\nabla_{j}\left(\mathcal{\mathcal{Y}}^{4s-2\sigma}\right)_{l}}-\cancel{\int_{M}\begin{array}{c}
\nabla^{i}\left(\mathcal{U}^{2\sigma}\right)^{j}\nabla_{i}\left(\mathcal{Y}^{2\sigma}\right)^{l}\nabla_{j}\left(\mathcal{\mathcal{Y}}^{4s-2\sigma}\right)_{l}\end{array}}\\
 & -\int_{M}\begin{array}{c}
\nabla^{i}\left(\mathcal{U}^{2\sigma}\right)^{j}\left(\mathcal{Y}^{2\sigma}\right)^{l}\nabla_{i}\nabla_{j}\left(\mathcal{\mathcal{Y}}^{4s-2\sigma}\right)_{l}\end{array}
\end{align*}
Second component:
\begin{align*}
 & \int_{M}\left\langle \mathcal{U}^{2\sigma}\otimes\Delta\mathcal{\mathcal{Y}}^{2\sigma},\nabla\left(\mathcal{Y}^{4s-2\sigma}\right)\right\rangle =\cancel{\int_{M}\mathcal{U}^{2\sigma}*R*\mathcal{Y}^{2\sigma}*\nabla\left(\mathcal{\mathcal{Y}}^{4s-2\sigma}\right)}+\int_{M}\begin{array}{c}
\left(\mathcal{U}^{2\sigma}\right)^{j}\nabla_{i}\nabla^{i}\left(\mathcal{Y}^{2\sigma}\right)^{l}\nabla_{j}\left(\mathcal{\mathcal{Y}}^{4s-2\sigma}\right)_{l}\end{array}\\
\approx & \cancel{\int_{\partial M}\left(\mathcal{U}^{2\sigma}\right)^{j}\nu_{i}\nabla^{i}\left(\mathcal{Y}^{2\sigma}\right)^{l}\nabla_{j}\left(\mathcal{\mathcal{Y}}^{4s-2\sigma}\right)_{l}}-\cancel{\int_{M}\begin{array}{c}
\nabla_{i}\left(\mathcal{U}^{2\sigma}\right)^{j}\nabla^{i}\left(\mathcal{Y}^{2\sigma}\right)^{l}\nabla_{j}\left(\mathcal{\mathcal{Y}}^{4s-2\sigma}\right)_{l}\end{array}}\\
 & -\int_{M}\begin{array}{c}
\left(\mathcal{U}^{2\sigma}\right)^{j}\nabla^{i}\left(\mathcal{Y}^{2\sigma}\right)^{l}\nabla_{i}\nabla_{j}\left(\mathcal{\mathcal{Y}}^{4s-2\sigma}\right)_{l}\end{array}
\end{align*}
For the third component, we use the identity $\nabla\left(R*U\right)=\nabla R*U+R*\nabla U$
to compute:
\begin{align*}
 & -\int_{M}\left\langle \mathcal{U}^{2\sigma}\otimes\mathcal{\mathcal{Y}}^{2\sigma},\nabla\left(\Delta\mathcal{\mathcal{Y}}^{4s-2\sigma}\right)\right\rangle =-\cancel{\int_{M}\mathcal{U}^{2\sigma}*\mathcal{\mathcal{Y}}^{2\sigma}*\nabla\left(R*\mathcal{Y}^{4s-2\sigma}\right)}-\int_{M}\begin{array}{c}
\left(\mathcal{U}^{2\sigma}\right)^{j}\left(\mathcal{Y}^{2\sigma}\right)^{l}\nabla_{j}\nabla^{i}\nabla_{i}\left(\mathcal{\mathcal{Y}}^{4s-2\sigma}\right)_{l}\end{array}\\
\approx & \cancel{\int_{M}\mathcal{U}^{2\sigma}*\mathcal{\mathcal{Y}}^{2\sigma}*R*\nabla\left(\mathcal{\mathcal{Y}}^{4s-2\sigma}\right)}-\int_{M}\begin{array}{c}
\left(\mathcal{U}^{2\sigma}\right)^{j}\left(\mathcal{Y}^{2\sigma}\right)^{l}\nabla^{i}\nabla_{j}\nabla_{i}\left(\mathcal{\mathcal{Y}}^{4s-2\sigma}\right)_{l}\end{array}\\
\approx & \cancel{\int_{M}\mathcal{U}^{2\sigma}*\mathcal{\mathcal{Y}}^{2\sigma}*\nabla\left(R*\mathcal{\mathcal{Y}}^{4s-2\sigma}\right)}-\int_{M}\begin{array}{c}
\left(\mathcal{U}^{2\sigma}\right)^{j}\left(\mathcal{Y}^{2\sigma}\right)^{l}\nabla^{i}\nabla_{i}\nabla_{j}\left(\mathcal{\mathcal{Y}}^{4s-2\sigma}\right)_{l}\end{array}\\
\approx & -\cancel{\int_{\partial M}\left(\mathcal{U}^{2\sigma}\right)^{j}\left(\mathcal{Y}^{2\sigma}\right)^{l}\nu^{i}\nabla_{i}\nabla_{j}\left(\mathcal{\mathcal{Y}}^{4s-2\sigma}\right)_{l}}+\int_{M}\begin{array}{c}
\nabla^{i}\left(\mathcal{U}^{2\sigma}\right)^{j}\left(\mathcal{Y}^{2\sigma}\right)^{l}\nabla_{i}\nabla_{j}\left(\mathcal{\mathcal{Y}}^{4s-2\sigma}\right)_{l}\end{array}\\
 & +\int_{M}\begin{array}{c}
\left(\mathcal{U}^{2\sigma}\right)^{j}\nabla^{i}\left(\mathcal{Y}^{2\sigma}\right)^{l}\nabla_{i}\nabla_{j}\left(\mathcal{\mathcal{Y}}^{4s-2\sigma}\right)_{l}\end{array}
\end{align*}
By adding them up, we are done.
\end{proof}

\appendix

\section{Local analysis\label{sec:Local_analysis}}

Let $M$ be as in \Eqref{Euler}. Throughout this section, we write
$e^{t\Delta}$ for the absolute Neumann heat flow, as we will not
need another heat flow.

Assume the absolute Neumann heat kernel is already constructed, with
off-diagonal decay (\Factref{off-diagonal-decay}). 

As before, define $P_{\leq N}=e^{\frac{1}{N^{2}}\Delta}$ for $N>0$
and $P_{N}=P_{\leq N}-P_{\leq\frac{N}{2}}$ for $N>1,N\in2^{\mathbb{Z}}$.

Let $\chi_{r}=1-\psi_{r}$ (see \Eqref{cutoff}). 

Then we have the localized Bernstein estimates: 
\begin{thm}
\label{thm:local_bernstein1}For any $r>0$; $m_{1},m_{2}\in\mathbb{N}_{0}$;
$p\in\left(1,\infty\right);N\geq1$ and $X\in W^{m_{1},p}\Omega\left(M\right)$: 

$\left\Vert P_{\leq N}X\right\Vert _{W^{m_{1}+m_{2},p}\left(M_{\geq2r}\right)}\lesssim_{r,m_{1},m_{2},p}N^{m_{2}}\left\Vert X\right\Vert _{W^{m_{1},p}\left(M_{\geq r}\right)}+O_{r}\left(\frac{1}{N^{\infty}}\right)\left\Vert X\right\Vert _{L^{p}\left(M_{\leq3r}\right)}$
\end{thm}

\begin{proof}
Observe that $1-\chi_{2r}=\psi_{2r}=\psi_{2r}\psi_{4r}$. Then:
\begin{align*}
\left\Vert P_{\leq N}X\right\Vert _{W^{m_{1}+m_{2},p}\left(M_{\geq2r}\right)} & \lesssim\left\Vert P_{\leq N}\left(\chi_{2r}X\right)\right\Vert _{W^{m_{1}+m_{2},p}\left(M_{\geq2r}\right)}+\left\Vert P_{\leq N}\left(\psi_{2r}\psi_{4r}X\right)\right\Vert _{W^{m_{1}+m_{2},p}\left(M_{\geq2r}\right)}\\
 & \lesssim_{m_{1},m_{2}}N^{m_{2}}\left\Vert \chi_{2r}X\right\Vert _{W^{m_{1},p}\left(M\right)}+O_{r}\left(\frac{1}{N^{\infty}}\right)\left\Vert \psi_{4r}X\right\Vert _{L^{p}\left(M\right)}\\
 & \lesssim_{r}N^{m_{2}}\left\Vert X\right\Vert _{W^{m_{1},p}\left(M_{\geq r}\right)}+O_{r}\left(\frac{1}{N^{\infty}}\right)\left\Vert X\right\Vert _{L^{p}\left(M_{\leq3r}\right)}
\end{align*}
where we have used the standard Bernstein estimate (\Thmref{global_bernstein})
and the off-diagonal decay of the heat kernel to pass from the first
line to the second line ($\supp\,\psi_{2r}\subseteq M_{\leq\frac{3}{2}r}$
which does not intersect $M_{\geq2r}$).
\end{proof}
\begin{cor}
\label{cor:local_bernstein2}For any $r,C_{1},C_{2}>0$; $N\geq1;p\in\left(1,\infty\right)$
and $X\in\mathscr{D}'_{N}\Omega\left(M\right):$
\begin{align*}
\left\Vert \left(P_{\leq C_{1}N}-P_{\leq C_{2}N}\right)X\right\Vert _{L^{p}\left(M_{\geq2r}\right)} & \lesssim_{C_{1},C_{2},r,p}\frac{1}{N^{2}}\left\Vert P_{\leq2\max\left(C_{1},C_{2}\right)N}X\right\Vert _{W^{2,p}\left(M_{\geq r}\right)}+O_{C_{1},C_{2},r}\left(\frac{1}{N^{\infty}}\right)\left\Vert X\right\Vert _{L^{p}\left(M\right)}\\
 & \lesssim_{C_{1},C_{2},r,p}\frac{1}{N}\left\Vert P_{\leq3\max\left(C_{1},C_{2}\right)N}X\right\Vert _{W^{1,p}\left(M_{\geq r/2}\right)}+O_{C_{1},C_{2},r}\left(\frac{1}{N^{\infty}}\right)\left\Vert X\right\Vert _{L^{p}\left(M\right)}
\end{align*}
\end{cor}

\begin{proof}
WLOG $C_{1}>C_{2}>0$. Let $C=2\max\left(C_{1},C_{2}\right)$. Then
by FTC:
\begin{align*}
\left\Vert \left(P_{\leq C_{1}N}-P_{\leq C_{2}N}\right)X\right\Vert _{L^{p}\left(M_{\geq2r}\right)} & \leq\int_{\frac{1}{C_{1}^{2}N^{2}}-\frac{1}{C^{2}N^{2}}}^{\frac{1}{C_{2}^{2}N^{2}}-\frac{1}{C^{2}N^{2}}}\text{d}t\;\left\Vert e^{\left(t+\frac{1}{C^{2}N^{2}}\right)\Delta}X\right\Vert _{W^{2,p}\left(M_{\geq2r}\right)}\\
 & \lesssim_{C_{1},C_{2},r,p}\int_{\frac{1}{C_{1}^{2}N^{2}}-\frac{1}{C^{2}N^{2}}}^{\frac{1}{C_{2}^{2}N^{2}}-\frac{1}{C^{2}N^{2}}}\text{d}t\;\left(\left\Vert e^{\frac{\Delta}{C^{2}N^{2}}}X\right\Vert _{W^{2,p}\left(M_{\geq r}\right)}+O_{r}\left(t^{\infty}\right)\left\Vert e^{\frac{\Delta}{2N^{2}}}X\right\Vert _{L^{p}\left(M\right)}\right)\\
 & \lesssim_{C_{1},C_{2}}\frac{1}{N^{2}}\left\Vert P_{\leq CN}X\right\Vert _{W^{2,p}\left(M_{\geq r}\right)}+O_{C_{1},C_{2},r}\left(\frac{1}{N^{\infty}}\right)\left\Vert X\right\Vert _{L^{p}\left(M\right)}
\end{align*}
We have used \Thmref{local_bernstein1} to pass to the second line. 

The rest is trivial.
\end{proof}
\begin{cor}
\label{cor:equiv_cN}Let $p\in\left(1,\infty\right)$ and $X\in L^{p}\Omega\left(M\right)$.
Then the following conditions are equivalent:
\begin{enumerate}
\item $N^{\frac{1}{p}-1}\left\Vert P_{\leq N}X\right\Vert _{W^{1,p}\left(M_{>r}\right)}\xrightarrow{N\to\infty}0$
$\forall r>0$
\item $N^{\frac{1}{p}}\left\Vert \left(P_{\leq C_{1}N}-P_{\leq C_{2}N}\right)X\right\Vert _{L^{p}\left(M_{>r}\right)}\xrightarrow{N\to\infty}0$
$\forall r,C_{1},C_{2}>0$
\item $N^{\frac{1}{p}}\left\Vert P_{>N}X\right\Vert _{L^{p}\left(M_{>r}\right)}\xrightarrow{N\to\infty}0$
$\forall r>0$
\end{enumerate}
\end{cor}

\begin{proof}
It is trivial to show $(3)\implies(2)$ as $P_{\leq C_{1}N}-P_{\leq C_{2}N}$=$P_{>C_{2}N}-P_{>C_{1}N}$.

Next, we show $(2)\implies(3)$. Let 
\[
w(N)=N^{\frac{1}{p}}\left\Vert P_{N}X\right\Vert _{L^{p}\left(M_{>r}\right)}=N^{\frac{1}{p}}\left\Vert \left(P_{\leq N}-P_{\leq N/2}\right)X\right\Vert _{L^{p}\left(M_{>r}\right)}\xrightarrow{N\to\infty}0
\]
Then: 
\begin{align*}
N^{\frac{1}{p}}\left\Vert P_{>N}X\right\Vert _{L^{p}\left(M_{>r}\right)} & \leq N^{\frac{1}{p}}\sum_{\substack{K\in2^{\mathbb{Z}}\\
K>N
}
}\left\Vert P_{K}X\right\Vert _{L^{p}\left(M_{>r}\right)}=N^{\frac{1}{p}}\sum_{\substack{K\in2^{\mathbb{Z}}\\
K>N
}
}K^{-\frac{1}{p}}w(K)\\
 & \lesssim\left\Vert w(\kappa)\right\Vert _{l^{\infty}\left(\kappa>N,\kappa\in2^{\mathbb{Z}}\right)}\xrightarrow{N\to\infty}0
\end{align*}

We proceed to show $(1)\implies(2)$. By \Corref{local_bernstein2}:
\begin{align*}
N^{\frac{1}{p}}\left\Vert \left(P_{\leq C_{1}N}-P_{\leq C_{2}N}\right)X\right\Vert _{L^{p}\left(M_{>r}\right)} & \lesssim_{C_{1},C_{2}}N^{\frac{1}{p}-1}\left\Vert P_{\leq3\max\left(C_{1},C_{2}\right)N}X\right\Vert _{W^{1,p}\left(M_{\geq r/4}\right)}+O_{C_{1},C_{2},r}\left(\frac{1}{N^{\infty}}\right)\left\Vert X\right\Vert _{L^{p}\left(M\right)}\\
 & \xrightarrow{N\to\infty}0
\end{align*}

Finally, we show $(2)\implies(1)$. Let $N_{0}\geq1$ and $N_{0}\in2^{\mathbb{Z}}$.
There are constants $C_{1},C_{2}>0$ such that $P_{N}=P_{\leq2N}\left(P_{\leq C_{1}N}-P_{\leq C_{2}N}\right)$.
\begin{align}
 & \limsup_{N\to\infty}N^{\frac{1}{p}-1}\left\Vert P_{\leq N}X\right\Vert _{W^{1,p}\left(M_{>r}\right)}=\limsup_{N\to\infty}N^{\frac{1}{p}-1}\left\Vert \left(P_{\leq N}-P_{\leq N_{0}}\right)X\right\Vert _{W^{1,p}\left(M_{>r}\right)}\nonumber \\
\lesssim & \limsup_{N\to\infty}N^{\frac{1}{p}-1}\sum_{\substack{K\in2^{\mathbb{Z}}\\
N_{0}<K\leq N
}
}\left\Vert P_{K}X\right\Vert _{W^{1,p}\left(M_{>r}\right)}\nonumber \\
\lesssim & \limsup_{N\to\infty}N^{\frac{1}{p}-1}\sum_{\substack{K\in2^{\mathbb{Z}}\\
N_{0}<K\leq N
}
}\left(K\left\Vert \left(P_{\leq C_{1}K}-P_{\leq C_{2}K}\right)X\right\Vert _{L^{p}\left(M_{>r/2}\right)}+O_{r}\left(\frac{1}{K^{\infty}}\right)\left\Vert X\right\Vert _{L^{p}\left(M\right)}\right)\nonumber \\
\lesssim & \limsup_{N\to\infty}N^{\frac{1}{p}-1}\sum_{\substack{K\in2^{\mathbb{Z}}\\
N_{0}<K\leq N
}
}K^{1-1/p}w\left(K\right)+\underbrace{\limsup_{N\to\infty}N^{\frac{1}{p}-1}O_{r}\left(\frac{1}{N_{0}^{\infty}}\right)\left\Vert X\right\Vert _{L^{p}\left(M\right)}}_{0}\label{eq:temp_1}
\end{align}
where $w\left(K\right):=K^{1/p}\left\Vert \left(P_{\leq C_{1}K}-P_{\leq C_{2}K}\right)X\right\Vert _{L^{p}\left(M_{>r/2}\right)}\xrightarrow{K\to\infty}0$.
Then we can bound (\ref{eq:temp_1}) by 

\[
\limsup_{N\to\infty}\sum_{\substack{K\in2^{\mathbb{Z}}\\
N_{0}<K\leq N
}
}\left(\frac{K}{N}\right)^{1-1/p}\left\Vert w(\kappa)\right\Vert _{l^{\infty}\left(\kappa\geq N_{0},\kappa\in2^{\mathbb{Z}}\right)}\lesssim\left\Vert w(\kappa)\right\Vert _{l^{\infty}\left(\kappa\geq N_{0},\kappa\in2^{\mathbb{Z}}\right)}
\]
But $N_{0}$ is arbitrary. Let $N_{0}\to\infty$ and we are done. 
\end{proof}
\begin{rem}
By repeating the proof, for $\mathcal{X}\in L_{t}^{p}L^{p}\Omega\left(M\right):$
\begin{align*}
 & \forall r>0:N^{\frac{1}{p}-1}\left\Vert P_{\leq N}\mathcal{X}\right\Vert _{L_{t}^{p}W^{1,p}\left(M_{>r}\right)}\xrightarrow{N\to\infty}0\\
\iff & \forall r>0:N^{\frac{1}{p}}\left\Vert P_{>N}\mathcal{X}\right\Vert _{L_{t}^{p}L^{p}\left(M_{>r}\right)}\xrightarrow{N\to\infty}0
\end{align*}
\end{rem}

We now prove a simple lemma from functional analysis.
\begin{lem}[Loss of norm]
\label{lem:func_anal_lemma} Let $X,Y$ be Banach spaces and $T:X\hookrightarrow Y$
is continuous injection. Let $\left(f_{j}\right)_{j\in\mathbb{N}_{1}}$be
a sequence in $X$ and $f\in X$. If $Tf_{j}\rightharpoonup Tf$ then
\[
\left\Vert f\right\Vert _{X}\leq\liminf_{j\to\infty}\left\Vert f_{j}\right\Vert _{X}
\]
\end{lem}

\begin{proof}
Note that $T^{*}:Y^{*}\to X^{*}$has dense image. Then 
\begin{align*}
\left\Vert f\right\Vert _{X} & =\sup_{\substack{\left\Vert x^{*}\right\Vert _{X^{*}}=1\\
x^{*}\in X^{*}
}
}\left|\left\langle f,x^{*}\right\rangle \right|=\sup_{\substack{y^{*}\in Y^{*}\\
\left\Vert T^{*}y^{*}\right\Vert _{X^{*}}=1
}
}\left|\left\langle f,T^{*}y^{*}\right\rangle \right|=\sup_{\substack{y^{*}\in Y^{*}\\
\left\Vert T^{*}y^{*}\right\Vert _{X^{*}}=1
}
}\lim_{j\to\infty}\left|\left\langle Tf_{j},y^{*}\right\rangle \right|\\
 & =\sup_{\substack{y^{*}\in Y^{*}\\
\left\Vert T^{*}y^{*}\right\Vert _{X^{*}}=1
}
}\lim_{j\to\infty}\left|\left\langle f_{j},T^{*}y^{*}\right\rangle \right|\leq\sup_{\substack{y^{*}\in Y^{*}\\
\left\Vert T^{*}y^{*}\right\Vert _{X^{*}}=1
}
}\liminf_{j\to\infty}\left\Vert f_{j}\right\Vert _{X}=\liminf_{j\to\infty}\left\Vert f_{j}\right\Vert _{X}
\end{align*}
\end{proof}
\begin{thm}
Let $p\in\left(1,\infty\right)$, $f\in\mathscr{D}_{N}\left(M\right)$,
and $X\in\widehat{B}_{p,V}^{1/p}\mathfrak{X}\left(M\right)$ (as in
\Defref{besov_v_space}). Then $fX\in\widehat{B}_{p,V}^{1/p}\mathfrak{X}$.
\end{thm}

\begin{proof}
To show $fX\in\widehat{B}_{p,V}^{1/p}\mathfrak{X}$, we just need
to show a commutator estimate (much like in the proof of Onsager's
conjecture): 
\[
\left\{ \begin{array}{rl}
\mathcal{W}(s) & :=f^{s}X^{s}-(fX)^{s}\\
\left(\sqrt{s}\right)^{1-\frac{1}{p}}\mathcal{W}(s) & \xrightarrow[s\downarrow0]{W^{1,p}(M)}0
\end{array}\right.
\]
where $X^{s}$ is short for $e^{s\Delta}X$. Indeed, assuming this
commutator estimate holds true, $\forall r>0$:
\begin{align*}
 & \limsup_{t\downarrow0}\left(\sqrt{t}\right)^{1-\frac{1}{p}}\left\Vert e^{t\Delta}\left(fX\right)\right\Vert _{W^{1,p}\left(M_{>r}\right)}\\
\leq & \limsup_{t\downarrow0}\left(\sqrt{t}\right)^{1-\frac{1}{p}}\left\Vert f^{t}X^{t}\right\Vert _{W^{1,p}\left(M_{>r}\right)}+\underbrace{\limsup_{t\downarrow0}\left(\sqrt{t}\right)^{1-\frac{1}{p}}\left\Vert \mathcal{W}(t)\right\Vert _{W^{1,p}\left(M_{>r}\right)}}_{0}\\
\lesssim & \limsup_{t\downarrow0}\left(\sqrt{t}\right)^{1-\frac{1}{p}}\left\Vert f^{t}\right\Vert _{C^{1}\left(M\right)}\left\Vert X^{t}\right\Vert _{W^{1,p}\left(M_{>r}\right)}=0
\end{align*}
where we have used the fact that $e^{t\Delta}f\xrightarrow[t\to0]{C^{\infty}}f$,
as $f\in\mathscr{D}_{N}\left(M\right)$.

Now we prove the commutator estimate. Define $\mathcal{N}(s)=\left(\partial_{s}-\Delta\right)\mathcal{W}(s)=\left(\Delta f^{s}\right)X^{s}+f^{s}\left(\Delta X^{s}\right)-\Delta\left(f^{s}X^{s}\right)$.
By the Weitzenbock formula, we get 
\[
\mathcal{N}(s)=\left(D^{1}f^{s}\right)*\left(D^{1}X^{s}\right)
\]
 where $D^{1}$ is schematic for some differential operator of order
at most 1, with smooth coefficients (independent of $s$), and $\left(D^{1}f^{s}\right)*\left(D^{1}X^{s}\right)$
is schematic for a linear combination of similar-looking tensor terms.

On the other hand, by the Duhamel formula for semigroups (cf. \parencite[Appendix A, Proposition 9.10 \& 9.11]{taylorPartialDifferentialEquations2011a}),
for any $s>\varepsilon>0$ we get 
\[
\mathcal{W}(s)=\mathcal{W}\left(\varepsilon\right)^{s-\varepsilon}+\int_{\varepsilon}^{s}\mathcal{N}\left(\sigma\right)^{s-\sigma}\mathrm{d}\sigma
\]
It is trivial to show that $\mathcal{W}\left(\varepsilon\right)^{s-\varepsilon}\xrightharpoonup[\varepsilon\downarrow0]{L^{p}}0$.
Indeed, let $Y\in L^{p'}\mathfrak{X}\left(M\right).$ Then 
\[
\left\langle \left\langle \mathcal{W}\left(\varepsilon\right)^{s-\varepsilon},Y\right\rangle \right\rangle =\left\langle \left\langle f^{\varepsilon}X^{\varepsilon}-(fX)^{\varepsilon},Y^{s-\varepsilon}\right\rangle \right\rangle \xrightarrow{\varepsilon\downarrow0}\left\langle \left\langle fX-fX,Y^{s}\right\rangle \right\rangle =0
\]
Then $\int_{\varepsilon}^{s}\mathcal{N}\left(\sigma\right)^{s-\sigma}\mathrm{d}\sigma\xrightharpoonup[\varepsilon\downarrow0]{L^{p}}\mathcal{W}(s)$,
and by \Lemref{func_anal_lemma}, we conclude
\begin{align*}
\left\Vert \mathcal{W}(s)\right\Vert _{W^{1,p}(M)} & \leq\liminf_{\varepsilon\downarrow0}\left\Vert \int_{\varepsilon}^{s}\mathcal{N}\left(\sigma\right)^{s-\sigma}\mathrm{d}\sigma\right\Vert _{W^{1,p}(M)}\leq\int_{0}^{s}\left\Vert e^{\left(s-\sigma\right)\Delta}\left(D^{1}f^{\sigma}*D^{1}X^{\sigma}\right)\right\Vert _{W^{1,p}(M)}\mathrm{d}\sigma\\
 & \lesssim\int_{0}^{s}\left(\frac{1}{s-\sigma}\right)^{\frac{1}{2}}\left\Vert D^{1}f^{\sigma}*D^{1}X^{\sigma}\right\Vert _{L^{p}(M)}\mathrm{d}\sigma\lesssim_{f}\int_{0}^{s}\left(\frac{1}{s-\sigma}\right)^{\frac{1}{2}}\left\Vert X^{\sigma}\right\Vert _{W^{1,p}(M)}\mathrm{d}\sigma\\
 & \lesssim\left\Vert X\right\Vert _{L^{p}(M)}\int_{0}^{s}\left(\frac{1}{s-\sigma}\right)^{\frac{1}{2}}\left(\frac{1}{\sigma}\right)^{\frac{1}{2}}\;\mathrm{d}\sigma\\
 & \stackrel{\sigma=s\tau}{=}\left\Vert X\right\Vert _{L^{p}(M)}\int_{0}^{1}\left(\frac{1}{1-\tau}\right)^{\frac{1}{2}}\left(\frac{1}{\tau}\right)^{\frac{1}{2}}\;\mathrm{d}\tau\lesssim_{p}\left\Vert X\right\Vert _{L^{p}(M)}
\end{align*}
This obviously implies $\left(\sqrt{s}\right)^{1-\frac{1}{p}}\mathcal{W}(s)\xrightarrow[s\downarrow0]{W^{1,p}(M)}0$.
\end{proof}
\begin{rem*}
By repeating the proof, with necessary modifications, for any $f\in\mathscr{D}_{N}\left(M\right)$,
and $\mathcal{X}\in L_{t}^{p}\widehat{B}_{p,V}^{1/p}\mathfrak{X}\left(M\right)$
(as in \Defref{besov_v_space}), we have: 
\[
f\mathcal{X}\in L_{t}^{p}\widehat{B}_{p,V}^{1/p}\mathfrak{X}
\]
\end{rem*}

\subsection{On flat backgrounds}
\begin{rem}
When $M$ is a bounded domain in $\mathbb{R}^{n}$, the third condition
in \Corref{equiv_cN} takes on a more familiar form. Indeed, let $\phi\in C_{c}^{\infty}\left(\mathbb{R}^{n}\right)$
with $\int\phi=1$ and $\phi_{\varepsilon}=\frac{1}{\varepsilon^{n}}\phi\left(\frac{\cdot}{\varepsilon}\right)$.
Then we have the analogy
\[
P_{\leq\frac{1}{\sqrt{t}}}f=e^{t\Delta}f\approx\phi_{\sqrt{t}}*f
\]
This means 
\begin{equation}
N^{\frac{1}{p}}\left\Vert P_{>N}X\right\Vert _{L^{p}\left(M_{>r}\right)}\xrightarrow{N\to\infty}0\label{eq:P_>N_eqn}
\end{equation}
 is analogous to 
\begin{equation}
\frac{1}{\varepsilon^{1/p}}\left\Vert X-\phi_{\varepsilon}*X\right\Vert _{L^{p}\left(M_{>r}\right)}\xrightarrow{\varepsilon\to0}0\label{eq:convolution_eqn}
\end{equation}
\end{rem}

\begin{defn}
As in \parencite{bardosOnsagerConjectureBounded2019,nguyenEnergyConservationInhomogeneous2020},
for $p\in\left(1,\infty\right)$, we say $X\in\underline{B}_{p,\text{VMO}}^{1/p}\mathfrak{X}\left(M\right)$
if $X\in L^{p}\mathfrak{X}\left(M\right)$ and $\forall r>0:$
\begin{equation}
A_{r}\left(\varepsilon\right):=\frac{1}{\varepsilon^{1/p}}\left\Vert \left\Vert X(x-\varepsilon h)-X(x)\right\Vert _{L_{|h|\leq1}^{p}}\right\Vert _{L_{x}^{p}\left(M_{>r}\right)}\xrightarrow{\varepsilon\downarrow0}0\label{eq:vmo_eqn}
\end{equation}
Similarly, we say $\mathcal{X}\in L_{t}^{p}\underline{B}_{p,\text{VMO}}^{1/p}\mathfrak{X}\left(M\right)$
if $X\in L_{t}^{p}L^{p}\mathfrak{X}\left(M\right)$ and $\forall r>0:$
\begin{equation}
A_{r}\left(\varepsilon\right):=\frac{1}{\varepsilon^{1/p}}\left\Vert \left\Vert \mathcal{X}(t,x-\varepsilon h)-\mathcal{X}(t,x)\right\Vert _{L_{|h|\leq1}^{p}}\right\Vert _{L_{t}^{p}L_{x}^{p}\left(M_{>r}\right)}\xrightarrow{\varepsilon\downarrow0}0\label{eq:vmo_eqn-1}
\end{equation}
\end{defn}

\begin{rem}
In (\ref{eq:vmo_eqn}), note that $A_{r}\left(\varepsilon\right)$
is continuous for $\varepsilon\in[0,r)$. Define 
\[
\widetilde{A_{r}}(\varepsilon):=\frac{1}{\varepsilon^{1/p}}\left\Vert \left\Vert \mathbf{1}_{M_{>r}}\left(x-\varepsilon h\right)\left(X(x-\varepsilon h)-X(x)\right)\right\Vert _{L_{|h|\leq1}^{p}}\right\Vert _{L_{x}^{p}\left(M_{>r}\right)}
\]
 for $\varepsilon\in\left(0,1\right]$ (well-defined). Then $\widetilde{A_{r}}\left(\varepsilon\right)$
is also continuous in $\varepsilon$, with $\widetilde{A_{r}}\left(\varepsilon\right)\leq A_{r}\left(\varepsilon\right)\;\forall\varepsilon\in(0,r)$
and $\widetilde{A_{r}}\left(\varepsilon\right)\lesssim_{r,p}\left\Vert X\right\Vert _{L^{p}\left(M\right)}\;\forall\varepsilon\in[\frac{r}{2},1]$.
By \parencite[Subsection 5.2]{huynhHodgetheoreticAnalysisManifolds2019},
we conclude 
\[
\left\Vert X\right\Vert _{B_{p,\infty}^{1,p}\left(M_{>r}\right)}\sim\left\Vert X\right\Vert _{L^{p}\left(M_{>r}\right)}+\left\Vert \widetilde{A_{r}}\left(\varepsilon\right)\right\Vert _{L_{\varepsilon}^{\infty}\left(\left(0,1\right)\right)}\lesssim_{r,p}\left\Vert X\right\Vert _{L^{p}\left(M\right)}+\left\Vert A_{r}\left(\varepsilon\right)\right\Vert _{L_{\varepsilon}^{\infty}\left([0,r/2]\right)}
\]

From this we conclude $\underline{B}_{p,\text{VMO}}^{1/p}\hookrightarrow B_{p,\infty,\mathrm{loc}}^{1/p}$
and $L_{t}^{p}\underline{B}_{p,\text{VMO}}^{1/p}\hookrightarrow L_{t}^{p}B_{p,\infty,\mathrm{loc}}^{1/p}$
where 
\[
B_{p,\infty,\mathrm{loc}}^{1/p}\left(M\right):=L^{p}\left(M\right)\cap\left(\bigcap_{r>0}B_{p,\infty}^{1/p}\left(M_{>r}\right)\right)
\]
and $L_{t}^{p}B_{p,\infty,\mathrm{loc}}^{1/p}\left(M\right):=L_{t}^{p}L^{p}\left(M\right)\cap\left(\bigcap_{r>0}L_{t}^{p}B_{p,\infty}^{1/p}\left(M_{>r}\right)\right)$.

We observe that (\ref{eq:vmo_eqn}) trivially implies (\ref{eq:convolution_eqn}).
To relate (\ref{eq:vmo_eqn}) to (\ref{eq:P_>N_eqn}), we now borrow
some results from the construction of the heat kernel (to be proven
in \Secref{Construction-of-heat-kernel}).
\end{rem}

\begin{fact}
Fix $r>0$. Let $K(t,x,y)$ be the Hodge-Neumann heat kernel as constructed
in \Secref{Construction-of-heat-kernel}.

For $r'>0,$ let $E_{r'}=\left\{ \left(x,y\right)\in M\times M:d\left(x,y\right)\geq r'\right\} $.
Then $E_{r'}$ is compact, and by the locally uniform off-diagonal
decay of the heat kernel, we conclude 
\begin{equation}
\forall x,y\in E_{r'},\forall t\leq1:\left|K(t,x,y)\right|=O_{r',\neg x,\neg y}\left(t^{\infty}\right)\label{eq:off-decay-r}
\end{equation}

Now let $F_{r,r'}=\left\{ \left(x,y\right)\in M_{\leq r}\times M:d\left(x,y\right)\leq r'\right\} $.
Then $F_{r,r'}$ is compact. By interior blow-up, there is $r'=r'\left(r\right)\in\left(0,\frac{r}{4}\right)$
such that 
\begin{equation}
\forall x,y\in F_{r,r'},\forall t\leq1:\left|K(t,x,y)\right|=O_{r,\neg x,\neg y}\left(\frac{1}{t^{n/2}}\left\langle \frac{x-y}{\sqrt{t}}\right\rangle ^{-\infty}\right)\label{eq:int-blow-r}
\end{equation}
\end{fact}

\begin{thm}
\label{thm:contain_VMO}Let $M$ be a bounded $C^{\infty}$-domain
in $\mathbb{R}^{n}$, $p\in\left(1,\infty\right)$ and $X\in L^{p}\mathfrak{X}\left(M\right)$.
Then
\begin{equation}
\forall r>0:A_{r}\left(\varepsilon\right):=\frac{1}{\varepsilon^{1/p}}\left\Vert \left\Vert X(x-\varepsilon h)-X(x)\right\Vert _{L_{|h|\leq1}^{p}}\right\Vert _{L_{x}^{p}\left(M_{>r}\right)}\xrightarrow{\varepsilon\downarrow0}0\label{eq:titi_vmo_ref}
\end{equation}
is equivalent to
\begin{equation}
\forall r>0:N^{\frac{1}{p}}\left\Vert P_{>N}X\right\Vert _{L^{p}\left(M_{>r}\right)}\xrightarrow{N\to\infty}0\label{eq:mine}
\end{equation}
\end{thm}

\begin{rem*}
The proof actually shows for $N\geq1:$ 
\[
N^{\frac{1}{p}}\left\Vert P_{>N}X\right\Vert _{L^{p}\left(M_{>r}\right)}=O_{r}\left(\left\Vert A_{r}\right\Vert _{L^{\infty}\left([0,\frac{r}{2}]\right)}+\frac{\left\Vert X\right\Vert _{L^{p}\left(M\right)}}{N^{\infty}}\right)
\]
\end{rem*}
\begin{proof}
We first show (\ref{eq:titi_vmo_ref}) implies (\ref{eq:mine}). Fix
$r>0$. Let $r'=r'\left(r\right)\in\left(0,\frac{r}{4}\right)$ as
in (\ref{eq:int-blow-r}). By (\ref{eq:off-decay-r}), we can disregard
the region $\left\{ d(x,y)>r'\right\} $, and just need to show 
\[
\left(\frac{1}{\sqrt{t}}\right)^{\frac{1}{p}}\left\Vert \int_{d(y,x)\leq r'}\text{d}y\;K(t,x,y)\left(X(y)-X(x)\right)\right\Vert _{L_{x}^{p}\left(M_{>r}\right)}\xrightarrow{t\to0}0
\]
But by (\ref{eq:int-blow-r}), the left-hand side is bounded by 
\begin{align}
 & \left(\frac{1}{\sqrt{t}}\right)^{\frac{1}{p}}\left\Vert \left\Vert O_{r}\left(\frac{1}{t^{n/2}}\left\langle \frac{x-y}{\sqrt{t}}\right\rangle ^{-\infty}\right)\left|X(y)-X(x)\right|\right\Vert _{L_{y}^{1}\left(B_{r'}\left(x\right)\right)}\right\Vert _{L_{x}^{p}\left(M_{>r}\right)}\nonumber \\
\lesssim_{r} & \left(\frac{1}{\sqrt{t}}\right)^{\frac{1}{p}}\left\Vert \left\Vert \left\langle \zeta\right\rangle ^{-\infty}\left|X(x-\sqrt{t}\zeta)-X(x)\right|\right\Vert _{L_{\left|\zeta\right|\leq\frac{r'}{\sqrt{t}}}^{1}}\right\Vert _{L_{x}^{p}\left(M_{>r}\right)}\label{eq:b4_dyadic-rings}
\end{align}
where we made the change of variable $\zeta=\frac{x-y}{\sqrt{t}}$.
By (\ref{eq:titi_vmo_ref}) and Holder's inequality, we can disregard
the region $\left\{ \left|\zeta\right|\leq1\right\} $. Then we split
$1<\left|\zeta\right|\leq\frac{r'}{\sqrt{t}}$ into dyadic rings:
\begin{align*}
\left(\ref{eq:b4_dyadic-rings}\right) & \lesssim\left(\frac{1}{\sqrt{t}}\right)^{\frac{1}{p}}\sum_{\substack{N\in2^{\mathbb{N}_{0}},N\lesssim\frac{r'}{\sqrt{t}}}
}\frac{1}{N^{\infty}}\left\Vert \left\Vert X(x-\sqrt{t}\zeta)-X(x)\right\Vert _{L_{\left|\zeta\right|\sim N}^{1}}\right\Vert _{L_{x}^{p}\left(M_{>r}\right)}\\
 & \lesssim\left(\frac{1}{\sqrt{t}}\right)^{\frac{1}{p}}\sum_{\substack{N\in2^{\mathbb{N}_{0}},N\lesssim\frac{r'}{\sqrt{t}}}
}\frac{1}{N^{\infty}}\left\Vert \left\Vert X(x-\sqrt{t}\zeta)-X(x)\right\Vert _{L_{\left|\zeta\right|\sim N}^{p}}\right\Vert _{L_{x}^{p}\left(M_{>r}\right)}\\
 & \lesssim\left(\frac{1}{\sqrt{t}}\right)^{\frac{n+1}{p}}\sum_{\substack{N\in2^{\mathbb{N}_{0}},N\lesssim\frac{r'}{\sqrt{t}}}
}\frac{1}{N^{\infty}}\left\Vert \left\Vert X(x-\tau)-X(x)\right\Vert _{L_{\left|\tau\right|\sim\sqrt{t}N}^{p}}\right\Vert _{L_{x}^{p}\left(M_{>r}\right)}
\end{align*}
where we made the change of variable $\tau=\sqrt{t}\zeta$. Now observe
that (\ref{eq:titi_vmo_ref}) implies that for $\varepsilon\leq r/2$:
\[
\left\Vert \left\Vert X\left(x-\tau\right)-X(x)\right\Vert _{L_{|\tau|\leq\varepsilon}^{p}}\right\Vert _{L_{x}^{p}\left(M_{>r}\right)}=\varepsilon^{\frac{n+1}{p}}A_{r}\left(\varepsilon\right)
\]
where $0\leq A_{r}(\varepsilon)\leq\left\Vert A_{r}\right\Vert _{L^{\infty}\left([0,r/2]\right)}$
and $A_{r}\left(\varepsilon\right)\xrightarrow{\varepsilon\downarrow0}0$.
Then 
\begin{align*}
\left(\ref{eq:b4_dyadic-rings}\right) & \lesssim\left(\frac{1}{\sqrt{t}}\right)^{\frac{n+1}{p}}\sum_{\substack{N\in2^{\mathbb{N}_{0}},N\lesssim\frac{r'}{\sqrt{t}}}
}\frac{1}{N^{\infty}}\left(\sqrt{t}N\right)^{\frac{n+1}{p}}A_{r}\left(\sqrt{t}N\right)\\
 & \lesssim\sum_{\substack{N\in2^{\mathbb{N}_{0}},N\lesssim\frac{r'}{\sqrt{t}}}
}\frac{1}{N^{\infty}}A_{r}\left(\sqrt{t}N\right)\\
 & \xrightarrow[\text{DCT}]{t\downarrow0}0
\end{align*}
Now we show (\ref{eq:mine}) implies (\ref{eq:titi_vmo_ref}). Observe
that by \Corref{equiv_cN}, (\ref{eq:mine}) is equivalent to 
\[
N^{\frac{1}{p}-1}\left\Vert P_{\leq N}X\right\Vert _{W^{1,p}\left(M_{>r}\right)}\xrightarrow{N\to\infty}0\;\forall r>0
\]
Now fix $r>0$. Then for $\varepsilon\in\left(0,\min\left(1,\frac{r}{2}\right)\right)$,
define $N=\frac{1}{\varepsilon}$, and we have: 
\begin{align*}
A_{r}\left(\varepsilon\right) & \leq N^{\frac{1}{p}}\left\Vert \left\Vert P_{\leq N}X\left(x-\frac{1}{N}h\right)-P_{\leq N}X\left(x\right)\right\Vert _{L_{|h|\leq1}^{p}}\right\Vert _{L_{x}^{p}\left(M_{>r}\right)}+N^{\frac{1}{p}}\left\Vert \left\Vert P_{>N}X\left(x-\frac{1}{N}h\right)\right\Vert _{L_{|h|\leq1}^{p}}\right\Vert _{L_{x}^{p}\left(M_{>r}\right)}\\
 & \;\;\;\;\;+N^{\frac{1}{p}}\left\Vert \left\Vert P_{>N}X\left(x\right)\right\Vert _{L_{|h|\leq1}^{p}}\right\Vert _{L_{x}^{p}\left(M_{>r}\right)}\\
 & \lesssim N^{\frac{1}{p}-1}\left\Vert \left\Vert \left\Vert \nabla P_{\leq N}X\left(x-\tau\frac{1}{N}h\right)\right\Vert _{L_{\tau}^{1}\left(\left[0,1\right]\right)}\right\Vert _{L_{|h|\leq1}^{p}}\right\Vert _{L_{x}^{p}\left(M_{>r}\right)}+N^{\frac{1}{p}}\left\Vert P_{>N}X\right\Vert _{L^{p}\left(M_{>\frac{r}{2}}\right)}\\
 & \lesssim N^{\frac{1}{p}-1}\left\Vert P_{\leq N}X\right\Vert _{W^{1,p}\left(M_{>\frac{r}{2}}\right)}+N^{\frac{1}{p}}\left\Vert P_{>N}X\right\Vert _{L^{p}\left(M_{>\frac{r}{2}}\right)}\xrightarrow{\varepsilon\to0}0
\end{align*}
Note that we used Minkowski's inequality, in passing to the last line.
\end{proof}

\section{Construction of the heat kernel\label{sec:Construction-of-heat-kernel}}

Recall the Japanese bracket notation $\left\langle a\right\rangle =\sqrt{1+|a|^{2}}\sim1+|a|$.
We also write $a=O\left(b^{\infty}\right)$ or $|a|\lesssim b^{\infty}$
to mean $|a|\lesssim_{l}b^{l}\;\forall l\in\mathbb{N}$.

Let $\left(M,g\right)$ be a compact Riemannian $n$-manifold with
boundary. A differential $k$-form is a member of $C^{\infty}\left(M;\Lambda^{k}M\right)$.

In this section, unless otherwise noted, we write $\Delta$ for the
Hodge Laplacian on forms. We also let $\left(t,x,y\right)$ be the
standard local coordinates for $[0,\infty)\times M\times M$. When
$x$ or $y$ is near the boundary, we can stipulate that $x_{n}$
and $y_{n}$ stand for the Riemannian distance to the boundary (\textbf{geodesic
normal coordinates}). 

We aim to construct a unique Hodge-Neumann heat kernel with the absolute
Neumann boundary condition. In particular, define $\mathrm{END}\left(\Lambda^{k}M\right)=\mathrm{Hom}\left(\pi_{2}^{*}\Lambda^{k}M,\pi_{1}^{*}\Lambda^{k}M\right)$,
where $\pi_{i}$ is the projection from $\left(0,\infty\right)\times M\times M$
onto the $i$-th $M$. We want
\[
K\in C_{\mathrm{loc}}^{\infty}\left(\left(0,\infty\right)\times M\times M;\mathrm{END}\left(\Lambda^{k}M\right)\right)
\]
such that 
\begin{align*}
\left(\partial_{t}-\Delta_{x}\right)K\left(t,x,y\right) & =0\\
\mathbf{n}_{x}K\left(t,x,y\right) & =0 &  & \text{for }x\in\partial M\\
\mathbf{n}_{x}d_{x}K\left(t,x,y\right) & =0 &  & \text{for }x\in\partial M\\
\lim_{t\downarrow0}K\left(t,x,y\right) & =\delta_{y}\left(x\right)
\end{align*}
where the last condition means $\forall u\in\mathscr{D}\left(M;\Lambda^{k}M\right),\int K\left(t,x,y\right)u(y)\;\mathrm{d}y\xrightarrow{t\downarrow0}u(x)$.

During the construction, we will be able to prove certain properties
of the kernel, such as off-diagonal decay for all derivatives. 

The construction of the heat kernel comes from \parencite{mazzeoAnalyticTorsionManifolds2013},
and we simply discuss the modifications required for our case, to
handle the Hodge-Neumann Laplacian on a smooth manifold with smooth
boundary.\footnote{The author thanks Daniel Grieser, András Vasy and Rafe Mazzeo for
discussing these ideas.

The original plan was to follow the note \parencite{grieserNOTESHEATKERNEL2004}
which is simpler and does not rely on Melrose's calculus, but we have
decided to clean up the note, modify some steps and publish it at
a later date.}

\subsubsection*{Kernel in Einstein sum notation\label{subsec:Kernel-in-Einstein}}

Let $A\in C_{\mathrm{loc}}^{\infty}\left(\left(0,\infty\right)\times M^{2};\mathrm{END}\left(\Lambda^{k}M\right)\right)$.
Let $U\subset M$ be a coordinate patch. Then, by using Einstein notation,
locally for $x,y\in U$ we have:
\[
A\left(t,x,y\right)=A_{I}\text{}^{J}\left(t,x,y\right)dx^{I}\otimes\partial_{y^{J}}
\]
where $I,J\in\mathcal{I}_{k}=\{\left(i_{1},...,i_{k}\right):i_{1}<i_{2}<...<i_{k}\}$
and $\partial_{y^{J}}$ is dual to the form $dy^{J}$. (also in Einstein
notation, we write $x^{n}$ instead of $x_{n}$)
\begin{itemize}
\item Note that we are abusing notation, as $dx^{I}$ here is a local section
of $\pi_{1}^{*}\Lambda^{k}M\twoheadrightarrow\left(0,\infty\right)\times M^{2}$,
defined by pulling back the actual form $dx^{I}$ on $M$. We can
explicitly write $A_{I}\text{}^{J}\left(t,x,y\right)\left.dx^{I}\right|_{x}\otimes\left.\partial_{y^{J}}\right|_{y}$
to emphasize the pullback.
\item Observe that $d_{x}A\left(t,x,y\right)=d_{x}\left(A_{I}\text{}^{J}\left(t,x,y\right)dx^{I}\right)\otimes\partial_{y^{J}}=\partial_{x^{i}}A_{I}\text{}^{J}\left(t,x,y\right)\left(dx^{i}\wedge dx^{I}\right)\otimes\partial_{y^{J}}$.
\item If $u(y)=u_{J}(y)dy^{J}$ is a differential form on $M$, we write
$A\left(t,x,y\right)u(y)=A_{I}\text{}^{J}\left(t,x,y\right)u_{J}(y)dx^{I}$,
which is a section of $\pi_{1}^{*}\Lambda^{k}M$.
\end{itemize}
As agreed above, when $U$ touches the boundary, $\partial_{x^{n}}$
is the inwards normal direction, so for $x\in\partial M$: $\mathbf{n}_{x}A\left(t,x,y\right)=1_{n\in I}A_{I}\text{}^{J}\left(t,x,y\right)dx^{I}\otimes\partial_{y^{J}}$.
\begin{itemize}
\item If $\mathbf{n}_{x}A=0$ for all $x\in\partial M$, then 
\[
\mathbf{n}_{x}d_{x}A=1_{n\notin I}\partial_{x^{n}}A_{I}\text{}^{J}\left(t,x,y\right)\left(dx^{n}\wedge dx^{I}\right)\otimes\partial_{y^{J}}
\]
 So $\mathbf{n}_{x}d_{x}A=0\iff\partial_{x^{n}}A_{I}{}^{J}\left(t,x,y\right)=0$
whenever $x\in\partial M,n\notin I$. In other words, $\mathbf{n}_{x}A=0$
and $\mathbf{n}_{x}d_{x}A=0$ mean the normal part obeys the Dirichlet
boundary condition, while the tangential part obeys the Neumann boundary
condition. This will inspire the choice of leading terms later on.
\end{itemize}

\subsection{Heat calculus}

Let $x=\left(x',x_{n}\right)$ and $y=\left(y',y_{n}\right)$ be points
in $\mathbb{R}^{n}$. Recall:
\begin{enumerate}
\item The scalar heat kernel on $\mathbb{R}^{n}$: $K\left(t,x,y\right)=\left(\frac{1}{4\pi}\right)^{n/2}\tau^{-n}e^{-\frac{|\zeta|^{2}}{4}}$
where $\tau=\sqrt{t},\zeta=\frac{x-y}{\tau}$.
\item The Dirichlet scalar heat kernel on $\mathbb{R}^{n-1}\times[0,\infty)$:
$K\left(t,x,y\right)=\left(\frac{1}{4\pi}\right)^{n/2}\tau^{-n}e^{-\frac{|\zeta'|^{2}}{4}}\left(e^{-\frac{1}{4}|\xi_{n}-\eta_{n}|^{2}}-e^{-\frac{1}{4}|\xi_{n}+\eta_{n}|^{2}}\right)$
where $\xi_{n}=\frac{x_{n}}{\tau},\eta_{n}=\frac{y_{n}}{\tau},\zeta'=\frac{x'-y'}{\tau}$.
\item The Neumann scalar heat kernel on $\mathbb{R}^{n-1}\times[0,\infty)$:
$K\left(t,x,y\right)=\left(\frac{1}{4\pi}\right)^{n/2}\tau^{-n}e^{-\frac{|\zeta'|^{2}}{4}}\left(e^{-\frac{1}{4}|\xi_{n}-\eta_{n}|^{2}}+e^{-\frac{1}{4}|\xi_{n}+\eta_{n}|^{2}}\right)$.
\end{enumerate}
They will inspire the formulation of our boundary heat calculus, which
describes heat-type kernels on manifolds.

We assume the reader is familiar with the spaces of conormal and polyhomogeneous
distributions on a manifold with corners \parencite{grieserBasicsBCalculus2001,melroseAtiyahPatodiSingerIndexTheorem2018}.

\subsubsection{Blown-up heat space}

We first construct the blown-up heat space $M_{h}^{2}$, with the
faces lf, ff, td, tf as defined in \parencite{mazzeoAnalyticTorsionManifolds2013}
(though our case is simpler). 

We start with $[0,\infty)\times M\times M$, with faces tf (temporal
face), rf (right face), lf (left face) being defined as $\left\{ 0\right\} \times M\times M$,
$[0,\infty)\times\partial M\times M$, $[0,\infty)\times M\times\partial M$
respectively. Then we perform a parabolic blow-up \parencite[Section 7.4]{melroseAtiyahPatodiSingerIndexTheorem2018}
on the submanifold $\left\{ 0\right\} \times\partial M\times\partial M$
in the time direction $dt$, to create the face ff (front face)\footnote{We are following \parencite{mazzeoAnalyticTorsionManifolds2013} by
letting rf be defined by $x_{n}=0$. Other authors might prefer $y_{n}=0$.}. This creates an intermediate manifold that we will call $M_{1}$. 

After that, we perform another parabolic blow-up on the lift of the
submanifold $\left\{ 0\right\} \times\Delta\left(M\right)$ to $M_{1}$
(to be more precisely defined in (\ref{eq:lift})), which creates
another face td (time diagonal). This is the space $M_{h}^{2}$ we
need.
\begin{figure}[th]
\centering{}\includegraphics[width=0.4\textwidth]{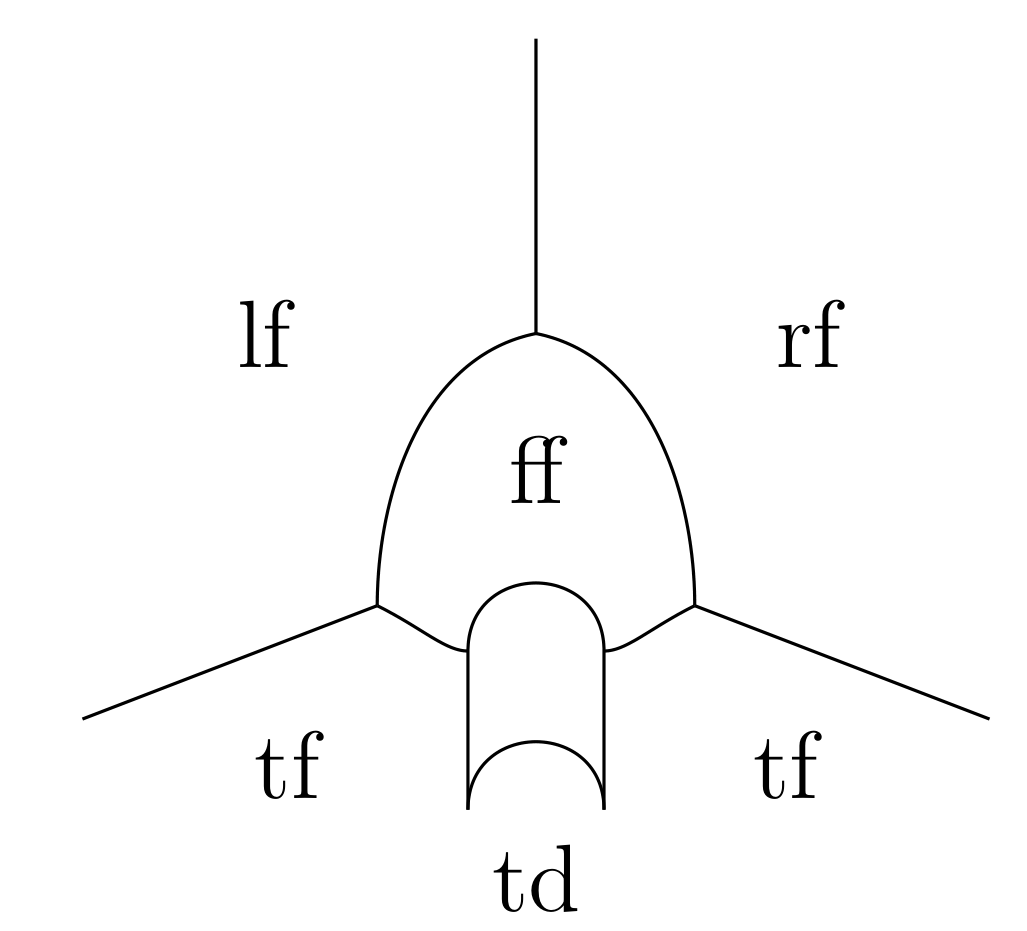}\caption{The blown-up heat space $M_{h}^{2}$}
\end{figure}

\subsubsection{Local coordinates}

By letting $\tau=\sqrt{t}$, we call $\left(\tau,x,y\right)$ the
\textbf{ts-coordinate system} (time-rescaled) for $[0,\infty)\times M\times M$.
\begin{itemize}
\item On $M_{1}$, near rf and away from lf (i.e. $y_{n}>0$), we use the\textbf{
rf-coordinate system}
\begin{equation}
T=\frac{t}{y_{n}^{2}},\theta'=\frac{x'-y'}{y_{n}},\theta_{n}=\frac{x_{n}}{y_{n}},y',y_{n}\label{eq:rf_coord}
\end{equation}
where $\theta_{n},y_{n},T$ are respectively the boundary defining
functions for rf, ff, tf. For blow-ups, it is also useful to define
the (time-rescaled) \textbf{tsrf-coordinate system }
\begin{equation}
\varsigma=\sqrt{T},\theta',\theta_{n},y',y_{n}\label{eq:tsrf_coor}
\end{equation}
We observe that as $\left(\varsigma,\theta',\theta_{n},y',y_{n}\right)\to\left(\varsigma,\theta',\theta_{n},y',0\right)$
in the tsrf-coordinate, in the ts-coordinate we have 
\[
\left(0,\left(y',0\right),\left(y',0\right)\right)+y_{n}\left(\varsigma,\left(\theta',\theta_{n}\right),\left(0,1\right)\right)\to\left(0,\left(y',0\right),\left(y',0\right)\right)
\]
The (time-rescaled) tangent vector $\left(\varsigma,\left(\theta',\theta_{n}\right),\left(0,1\right)\right)$\footnote{Explicitly, the tangent vector is $\varsigma\partial_{\tau}+\left(\theta',\theta_{n}\right)\cdot\partial_{x}+\left(0,1\right)\cdot\partial_{y}$.}
at $\left(0,\left(y',0\right),\left(y',0\right)\right)$ (modulo vectors
tangent to $\left\{ 0\right\} \times\partial M\times\partial M$,
and modulo positive scalar multiplication) corresponds to a point
on ff, which is $\left[\left(\varsigma,\left(\theta',\theta_{n}\right),\left(0,1\right)\right)\right]=\left[\left(\varsigma,\left(0,\theta_{n}\right),\left(-\theta',1\right)\right)\right]$.
This is what allows us to extend the (ts)rf-coordinate systems from
$[0,\infty)\times M\times M$ to $M_{1}$, with $\{y_{n}=0\}$ being
the face ff.
\item On $M_{1}$, near ff and away from tf, we use the \textbf{ff-coordinate
system }
\begin{equation}
\tau=\sqrt{t},x',\xi_{n}=\frac{x_{n}}{\sqrt{t}},\zeta'=\frac{x'-y'}{\sqrt{t}},\eta_{n}=\frac{y_{n}}{\sqrt{t}}\label{eq:ff_coor}
\end{equation}
where $\tau,\xi_{n},\eta_{n}$ are respectively the boundary defining
functions for ff, rf, lf. As $\left(\tau,x',\xi_{n},\zeta',\eta_{n}\right)\to\left(0,x',\xi_{n},\zeta',\eta_{n}\right)$
in the ff-coordinate, in the ts-coordinate we have
\[
\left(0,\left(x',0\right),\left(x',0\right)\right)+\tau\left(1,\left(0,\xi_{n}\right),\left(-\zeta',\eta_{n}\right)\right)\to\left(0,\left(x',0\right),\left(x',0\right)\right)
\]
The (time-rescaled) tangent vector $\left(1,\left(0,\xi_{n}\right),\left(-\zeta',\eta_{n}\right)\right)$
at $\left(0,\left(x',0\right),\left(x',0\right)\right)$ corresponds
to a point on ff, which is $\left[\left(1,\left(0,\xi_{n}\right),\left(-\zeta',\eta_{n}\right)\right)\right]$.
\item On $M_{h}^{2}$, near td, near ff, away from lf, away from tf, we
can use the rf-coordinate system from (\ref{eq:rf_coord}) to define
the \textbf{fftd-coordinate system }
\begin{equation}
\vartheta=\sqrt{T},\sigma'=\frac{\theta'}{\sqrt{T}},\sigma_{n}=\frac{\theta_{n}-1}{\sqrt{T}},y',y_{n}\label{eq:fftd_coord}
\end{equation}
where $\vartheta$ is the defining function for td. Note that as $\left(\vartheta,\sigma',\sigma_{n},y',y_{n}\right)\to\left(0,\sigma',\sigma_{n},y',y_{n}\right)$
in the fftd-coordinate, in the tsrf-coordinate we have 
\begin{equation}
\left(0,0,1,y',y_{n}\right)+\vartheta\left(1,\sigma',\sigma_{n},0,0\right)\to\left(0,0,1,y',y_{n}\right)\label{eq:lift}
\end{equation}
We observe that the points $\left(0,0,1,y',y_{n}\right)$ in the tsrf-coordinate,
are precisely the lift of the submanifold $D_{0}:=\left\{ 0\right\} \times\Delta\left(M\right)$
to $M_{1}$, which we will write as $D_{1}$. By blowing up $D_{1}$,
we create the face td and $M_{h}^{2}$. Note that $\theta_{n}=1>0$,
so \uline{td does not intersect rf} (or lf). Also, the (time-rescaled)
tangent vector $\left(1,\sigma',\sigma_{n},0,0\right)$ at $\left(0,0,1,y',y_{n}\right)$
corresponds to a point on the face td.\\
On the other hand, the point $\left(\vartheta,\sigma',\sigma_{n},y',0\right)$
in the fftd-coordinate on $M_{h}^{2}$ maps down to the point $\left(\vartheta,\vartheta\sigma',\vartheta\sigma_{n}+1,y',0\right)$
in the tsrf-coordinate on $M_{1}$ (the map being injective on $\left\{ \vartheta>0\right\} $),
which in turn corresponds to the point $\left[\left(\vartheta,\left(0,\vartheta\sigma_{n}+1\right),\left(-\vartheta\sigma',1\right)\right)\right]$
on ff.
\item The points $\left(0,0,1,y',0\right)$ in the (ts)rf-coordinate are
precisely the intersection $\mathrm{ff}\cap D_{1}$ in $M_{1}$. \\
The points $\left(0,\sigma',\sigma_{n},y',0\right)$ in the fftd-coordinate
are precisely the intersection $\mathrm{ff\cap td}$ in $M_{h}^{2}$.
\item On $M_{h}^{2}$, near td, away from ff and away from tf, we use the
\textbf{td-coordinate system}
\begin{equation}
\tau=\sqrt{t},x,\zeta=\frac{x-y}{\sqrt{t}}\label{eq:td_coor}
\end{equation}
where $\tau$ is the defining function for td. As $\left(\tau,x,\zeta\right)\to\left(0,x,\zeta\right)$
in td-coordinate, in ts-coordinate we have 
\[
\left(0,x,x\right)+\tau\left(1,0,-\zeta\right)\to\left(0,x,x\right)
\]
So we identify the point $\left(0,x,\zeta\right)$ in td-coordinate
with the (time-rescaled) tangent vector $\left(1,0,-\zeta\right)$
at $\left(0,x,x\right)\in D_{0}$, which gives a point of td (or to
be precise, away from the edges, $D_{1}$ and $D_{0}$ are locally
diffeomorphic, and td being defined as a bundle over $D_{1}$ is also
locally defined over $D_{0}$).\\
\item Wherever we have both the td-coordinate system and the fftd-coordinate
system, the point $\left(\tau,x,\zeta\right)=\left(\tau,\left(x',x_{n}\right),\left(\zeta',\zeta_{n}\right)\right)$
in the td-coordinate (with $x_{n}>0,x_{n}-\tau\zeta_{n}>0$) corresponds
to the point $\left(\frac{\tau}{x_{n}-\tau\zeta_{n}},\zeta',\zeta_{n},x'-\tau\zeta',x_{n}-\tau\zeta_{n}\right)$
in the fftd-coordinate. Conversely, $\left(\vartheta,\sigma',\sigma_{n},y',y_{n}\right)$
in the fftd-coordinate corresponds to $\left(\vartheta y_{n},\left(y'+\vartheta y_{n}\sigma',y_{n}+\vartheta y_{n}\sigma_{n}\right),\left(\sigma',\sigma_{n}\right)\right)$
in the td-coordinate. Consequently, 
\begin{equation}
\left(0,\left(x',x_{n}\right),\left(\zeta',\zeta_{n}\right)\right)\text{ in td-coordinate corresponds to }\left(0,\zeta',\zeta_{n},x',x_{n}\right)\text{ in fftd-coordinate}\label{eq:correspond}
\end{equation}
 and we identify the tangent vector $\left(1,0,-\zeta\right)$ at
$\left(0,x,x\right)\in D_{0}$ (in the ts-coordinate) with the tangent
vector $\left(1,\zeta',\zeta_{n},0,0\right)$ at $\left(0,0,1,x',x_{n}\right)\in D_{1}$
(in the tsrf-coordinate), as the same point in td.
\end{itemize}
\begin{rem}[Compatibility condition at $\mathrm{ff\cap td}$]
\label{rem:compatibility} For any smooth functions $u$ on ff and
$v$ on td, the following are equivalent:
\end{rem}

\begin{enumerate}
\item In the fftd-coordinate:
\begin{equation}
u\left(\vartheta,\sigma',\sigma_{n},y',0\right)\xrightarrow{\vartheta\to0}v\left(0,\sigma',\sigma_{n},y',0\right)\label{eq:compatibility}
\end{equation}
\item There is a smooth function $f$ on $M_{h}^{2}$ such that $N_{\mathrm{ff}}^{0}\left(f\right)=u$,
$N_{\mathrm{td}}^{0}\left(f\right)=v$.
\end{enumerate}

\subsubsection{Edge calculus}
\begin{defn}
For $\alpha,\alpha'\in-\mathbb{N}_{0}$, we define $\Psi_{\mathrm{e-h}}^{\alpha,\alpha',E_{\mathrm{lf}},E_{\mathrm{rf}}}\left(M;\Lambda^{k}M\right)$\footnote{To translate to the definition of $\Psi_{\mathrm{e-h}}^{l,p,E_{\mathrm{lf}},E_{\mathrm{rf}}}$
from \parencite[Section 3.2]{mazzeoAnalyticTorsionManifolds2013},
we can use the formulas $\alpha=-l,\alpha'=-p-2,n=m,n-1=b$.} as the space of Schwartz kernels $K$ that are pushforwards of polyhomogeneous
kernels $\widetilde{K}$ on $M_{h}^{2}$ (though we will abuse notation
and also write $K$ for $\widetilde{K}$) such that:
\begin{itemize}
\item the index sets at lf and rf are $E_{\mathrm{lf}}=\left(E_{\mathrm{lf}}^{\mathbf{t}},E_{\mathrm{lf}}^{\mathbf{n}}\right)$
and $E_{\mathrm{rf}}=\left(E_{\mathrm{rf}}^{\mathbf{t}},E_{\mathrm{rf}}^{\mathbf{n}}\right)$.
Here $E_{\mathrm{lf}}^{\mathbf{t}},E_{\mathrm{rf}}^{\mathbf{t}}$
describe the local coefficients of $\mathbf{t}_{x}K$ (the tangent
component), while $E_{\mathrm{lf}}^{\mathbf{n}},E_{\mathrm{rf}}^{\mathbf{n}}$
describe the local coefficients of $\mathbf{n}_{x}K$.
\item the index set at ff is $\text{\ensuremath{\left\{  \left(j-\left(n+2+\alpha\right),0\right):j\in\mathbb{N}_{0}\right\} } }$
(expansion in $\tau$ from (\ref{eq:ff_coor}))
\item the index set at td is $\text{\ensuremath{\left\{  \left(j-\left(n+2+\alpha'\right),0\right):j\in\mathbb{N}_{0}\right\} } }$(expansion
in $\tau$ from (\ref{eq:td_coor})). By convention, it is $\emptyset$
when $\alpha'=-\infty$.
\item the index set at tf is $\emptyset$ (off-diagonal decay).%
\begin{comment}
\begin{itemize}
\item in the polyhomogeneous expansion at td in the interior projective
coordinate system $\left(\tau,x,\zeta\right)$:
\[
n+\left(\text{power of }\tau\right)+\left(\text{power of }\zeta\right)\equiv0\text{ mod }2
\]
\item in the polyhomogeneous expansion at ff in the boundary projective
coordinate system $\left(\tau,x',\xi_{n},\zeta',\eta_{n}\right)$:
\[
n+\left(\text{power of }\tau\right)+\left(\text{power of }\zeta'\right)\equiv0\text{ mod }2
\]
\end{itemize}
\end{comment}
\end{itemize}
\end{defn}

\begin{thm}
\label{thm:heat_calc}The absolute Neumann heat kernel $H$ lies in
$\Psi_{\mathrm{e-h}}^{-2,-2,E_{\mathrm{lf}},E_{\mathrm{rf}}}\left(M;\Lambda^{k}M\right)$
where
\begin{itemize}
\item $E_{\mathrm{lf}}^{\mathbf{t}},E_{\mathrm{lf}}^{\mathbf{n}}=\mathbb{N}_{0}\times\{0\}$
(smoothness at $\mathrm{lf}$)\footnote{In fact, due to symmetry, we must have $E_{\mathrm{lf}}=E_{\mathrm{rf}}$,
but for this paper we will not need this fact.}
\item $E_{\mathrm{rf}}^{\mathbf{n}}\subset\mathbb{N}_{1}\times\{0\}$, $E_{\mathrm{rf}}^{\mathbf{t}}\subset\left(\mathbb{N}_{0}\backslash\left\{ 1\right\} \right)\times\{0\}$
(absolute Neumann boundary condition)
\end{itemize}
We can also write $\Psi_{\mathrm{e-h}}^{-2,-2,\mathbb{N}_{0},E_{\mathrm{rf}}}$
to describe smoothness at $\mathrm{lf}$.
\end{thm}

\subsection{Proof of Theorem \ref{thm:heat_calc}}

We proceed exactly as in \parencite[Section 3.2]{mazzeoAnalyticTorsionManifolds2013}. 

For any $A\in\Psi_{\mathrm{e-h}}^{\alpha,\alpha',E_{\mathrm{lf}},E_{\mathrm{rf}}}\left(M;\Lambda^{k}M\right)$,
we can expand w.r.t. ff (with coordinates as in (\ref{eq:ff_coor}))
\[
A=A_{-n-2-\alpha}^{\mathrm{ff}}\left(x,\xi_{n},\zeta',\eta_{n}\right)\tau^{-n-2-\alpha}+A_{-n-2-\alpha+1}^{\mathrm{ff}}\left(x,\xi_{n},\zeta',\eta_{n}\right)\tau^{-n-2-\alpha+1}+....
\]
We write $N_{\mathrm{ff}}^{-n-2-\alpha}\left(A\right)$ for the leading
coefficient $A_{-n-2-\alpha}^{\mathrm{ff}}$. We can expand similarly
w.r.t. td and define $N_{\mathrm{td}}^{-n-2-\alpha'}\left(A\right)$.

Then we note that $t\left(\partial_{t}-\Delta_{x}\right)$ is a $b$-operator
which could be restricted to ff and td. In particular, 
\[
\left\{ \begin{array}{rl}
N_{\mathrm{ff}}^{-n-2-\alpha}\left(t\left(\partial_{t}-\Delta_{x}\right)A\right) & =N_{\mathrm{ff}}^{-n-2-\alpha}\left(t\left(\partial_{t}-\Delta_{x}\right)\right)N_{\mathrm{ff}}^{-n-2-\alpha}\left(A\right)\\
N_{\mathrm{td}}^{-n-2-\alpha'}\left(t\left(\partial_{t}-\Delta_{x}\right)A\right) & =N_{\mathrm{td}}^{-n-2-\alpha'}\left(t\left(\partial_{t}-\Delta_{x}\right)\right)N_{\mathrm{td}}^{-n-2-\alpha'}\left(A\right)
\end{array}\right.
\]

where, in the td-coordinate system from (\ref{eq:td_coor}) and the
ff-coordinate system from (\ref{eq:ff_coor}):
\begin{equation}
\left\{ \begin{array}{rl}
N_{\mathrm{td}}^{-n-2-\alpha}\left(t\left(\partial_{t}-\Delta_{x}\right)\right) & =-\Delta_{\zeta}\left(x\right)-\frac{1}{2}\zeta\cdot\partial_{\zeta}-\frac{n+2+\alpha}{2}\\
N_{\mathrm{ff}}^{-n-2-\alpha}\left(t\left(\partial_{t}-\Delta_{x}\right)\right) & =-\Delta_{\left(\zeta',\xi_{n}\right)}\left(x',0\right)-\frac{1}{2}\left(\zeta',\xi_{n},\eta_{n}\right)\cdot\partial_{\left(\zeta',\xi_{n},\eta_{n}\right)}-\frac{n+2+\alpha'}{2}
\end{array}\right.\label{eq:leading_PDE_op}
\end{equation}
Here we have written $\zeta\cdot\partial_{\zeta}=\sum_{i}\zeta_{i}\partial_{\zeta_{i}}$
and $\Delta_{\zeta}\left(x\right)=\sum_{i,j}g^{ij}(x)\partial_{\zeta_{i}}\partial_{\zeta{}_{j}}$.

Then we have $t\left(\partial_{t}-\Delta_{x}\right)\Psi_{\mathrm{e-h}}^{\alpha,\alpha',E_{\mathrm{lf}},E_{\mathrm{rf}}}\subseteq\Psi_{\mathrm{e-h}}^{\alpha,\alpha',\mathbb{N}_{0},\mathbb{N}_{0}}$.

From this point on, \uline{we fix \mbox{$E_{\mathrm{lf}},E_{\mathrm{rf}}$}
to be as in \mbox{\Thmref{heat_calc}}. }
\begin{claim}
\label{claim:1st_claim}There is an element $H^{\left(1\right)}\in\Psi_{\mathrm{e-h}}^{-2,-2,E_{\mathrm{lf}},E_{\mathrm{rf}}}\left(M;\Lambda^{k}M\right)$
such that 
\[
\left\{ \begin{array}{rl}
P^{\left(1\right)} & :=t\left(\partial_{t}-\Delta_{x}\right)H^{\left(1\right)}\in\Psi_{\mathrm{e-h}}^{-3,-\infty,\mathbb{N}_{0},\mathbb{N}_{0}}\\
\lim_{t\downarrow0}H^{\left(1\right)}\left(t,x,y\right) & =\delta_{y}\left(x\right)
\end{array}\right.
\]
\end{claim}

\begin{proof}
To prove this claim, we construct $A\in\Psi_{\mathrm{e-h}}^{-2,-2,E_{\mathrm{lf}},E_{\mathrm{rf}}}$
such that
\begin{align}
N_{\mathrm{td}}^{-n}\left(A\right)\left(x,\zeta\right) & =\left(\frac{1}{4\pi}\right)^{n/2}e^{-\frac{|\zeta|_{g\left(x\right)}^{2}}{4}}\mathrm{Id}=\left(\frac{1}{4\pi}\right)^{n/2}e^{-\frac{|\zeta|_{g\left(x\right)}^{2}}{4}}\left.dx^{I}\right|_{x}\otimes\left.\partial_{y^{I}}\right|_{x}\label{eq:int_leading}\\
N_{\mathrm{ff}}^{-n}\left(A\right)\left(x',\xi_{n},\zeta',\eta_{n}\right) & =\left(\frac{1}{4\pi}\right)^{n/2}e^{-\frac{|\zeta'|_{g\left(x',0\right)}^{2}}{4}}\left(e^{-\frac{1}{4}|\xi_{n}-\eta_{n}|^{2}}\left(\mathbf{t}+\mathbf{n}\right)+e^{-\frac{1}{4}|\xi_{n}+\eta_{n}|^{2}}\left(\mathbf{t}-\mathbf{n}\right)\right)\nonumber \\
 & =1_{n\notin I}\left(\frac{1}{4\pi}\right)^{n/2}e^{-\frac{|\zeta'|_{g\left(x',0\right)}^{2}}{4}}\left(e^{-\frac{1}{4}|\xi_{n}-\eta_{n}|^{2}}+e^{-\frac{1}{4}|\xi_{n}+\eta_{n}|^{2}}\right)\left.dx^{I}\right|_{\left(x',0\right)}\otimes\left.\partial_{y^{I}}\right|_{\left(x',0\right)}\nonumber \\
 & \;\;\;\;\;\;\;\;\;+1_{n\in I}\left(\frac{1}{4\pi}\right)^{n/2}e^{-\frac{|\zeta'|_{g\left(x',0\right)}^{2}}{4}}\left(e^{-\frac{1}{4}|\xi_{n}-\eta_{n}|^{2}}-e^{-\frac{1}{4}|\xi_{n}+\eta_{n}|^{2}}\right)\left.dx^{I}\right|_{\left(x',0\right)}\otimes\left.\partial_{y^{I}}\right|_{\left(x',0\right)}\nonumber 
\end{align}

This choice satisfies the compatibility condition from (\ref{eq:compatibility})
(with $N_{\mathrm{td}}^{-n}\left(A\right)=N_{\mathrm{td}}^{0}\left(t^{\frac{n}{2}}A\right)$
and $N_{\mathrm{ff}}^{-n}\left(A\right)=N_{\mathrm{ff}}^{0}\left(t^{\frac{n}{2}}A\right)$),
since 
\[
e^{-\frac{1}{4}|\sigma'|_{g\left(y',0\right)}^{2}}\left(e^{-\frac{1}{4}|\sigma_{n}|^{2}}\left(\mathbf{t}+\mathbf{n}\right)+e^{-\frac{1}{4}|\sigma_{n}+\frac{2}{\vartheta}|^{2}}\left(\mathbf{t}-\mathbf{n}\right)\right)\xrightarrow{\vartheta\to0}e^{-\frac{1}{4}|\sigma'|_{g\left(y',0\right)}^{2}}\left(e^{-\frac{1}{4}|\sigma_{n}|^{2}}\right)=e^{-\frac{1}{4}|\left(\sigma',\sigma_{n}\right)|_{g\left(y',0\right)}^{2}}
\]

We note that $A$ is smooth on $\left(0,\infty\right)\times M\times M$,
and we can make $A$ have the same index set for rf as $N_{\mathrm{ff}}^{-n}\left(A\right)$,
therefore satisfying the absolute Neumann condition. Off-diagonal
decay is also explicit from these formulas (when $x\neq y$ stay fixed
and $t\to0$, we have $\zeta=\frac{x-y}{\sqrt{t}}\to\infty$).

By direct calculations, we see that $N_{\mathrm{ff}}^{-n}\left(t\left(\partial_{t}-\Delta_{x}\right)A\right)=0$
and $N_{\mathrm{td}}^{-n}\left(t\left(\partial_{t}-\Delta_{x}\right)A\right)=0$.
Therefore $t\left(\partial_{t}-\Delta_{x}\right)A\in\Psi_{\mathrm{e-h}}^{-3,-3,\mathbb{N}_{0},\mathbb{N}_{0}}$.
We then observe two facts:
\begin{itemize}
\item In the expansion of $A$ at td, $A_{j}^{\mathrm{td}}$ for $j>-n$
can be freely changed. 
\item For any smooth $f\left(x,\zeta\right)$ that is Schwartz in $\zeta$
(rapidly decaying) and $j\geq1$, there is a unique $F\left(x,\zeta\right)$
rapidly decaying in $\zeta$ such that 
\[
N_{\mathrm{td}}^{-n+j}\left(t\left(\partial_{t}-\Delta_{x}\right)\right)F\left(x,\zeta\right)=f\left(x,\zeta\right)
\]
In particular, by using the Fourier transform $\zeta\mapsto z$ (with
the convention $\widehat{F}(z)=\int_{\mathbb{R}^{n}}F(\zeta)e^{-i2\pi\zeta\cdot z}\;\mathrm{d}\zeta$):
\[
\widehat{F}\left(x,z\right)=\int_{0}^{1}\mathrm{d}s\;2s^{j-1}\hat{f}\left(x,sz\right)e^{-\left(4\pi^{2}\right)\left(1-s^{2}\right)\left|z\right|_{g(x)}^{2}}
\]
See also \parencite[Section 6.2]{albinLinearAnalysisManifolds2017}
for an explanation of this. It boils down to the fact that $\Delta_{\zeta}$
is smoothing (elliptic) for $\zeta$. 
\end{itemize}
Therefore it is possible to change $\left(A_{j}^{\mathrm{td}}\right)_{j>-n}$to
make $t\left(\partial_{t}-\Delta_{x}\right)A$ vanish to infinite
order at td. It boils down to solving
\[
N_{\mathrm{td}}^{j}\left(t\left(\partial_{t}-\Delta_{x}\right)\right)A_{j}^{\mathrm{td}}=B_{j},\;j>-n
\]
where $B_{j}$ is an inhomogeneous term depending on $A_{-n}^{\mathrm{td}},...,A_{j-1}^{\mathrm{td}}$.
Changing $\left(A_{j}^{\mathrm{td}}\right)_{j>-n}$ will not affect
the index set of $A$ at rf, since td does not intersect rf and lf,
by the above reasoning with (\ref{eq:fftd_coord}). $A$ is smooth
at rf and lf, and we therefore obtain $t\left(\partial_{t}-\Delta_{x}\right)A\in\Psi_{\mathrm{e-h}}^{-3,-\infty,\mathbb{N}_{0},\mathbb{N}_{0}}$. 

We finally note that $\lim_{t\downarrow0}A\left(t,x,y\right)=\delta_{y}\left(x\right)$
due to (\ref{eq:int_leading}), which is the ``universal'' formula
for the expansion of heat kernels in the interior of manifolds. The
claim is then proven. We refer to \parencite[Proposition 3.2]{mazzeoAnalyticTorsionManifolds2013}
for more details. 
\end{proof}
So we have solved away the leading coefficient of $t\left(\partial_{t}-\Delta_{x}\right)H$
at ff, as well as all the coefficients at td.

Next, we solve away all the coefficients at rf.
\begin{claim}
There is an element $H^{\left(2\right)}\in\Psi_{\mathrm{e-h}}^{-2,-2,E_{\mathrm{lf}},E_{\mathrm{rf}}}\left(M;\Lambda^{k}M\right)$
such that 
\[
\left\{ \begin{array}{rl}
P^{\left(2\right)} & :=t\left(\partial_{t}-\Delta_{x}\right)H^{\left(2\right)}\in\Psi_{\mathrm{e-h}}^{-3,-\infty,\mathbb{N}_{0},\emptyset}\\
\lim_{t\downarrow0}H^{\left(2\right)}\left(t,x,y\right) & =\delta_{y}\left(x\right)
\end{array}\right.
\]
\end{claim}

\begin{proof}
Let $r(x)$ be a boundary-defining function for rf such that $r(x)=\mathrm{dist}\left(x,\partial M\right)=x_{n}$
near rf. We observe that $r^{2}\left(\partial_{t}-\Delta_{x}\right)$
is a $b$-operator which can be restricted to rf (defined by $\theta_{n}=0$
in the rf-coordinate system from (\ref{eq:rf_coord})). On the other
hand, in the ff-coordinate system from (\ref{eq:ff_coor}), $r=\tau\xi_{n}$,
so $r^{2}\left(\partial_{t}-\Delta_{x}\right)$ is also a $b$-operator
w.r.t. ff. 

We observe that $\left(\partial_{t}-\Delta_{x}\right)H^{\left(1\right)}\in\Psi_{\mathrm{e-h}}^{-1,-\infty,\mathbb{N}_{0},\mathbb{N}_{0}}$
and we want to obtain $\left(\partial_{t}-\Delta_{x}\right)H^{\left(2\right)}\in\Psi_{\mathrm{e-h}}^{-1,-\infty,\mathbb{N}_{0},\emptyset}$.
Therefore it is enough to find $J\in\Psi_{\mathrm{e-h}}^{-3,-3,E_{\mathrm{lf}},E_{\mathrm{rf}}}\left(M;\Lambda^{k}M\right)$
such that $r^{2}\left(\partial_{t}-\Delta_{x}\right)\left(H^{\left(1\right)}-J\right)$
vanishes to infinite order at rf.

Let $B=r^{2}\left(\partial_{t}-\Delta_{x}\right)H^{\left(1\right)}\in\Psi_{\mathrm{e-h}}^{-3,-\infty,\mathbb{N}_{0},\mathbb{N}_{0}+2}$.
We note that $B_{0}^{\mathrm{rf}}=B_{1}^{\mathrm{rf}}=0$, so it is
fine to set $J_{0}^{\mathrm{rf}}=J_{1}^{\mathrm{rf}}=0$. 

Recall that $\Delta_{x}=\sum_{ij}g^{ij}\left(x\right)\partial_{x_{i}}\partial_{x_{j}}+\sum_{i}b_{i}\partial_{x_{i}}+c$
where $b_{i},c$ are smooth. Then by translating $r^{2}\left(\partial_{t}-\Delta_{x}\right)$
into rf-coordinates, we have to solve the formal expansion at rf:
\begin{equation}
\theta_{n}^{2}\left(\partial_{T}-\sum_{i,j\neq n}g^{ij}\partial_{\theta_{i}}\partial_{\theta_{j}}-\sum_{i\neq n}y_{n}b_{i}\partial_{\theta_{i}}-\partial_{\theta_{n}}^{2}-y_{n}b_{n}\partial_{\theta_{n}}-cy_{n}^{2}\right)\left(\sum_{j\geq2}J_{j}^{\mathrm{rf}}\theta_{n}^{j}\right)=\sum_{j\geq2}B_{j}^{\mathrm{rf}}\theta_{n}^{j}\label{eq:explicit}
\end{equation}
Note that near rf, because we have chosen the geodesic normal coordinates,
$g^{in}=\delta^{in}$ for any $i\in\{1,...,n\}$. Then, (\ref{eq:explicit})
boils down to solving 
\[
N_{\mathrm{rf}}^{j}\left(r^{2}\left(\partial_{t}-\Delta_{x}\right)\right)J_{j}^{\mathrm{rf}}\left(T,\theta',y',y_{n}\right)=C_{j}\left(T,\theta',y',y_{n}\right),\;j\geq2
\]
where
\begin{itemize}
\item $N_{\mathrm{rf}}^{j}\left(r^{2}\left(\partial_{t}-\Delta_{x}\right)\right)=-j\left(j-1\right),\;j\geq1$. 
\item $C_{j}$ is an inhomogeneous term depending on $B_{j}^{\mathrm{rf}}$
and $J_{2}^{\mathrm{rf}},...,J_{j-1}^{\mathrm{rf}}$. In particular,
$C_{2}=B_{2}^{\mathrm{rf}}$. 
\end{itemize}
Solving this is trivial (and in fact the solutions are unique), since
for $j\geq2$, $N_{\mathrm{rf}}^{j}\left(r^{2}\left(\partial_{t}-\Delta_{x}\right)\right)$
is a nonzero constant. We note that $\left(J_{j}^{\mathrm{rf}}\right)_{j\geq2}$
inherits many properties from $\left(B_{j}^{\mathrm{rf}}\right)_{j\geq2}$
by induction:
\begin{itemize}
\item In the rf-coordinate system, $B_{j}^{\mathrm{rf}}$ is defined from
$\frac{1}{j!}\left.\partial_{\theta_{n}}^{j}\right|_{\theta_{n}=0}B$
(abuse of notation). But $y_{n}$ is the defining function for ff,
so the index set of $B_{j}^{\mathrm{rf}}$ at ff is the same as that
of $B$, and therefore this is also true for $J_{2}^{\mathrm{rf}}$.
This extends to $J_{j}^{\mathrm{rf}}$ $\forall j\geq2$, because
we can explicitly derive $C_{j}$ from (\ref{eq:explicit}), and see
that the powers of $y_{n}$ never get lowered (no $\partial_{y_{n}}$
or $\frac{1}{y_{n}}$).
\item The index sets of $B$ at td and tf are empty (i.e. $B=O\left(T^{\infty}\right)$
as $T\to0$), which implies $J_{j}^{\mathrm{rf}}=O\left(T^{\infty}\right)$.
\end{itemize}
Note that we also have to solve for $J_{j}^{\mathrm{rf}}$ where $y$
is away from the boundary (which means there is no rf-coordinate system).
In that case, we use the ts-coordinate system and solve the formal
expansion at rf. This proceeds in the same fashion (but it is even
simpler, since we are far away from ff).

Consequently, constructing $J$ from $\left(J_{j}^{\mathrm{rf}}\right)_{j\geq0}$
will give us $J\in\Psi_{\mathrm{e-h}}^{-3,-\infty,\mathbb{N}_{0},\mathbb{N}_{0}+2}\left(M;\Lambda^{k}M\right)$
such that $B-r^{2}\left(\partial_{t}-\Delta_{x}\right)J$ vanishes
to infinite order at rf.

With the index set at rf being $\mathbb{N}_{0}+2$, $J$ satisfies
the absolute Neumann boundary condition. Also, because the index sets
of $J$ at ff and td are higher than those of $H^{\left(1\right)}$,
we conclude 
\[
\left\{ \begin{array}{rl}
N_{\mathrm{ff}}^{-n}\left(H^{\left(1\right)}-J\right) & =N_{\mathrm{ff}}^{-n}\left(H^{\left(1\right)}\right)\\
N_{\mathrm{td}}^{-n}\left(H^{\left(1\right)}-J\right) & =N_{\mathrm{td}}^{-n}\left(H^{\left(1\right)}\right)
\end{array}\right.
\]
By setting $H^{\left(2\right)}=H^{\left(1\right)}-J$, the claim is
proven.
\end{proof}
For the last step, we consider the formal Volterra series:
\[
H=H^{\left(2\right)}+H^{\left(2\right)}*R^{\left(2\right)}+H^{\left(2\right)}*R^{\left(2\right)}*R^{\left(2\right)}+....
\]
where $R^{\left(2\right)}:=-\left(\partial_{t}-\Delta_{x}\right)H^{\left(2\right)}\in\Psi_{\mathrm{e-h}}^{-1,-\infty,\mathbb{N}_{0},\emptyset}$,
and the composition $A*B$ is defined by
\[
A*B\left(t,x,y\right)=\int_{0}^{t}\mathrm{d}s\int_{M}\mathrm{d}\mathrm{vol}_{g}\left(z\right)\;A(t-s,x,z)B(s,z,y)
\]
 By \parencite[Theorem 5.3]{mazzeoAnalyticTorsionManifolds2013},
if $Q_{\mathrm{lf}}+Q'_{\mathrm{rf}}>-1$; $\alpha,\gamma,\beta\in-\mathbb{N}_{1}$,
we have the formula
\[
\Psi_{\mathrm{e-h}}^{\alpha,\gamma,Q_{\mathrm{lf}},Q_{\mathrm{rf}}}*\Psi_{\mathrm{e-h}}^{\beta,-\infty,Q'_{\mathrm{lf}},Q'_{\mathrm{rf}}}\subset\Psi_{\mathrm{e-h}}^{\alpha+\beta,-\infty,P_{\mathrm{lf}},P_{\mathrm{rf}}}
\]
 where $P_{\mathrm{lf}}=Q'_{\mathrm{lf}}\overline{\cup}\left(Q_{\mathrm{lf}}-\beta\right)$;
$P_{\mathrm{rf}}=Q_{\mathrm{rf}}\overline{\cup}\left(Q'_{\mathrm{rf}}-\alpha\right)$.
This means that for $N\in\mathbb{N}_{1}:$
\[
H^{\left(2\right)}*\left(R^{\left(2\right)}\right)^{*N}\in\Psi_{\mathrm{e-h}}^{-2-N,-\infty,E_{\mathrm{lf},N},E_{\mathrm{rf}}}
\]
where $E_{\mathrm{lf,N}}$ is defined inductively by $E_{\mathrm{lf},1}=\mathbb{N}_{0}\overline{\cup}\left(\mathbb{N}_{0}+1\right)$
and $E_{\mathrm{lf},N+1}=\mathbb{N}_{0}\overline{\cup}\left(E_{\mathrm{lf},N}+1\right)$
for $N\geq1$. 

Letting $\mathbb{N}_{j}=\left\{ x\in\mathbb{N}:x\geq j\right\} $
and $\mathscr{N}=\overline{\bigcup}_{j\in\mathbb{N}_{0}}\mathbb{N}_{j}$,
we conclude that 
\[
\forall N:E_{\mathrm{lf},N}\subset\mathscr{N}=\{\left(x,y\right)\in\mathbb{N}_{0}^{2}:y\leq x\}
\]
 which is a well-defined index set. 

A common property of Volterra series is that they converge. We can
observe this from the fact that $\forall m\in\mathbb{N}_{2},$ $L^{*m}\left(t,x,y\right)$
is equal to
\[
\int_{M^{m-1}}\mathrm{d}\mathrm{vol}_{g}^{m-1}\left(z_{1},...,z_{m-1}\right)\int_{\Delta_{m-1}^{t}}\mathrm{d}\left(s_{1},...,s_{m-1}\right)\;L\left(t-s_{1}-...-s_{m-1},x,z_{m-1}\right)...L\left(s_{1},z_{1},y\right)
\]
where $\Delta_{m-1}^{t}$ is the simplex defined by $\left\{ 0\leq s_{1}\leq s_{1}+s_{2}\leq...\leq s_{1}+...+s_{m-1}\leq t\right\} $.
As the volume of $\Delta_{m-1}^{t}$ is $\frac{t^{m-1}}{\left(m-1\right)!}$,
the factorial factor $\frac{1}{\left(m-1\right)!}$ ultimately forces
strong convergence as $m\to\infty$. See \parencite[Section 2.4]{berlineHeatKernelsDirac2004},
\parencite[Section 3.2]{mazzeoAnalyticTorsionManifolds2013}, and
\parencite{melroseAtiyahPatodiSingerIndexTheorem2018} for more details
and estimates.

Consequently, we obtain $H\in\Psi_{\mathrm{e-h}}^{-2,-2,\mathscr{N},E_{\mathrm{rf}}}$.
Because of the identity $\left(\partial_{t}-\Delta_{x}\right)\left(H^{\left(2\right)}*\left(R^{\left(2\right)}\right)^{*N}\right)=\left(R^{\left(2\right)}\right)^{*N}-\left(R^{\left(2\right)}\right)^{*\left(N+1\right)},$
we conclude
\[
\left(\partial_{t}-\Delta_{x}\right)H=0
\]
Let us check that $H$ is the true Hodge-Neumann heat kernel.
\begin{itemize}
\item The index set $E_{\mathrm{rf}}$ satisfies the Neumann boundary condition.
\item For any $u\in L^{2}\left(M;\Lambda^{k}M\right)$: 
\[
H(t)u(x):=\int_{M}H\left(t,x,y\right)u(y)\;\mathrm{d}\mathrm{vol_{g}}y\in C^{\infty}\left(\left(0,\infty\right),\Omega_{\mathrm{hom}N}^{k}\right)
\]
and satisfies $\left(\partial_{t}-\Delta_{x}\right)\left(H(t)u(x)\right)=0$
on $\{t>0\}$. In particular, $H\left(t\right)\in\mathrm{End}\left(L^{2}\right)$
for all $t>0$ and 
\begin{equation}
\partial_{t}\left(\left\Vert H(t)u\right\Vert _{L^{2}}^{2}\right)\leq0\label{eq:decrea}
\end{equation}
 because the Neumann Laplacian $\widetilde{\Delta_{N}}$ is self-adjoint
and dissipative.
\item We have $N_{\mathrm{td}}^{-n}\left(H\right)=N_{\mathrm{td}}^{-n}\left(H^{\left(1\right)}\right)$,
therefore $\lim_{t\downarrow0}H\left(t,x,y\right)=\delta_{y}\left(x\right)$.
For any $u\in\Omega_{00}^{k}\left(M\right):H(t)u\xrightarrow[t\downarrow0]{L^{2}}u$,
which, along with (\ref{eq:decrea}), implies $\left\Vert H(t)u\right\Vert _{L^{2}}\leq\left\Vert u\right\Vert _{L^{2}}$
. By density, we conclude the same for $u\in L^{2}\left(M;\Lambda^{k}M\right)$.
Recall that $e^{t\widetilde{\Delta_{N}}}$ is the heat semigroup defined
by functional analysis. For any $u\in L^{2}\left(M;\Lambda^{k}M\right)$,
$U(t):=H(t)u-e^{t\widetilde{\Delta_{N}}}u$ is a $C_{t}^{0}L_{x}^{2}$
solution of
\[
\left\{ \begin{array}{rl}
\left(\partial_{t}-\Delta_{x}\right)U(t,x)= & 0\;\forall t>0\\
U(t)\xrightarrow[t\downarrow0]{L^{2}} & 0
\end{array}\right.
\]
By an energy argument just like (\ref{eq:decrea}), we must have $U\left(t\right)=0$
for all $t$. Then, $H(t)=e^{t\widetilde{\Delta}_{N}}.$
\end{itemize}
So $H$ is the true heat kernel, which must be smooth on $\left(0,\infty\right)\times M\times M$
by standard parabolic theory. Another way to see this is that the
heat kernel must be symmetric, therefore smoothness in $x$ implies
smoothness in $y$. Either way, because we have smoothness, there
are no log terms on lf, and we conclude $H\in\Psi_{\mathrm{e-h}}^{-2,-2,\mathbb{N}_{0},E_{\mathrm{rf}}}$.

\subsection{Relevant properties}

We extract some key properties from \Thmref{heat_calc} that we need
for this paper, and write them in a language more familiar with analysts.
\begin{enumerate}
\item (\textbf{off-diagonal decay}) For any multi-index $\gamma$ and $x\neq y$,
\begin{equation}
D_{t,x,y}^{\gamma}H\left(t,x,y\right)=O\left(t^{\infty}\right)\label{eq:1st_decay}
\end{equation}
 as $t\downarrow0$, locally uniform in $\left(x,y\right)\notin\Delta\left(M\right)$.
\item (\textbf{interior blow-up}) For $x\in\mathrm{int}\left(M\right)$,
locally in projective coordinates $\left(\tau,x,\zeta\right)=\left(\sqrt{t},x,\frac{x-y}{\sqrt{t}}\right)$,
with $\widetilde{H}$ being the pullback of $t^{\frac{n}{2}}H$ in
these coordinates, we have 
\begin{enumerate}
\item $\widetilde{H}$ smooth in $\tau,x,\zeta$, up to $\{\tau=0\}$.
\item (\textbf{rapid decay}) For any multi-index $\gamma$ and bounded $\tau$:
\begin{equation}
D_{\tau,x,\zeta}^{\gamma}\widetilde{H}\left(\tau,x,\zeta\right)=O\left(\left\langle \zeta\right\rangle ^{-\infty}\right)\label{eq:2nd_Decay}
\end{equation}
\end{enumerate}
\end{enumerate}
\begin{rem}
Both (\ref{eq:1st_decay}) and (\ref{eq:2nd_Decay}) come from the
empty index set at tf. We also refer to \parencite[Section 2.3.3]{kottkeLinearAnalysisManifolds2016a}
for an explanation of (\ref{eq:2nd_Decay}).

There are more specific properties from \Thmref{heat_calc}, which
we do not currently need.
\end{rem}

\printbibliography

@online{albinLinearAnalysisManifolds2017,
  title = {Linear {{Analysis}} on {{Manifolds}}},
  shorttitle = {Linear {{Analysis}} on {{Manifolds}}},
  author = {Albin, Pierre},
  date = {2017-01-18},
  journaltitle = {American Mathematical Society},
  url = {https://www.ams.org/open-math-notes/omn-view-listing?listingId=110667},
  urldate = {2020-10-19},
  langid = {english}
}

@article{bardosExtensionOnsagerConjecture2019,
  title = {On the {{Extension}} of {{Onsager}}’s {{Conjecture}} for {{General Conservation Laws}}},
  author = {Bardos, Claude and Gwiazda, Piotr and Świerczewska-Gwiazda, Agnieszka and Titi, Edriss S. and Wiedemann, Emil},
  date = {2019-04-01},
  journaltitle = {Journal of Nonlinear Science},
  shortjournal = {J Nonlinear Sci},
  volume = {29},
  pages = {501--510},
  issn = {1432-1467},
  doi = {10.1007/s00332-018-9496-4},
  url = {https://doi.org/10.1007/s00332-018-9496-4},
  urldate = {2020-09-13},
  langid = {english},
  number = {2}
}

@online{bardosOnsagerConjectureBounded2019,
  title = {Onsager's Conjecture in Bounded Domains for the Conservation of Entropy and Other Companion Laws},
  author = {Bardos, Claude and Gwiazda, Piotr and Świerczewska-Gwiazda, Agnieszka and Titi, Edriss S. and Wiedemann, Emil},
  date = {2019-02-19},
  url = {http://arxiv.org/abs/1902.07120},
  urldate = {2020-09-13},
  archivePrefix = {arXiv},
  eprint = {1902.07120},
  eprinttype = {arxiv},
  primaryClass = {math}
}

@article{bardosOnsagerConjectureIncompressible2018,
  title = {Onsager’s {{Conjecture}} for the {{Incompressible Euler Equations}} in {{Bounded Domains}}},
  author = {Bardos, Claude and Titi, Edriss S.},
  date = {2018-04-01},
  journaltitle = {Archive for Rational Mechanics and Analysis},
  shortjournal = {Arch Rational Mech Anal},
  volume = {228},
  pages = {197--207},
  issn = {1432-0673},
  doi = {10.1007/s00205-017-1189-x},
  url = {https://doi.org/10.1007/s00205-017-1189-x},
  urldate = {2020-09-13},
  langid = {english},
  number = {1}
}

@article{bardosOnsagerConjecturePhysical2019,
  title = {Onsager’s {{Conjecture}} with {{Physical Boundaries}} and an {{Application}} to the {{Vanishing Viscosity Limit}}},
  author = {Bardos, Claude and Titi, Edriss S. and Wiedemann, Emil},
  date = {2019-08-01},
  journaltitle = {Communications in Mathematical Physics},
  shortjournal = {Commun. Math. Phys.},
  volume = {370},
  pages = {291--310},
  issn = {1432-0916},
  doi = {10.1007/s00220-019-03493-6},
  url = {https://doi.org/10.1007/s00220-019-03493-6},
  urldate = {2020-09-13},
  langid = {english},
  number = {1}
}

@book{berlineHeatKernelsDirac2004,
  title = {Heat Kernels and {{Dirac}} Operators},
  author = {Berline, Nicole and Getzler, Ezra and Vergne, Michèle},
  date = {2004},
  publisher = {{Springer}},
  location = {{Berlin ; New York}},
  isbn = {978-3-540-20062-8},
  pagetotal = {363},
  series = {Grundlehren Text Editions}
}

@article{cheskidovEnergyConservationOnsager2008,
  title = {Energy Conservation and {{Onsager}}'s Conjecture for the {{Euler}} Equations},
  author = {Cheskidov, A and Constantin, P and Friedlander, S and Shvydkoy, R},
  date = {2008-06-01},
  journaltitle = {Nonlinearity},
  shortjournal = {Nonlinearity},
  volume = {21},
  pages = {1233--1252},
  issn = {0951-7715, 1361-6544},
  doi = {10.1088/0951-7715/21/6/005},
  url = {https://iopscience.iop.org/article/10.1088/0951-7715/21/6/005},
  urldate = {2020-09-19},
  number = {6}
}

@article{constantinOnsagerConjectureEnergy1994,
  title = {Onsager's Conjecture on the Energy Conservation for Solutions of {{Euler}}'s Equation},
  author = {Constantin, Peter and E, Weinan and Titi, Edriss S.},
  date = {1994},
  journaltitle = {Communications in Mathematical Physics},
  shortjournal = {Comm. Math. Phys.},
  volume = {165},
  pages = {207--209},
  publisher = {{Springer-Verlag}},
  issn = {0010-3616, 1432-0916},
  url = {https://projecteuclid.org/euclid.cmp/1104271041},
  urldate = {2020-09-13},
  langid = {english},
  mrnumber = {MR1298949},
  number = {1},
  zmnumber = {0818.35085}
}

@article{delellisDissipativeContinuousEuler2013,
  title = {Dissipative Continuous {{Euler}} Flows},
  author = {De Lellis, Camillo and Székelyhidi, László},
  date = {2013-08-01},
  journaltitle = {Inventiones mathematicae},
  shortjournal = {Invent. math.},
  volume = {193},
  pages = {377--407},
  issn = {1432-1297},
  doi = {10.1007/s00222-012-0429-9},
  url = {https://doi.org/10.1007/s00222-012-0429-9},
  urldate = {2020-09-13},
  langid = {english},
  number = {2}
}

@article{delellisDissipativeEulerFlows2014,
  title = {Dissipative {{Euler}} Flows and {{Onsager}}'s Conjecture},
  author = {De Lellis, Camillo and Székelyhidi Jr., László},
  date = {2014-08-23},
  journaltitle = {Journal of the European Mathematical Society},
  shortjournal = {J. Eur. Math. Soc.},
  volume = {16},
  pages = {1467--1505},
  issn = {1435-9855},
  doi = {10.4171/JEMS/466},
  url = {https://www.ems-ph.org/journals/show_abstract.php?issn=1435-9855&vol=16&iss=7&rank=5},
  urldate = {2020-09-13},
  number = {7}
}

@article{delellisTurbulenceGeometryNash2019,
  title = {On {{Turbulence}} and {{Geometry}}: From {{Nash}} to {{Onsager}}},
  shorttitle = {On {{Turbulence}} and {{Geometry}}},
  author = {De Lellis, Camillo and Székelyhidi, László},
  date = {2019-05-01},
  journaltitle = {Notices of the American Mathematical Society},
  shortjournal = {Notices Amer. Math. Soc.},
  volume = {66},
  pages = {1},
  issn = {0002-9920, 1088-9477},
  doi = {10.1090/noti1868},
  url = {http://www.ams.org/notices/201905/rnoti-p677.pdf},
  urldate = {2020-09-13},
  langid = {english},
  number = {05}
}

@article{drivasOnsagerConjectureAnomalous2018,
  title = {Onsager's {{Conjecture}} and {{Anomalous Dissipation}} on {{Domains}} with {{Boundary}}},
  author = {Drivas, Theodore D. and Nguyen, Huy Q.},
  date = {2018-01-01},
  journaltitle = {SIAM Journal on Mathematical Analysis},
  shortjournal = {SIAM J. Math. Anal.},
  volume = {50},
  pages = {4785--4811},
  publisher = {{Society for Industrial and Applied Mathematics}},
  issn = {0036-1410},
  doi = {10.1137/18M1178864},
  url = {https://epubs.siam.org/doi/10.1137/18M1178864},
  urldate = {2020-09-13},
  number = {5}
}

@article{eyinkEnergyDissipationViscosity1994,
  title = {Energy Dissipation without Viscosity in Ideal Hydrodynamics {{I}}. {{Fourier}} Analysis and Local Energy Transfer},
  author = {Eyink, Gregory L.},
  date = {1994-11},
  journaltitle = {Physica D: Nonlinear Phenomena},
  shortjournal = {Physica D: Nonlinear Phenomena},
  volume = {78},
  pages = {222--240},
  issn = {01672789},
  doi = {10.1016/0167-2789(94)90117-1},
  url = {https://linkinghub.elsevier.com/retrieve/pii/0167278994901171},
  urldate = {2020-09-13},
  langid = {english},
  number = {3-4}
}

@incollection{grieserBasicsBCalculus2001,
  title = {Basics of the B-{{Calculus}}},
  booktitle = {Approaches to {{Singular Analysis}}},
  author = {Grieser, Daniel},
  editor = {Gil, Juan B. and Grieser, Daniel and Lesch, Matthias},
  date = {2001},
  pages = {30--84},
  publisher = {{Birkhäuser Basel}},
  location = {{Basel}},
  doi = {10.1007/978-3-0348-8253-8_2},
  url = {http://link.springer.com/10.1007/978-3-0348-8253-8_2},
  urldate = {2020-01-07},
  isbn = {978-3-0348-9492-0 978-3-0348-8253-8},
  langid = {english}
}

@online{grieserNOTESHEATKERNEL2004,
  title = {{{NOTES ON HEAT KERNEL ASYMPTOTICS}}},
  author = {Grieser, D.},
  date = {2004},
  url = {https://www.semanticscholar.org/paper/NOTES-ON-HEAT-KERNEL-ASYMPTOTICS-Grieser/81e2265316899f63118a185135f1e20e5bd0d32d},
  urldate = {2020-09-14},
  langid = {english}
}

@online{huynhHodgetheoreticAnalysisManifolds2019,
  title = {Hodge-Theoretic Analysis on Manifolds with Boundary, Heatable Currents, and {{Onsager}}'s Conjecture in Fluid Dynamics},
  author = {Huynh, Khang Manh},
  date = {2019-07-14},
  url = {http://arxiv.org/abs/1907.05360},
  urldate = {2020-02-03},
  archivePrefix = {arXiv},
  eprint = {1907.05360},
  eprinttype = {arxiv},
  primaryClass = {math}
}

@online{isettHeatFlowApproach2014,
  title = {A Heat Flow Approach to {{Onsager}}'s Conjecture for the {{Euler}} Equations on Manifolds},
  author = {Isett, Philip and Oh, Sung-Jin},
  date = {2014-01-17},
  url = {http://arxiv.org/abs/1310.7947},
  urldate = {2019-11-13},
  archivePrefix = {arXiv},
  eprint = {1310.7947},
  eprinttype = {arxiv},
  primaryClass = {math}
}

@article{isettProofOnsagerConjecture2018,
  title = {A Proof of {{Onsager}}'s Conjecture},
  author = {Isett, Philip},
  date = {2018},
  journaltitle = {Annals of Mathematics},
  volume = {188},
  pages = {871--963},
  publisher = {{[Annals of Mathematics, Trustees of Princeton University on Behalf of the Annals of Mathematics, Mathematics Department, Princeton University]}},
  issn = {0003-486X},
  eprint = {10.4007/annals.2018.188.3.4},
  eprinttype = {jstor},
  number = {3}
}

@misc{kottkeLinearAnalysisManifolds2016a,
  title = {Linear {{Analysis}} on {{Manifolds}}: {{Notes}} for {{Math}} 7376, {{Spring}} 2016},
  author = {Kottke, Chris},
  date = {2016-06-14},
  url = {https://ckottke.ncf.edu/neu/7376_sp16/lacm.pdf},
  urldate = {2020-10-22}
}

@book{lemarie-rieussetRecentDevelopmentsNavierStokes2002,
  title = {Recent Developments in the {{Navier}}-{{Stokes}} Problem},
  author = {Lemarie-Rieusset, Pierre Gilles},
  date = {2002-04-26},
  publisher = {{CRC Press}},
  eprint = {Rjn1NqO959sC},
  eprinttype = {googlebooks},
  isbn = {978-1-4200-3567-4},
  langid = {english},
  pagetotal = {412}
}

@online{mazzeoAnalyticTorsionManifolds2013,
  title = {Analytic {{Torsion}} on {{Manifolds}} with {{Edges}}},
  author = {Mazzeo, Rafe and Vertman, Boris},
  date = {2013-07-13},
  url = {http://arxiv.org/abs/1103.0448},
  urldate = {2020-10-02},
  archivePrefix = {arXiv},
  eprint = {1103.0448},
  eprinttype = {arxiv},
  primaryClass = {math}
}

@book{melroseAtiyahPatodiSingerIndexTheorem2018,
  title = {The {{Atiyah}}-{{Patodi}}-{{Singer}} Index Theorem},
  author = {Melrose, Richard B},
  date = {2018},
  url = {http://public.ebookcentral.proquest.com/choice/publicfullrecord.aspx?p=5437055},
  urldate = {2020-09-14},
  isbn = {978-1-4398-6460-9},
  langid = {english}
}

@article{melroseCalculusConormalDistributions1992,
  title = {Calculus of Conormal Distributions on Manifolds with Corners},
  author = {Melrose, Richard B.},
  date = {1992-01-01},
  journaltitle = {International Mathematics Research Notices},
  shortjournal = {Int Math Res Notices},
  volume = {1992},
  pages = {51--61},
  publisher = {{Oxford Academic}},
  issn = {1073-7928},
  doi = {10.1155/S1073792892000060},
  url = {https://academic.oup.com/imrn/article/1992/3/51/824177},
  urldate = {2020-10-05},
  langid = {english},
  number = {3}
}

@article{nguyenEnergyConservationInhomogeneous2020,
  title = {Energy Conservation for Inhomogeneous Incompressible and Compressible {{Euler}} Equations},
  author = {Nguyen, Quoc-Hung and Nguyen, Phuoc-Tai and Tang, Bao Quoc},
  date = {2020-10-15},
  journaltitle = {Journal of Differential Equations},
  shortjournal = {Journal of Differential Equations},
  volume = {269},
  pages = {7171--7210},
  issn = {0022-0396},
  doi = {10.1016/j.jde.2020.05.025},
  url = {http://www.sciencedirect.com/science/article/pii/S0022039620302904},
  urldate = {2020-09-14},
  langid = {english},
  number = {9}
}

@article{nguyenOnsagerConjectureEnergy2019,
  title = {Onsager’s {{Conjecture}} on the {{Energy Conservation}} for {{Solutions}} of {{Euler Equations}} in {{Bounded Domains}}},
  author = {Nguyen, Quoc-Hung and Nguyen, Phuoc-Tai},
  date = {2019-02-01},
  journaltitle = {Journal of Nonlinear Science},
  shortjournal = {J Nonlinear Sci},
  volume = {29},
  pages = {207--213},
  issn = {1432-1467},
  doi = {10.1007/s00332-018-9483-9},
  url = {https://doi.org/10.1007/s00332-018-9483-9},
  urldate = {2020-09-13},
  langid = {english},
  number = {1}
}

@article{taoLocalisationCompactnessProperties2013,
  title = {Localisation and Compactness Properties of the {{Navier}}–{{Stokes}} Global Regularity Problem},
  author = {Tao, Terence},
  date = {2013},
  journaltitle = {Analysis \& PDE},
  shortjournal = {Anal. PDE},
  volume = {6},
  pages = {25--107},
  issn = {2157-5045, 1948-206X},
  doi = {10.2140/apde.2013.6.25},
  url = {https://projecteuclid.org/euclid.apde/1513731290},
  urldate = {2020-01-30},
  langid = {english},
  mrnumber = {MR3068540},
  number = {1},
  zmnumber = {1287.35058}
}

@book{taoNonlinearDispersiveEquations2006,
  title = {Nonlinear Dispersive Equations: Local and Global Analysis},
  shorttitle = {Nonlinear Dispersive Equations},
  author = {Tao, Terence},
  date = {2006},
  isbn = {978-0-8218-4143-3 978-1-4704-2466-4},
  langid = {english}
}

@online{taoQuantitativeFormulationGlobal2009,
  title = {A Quantitative Formulation of the Global Regularity Problem for the Periodic {{Navier}}-{{Stokes}} Equation},
  author = {Tao, Terence},
  date = {2009-05-21},
  url = {http://arxiv.org/abs/0710.1604},
  urldate = {2020-02-11},
  archivePrefix = {arXiv},
  eprint = {0710.1604},
  eprinttype = {arxiv},
  primaryClass = {math}
}

@book{taylorPartialDifferentialEquations2011,
  title = {Partial {{Differential Equations III}}: {{Nonlinear Equations}}},
  shorttitle = {Partial {{Differential Equations III}}},
  author = {Taylor, Michael E.},
  date = {2011},
  volume = {117},
  publisher = {{Springer New York}},
  location = {{New York, NY}},
  doi = {10.1007/978-1-4419-7049-7},
  url = {http://link.springer.com/10.1007/978-1-4419-7049-7},
  urldate = {2020-08-06},
  isbn = {978-1-4419-7048-0 978-1-4419-7049-7},
  langid = {english},
  series = {Applied {{Mathematical Sciences}}}
}

@book{taylorPartialDifferentialEquations2011a,
  title = {Partial {{Differential Equations I}}: {{Basic Theory}}},
  shorttitle = {Partial {{Differential Equations I}}},
  author = {Taylor, Michael E.},
  date = {2011},
  volume = {115},
  publisher = {{Springer New York}},
  location = {{New York, NY}},
  doi = {10.1007/978-1-4419-7055-8},
  url = {http://link.springer.com/10.1007/978-1-4419-7055-8},
  urldate = {2020-08-06},
  isbn = {978-1-4419-7054-1 978-1-4419-7055-8},
  langid = {english},
  series = {Applied {{Mathematical Sciences}}}
}

@online{wiedemannConservedQuantitiesRegularity2020,
  title = {Conserved Quantities and Regularity in Fluid Dynamics},
  author = {Wiedemann, Emil},
  date = {2020-03-17},
  url = {http://arxiv.org/abs/2003.07807},
  urldate = {2020-04-02},
  archivePrefix = {arXiv},
  eprint = {2003.07807},
  eprinttype = {arxiv},
  primaryClass = {physics}
}
\end{document}